\documentclass{article}
\usepackage{graphicx}
\usepackage{bm}
\usepackage[top=1in,bottom=1in,left=1in,right=1in]{geometry}
\usepackage{textgreek,mathtools}
\usepackage{amsmath}
\usepackage{graphicx,subfig}
\usepackage[colorinlistoftodos]{todonotes}
\usepackage[colorlinks=true, allcolors=blue]{hyperref}
\usepackage{amsmath}
\usepackage{upgreek}
\usepackage{amssymb}
\usepackage{amsfonts}
\usepackage{amsthm}
\usepackage{mathrsfs}
\usepackage{fontenc}
\usepackage{empheq}
\usepackage{enumerate}
\usepackage{comment}
\usepackage{algorithm}
\usepackage{algorithmic}
\usepackage{appendix}
\usepackage{bbm}
\usepackage{cancel}
\usepackage{tcolorbox}
\mathtoolsset{showonlyrefs}

\numberwithin{equation}{section}
\DeclareMathOperator*{\argmin}{arg\,min} 
\DeclareMathOperator*{\argmax}{arg\,max} 

\theoremstyle{plain}
\newtheorem{thm}{Theorem}[section]

\newtheorem{lem}[thm]{Lemma}
\newtheorem{prop}[thm]{Proposition}
\newtheorem{cor}[thm]{Corollary}
\newtheorem{assu}[thm]{Assumption}

\theoremstyle{definition}
\newtheorem{defn}[thm]{Definition}

\theoremstyle{remark}
\newtheorem{rem}[thm]{Remark}

\newcommand{\ud}{\,\mathrm{d}}
\newcommand{\mc}[1]{\mathcal{#1}}
\newcommand{\mb}[1]{\mathbb{#1}}

\newcommand{\abs}[1]{\left|#1\right|}
\newcommand{\transpose}{^{\operatorname{T}}}

\newcommand{\Var}{\mathrm{Var}}

\newcommand{\RR}{\mathbb{R}}
\newcommand{\EE}{\mathbb{E}}

\newcommand{\PP}{\mathbb{P}}

\newcommand{\tp}{^{\top}}
\newcommand{\rd}{\mathrm{d}}
\newcommand{\rD}{\mathrm{D}}
\newcommand{\mA}{\mathcal{A}}
\newcommand{\mC}{\mathcal{C}}
\newcommand{\mF}{\mathcal{F}}
\newcommand{\mG}{\mathcal{G}}

\newcommand{\mL}{\mathcal{L}}
\newcommand{\mM}{\mathcal{M}}
\newcommand{\mO}{\mathcal{O}}
\newcommand{\mP}{\mathcal{P}}
\newcommand{\mS}{\mathcal{S}}

\newcommand{\mV}{\mathcal{V}}
\newcommand{\mW}{\mathcal{W}}

\newcommand{\ve}{\varepsilon}
\newcommand{\vp}{\varphi}
\newcommand{\lam}{\lambda}
\newcommand{\io}{\iota}
\newcommand{\nx}{\nabla_x}
\newcommand{\ny}{\nabla_y}
\newcommand{\na}{\nabla_\alpha}
\newcommand{\nxa}{\nabla_{x,\alpha}}
\newcommand{\pt}{\partial_t}
\newcommand{\pta}{\partial_\tau}
\newcommand{\ps}{\partial_s}
\newcommand{\wt}{\widetilde}
\newcommand{\dt}{\Delta t}
\newcommand{\dta}{\Delta\tau}

\newcommand{\ma}{^{\mu,\alpha}}
\newcommand{\ms}{^{\mu,*}}
\newcommand{\mts}{^{\mu^\tau,*}}
\newcommand{\mao}{{\mu,\alpha}} 
\newcommand{\mat}{^{\mu^\tau,\alpha^\tau}}
\newcommand{\map}{^{\mu',\alpha'}}

\newcommand{\mato}{{\mu^\tau,\alpha^\tau}} 
\newcommand{\matk}{{\mu^{\tau_k},\alpha^{\tau_k}}}
\newcommand{\matp}{^{\mu^{\tau'},\alpha^{\tau'}}}
\newcommand{\bt}{\bar{t}}
\newcommand{\bxt}{\bar{x}_t}
\newcommand{\bxk}{\bar{x}_k}
\newcommand{\hPk}{\widehat{\Phi}_k}
\newcommand{\txk}{(t_k,x_k)}
\newcommand{\syk}{(s_k,y_k)}
\newcommand{\htxk}{(\widehat{t}_k,\widehat{x}_k)}
\newcommand{\hsyk}{(\widehat{s}_k,\widehat{y}_k)}
\newcommand{\txkp}{(t_k',x_k')}
\newcommand{\sykp}{(s_k',y_k')}
\newcommand{\htxkp}{(\widehat{t}_k',\widehat{x}_k')}
\newcommand{\hsykp}{(\widehat{s}_k',\widehat{y}_k')}
\newcommand{\txsy}{(t,x,s,y)}
\newcommand{\txsyk}{(t_k,x_k,s_k,y_k)}
\newcommand{\htxsyk}{(\widehat{t}_k,\widehat{x}_k,\widehat{s}_k,\widehat{y}_k)}

\newcommand{\Vki}{V_k^\iota}
\newcommand{\Vkis}{V_{k,\iota}^*}
\newcommand{\wh}{\widehat}
\newcommand{\hp}{\widehat{p}}
\newcommand{\hq}{\widehat{q}}
\newcommand{\htk}{\widehat{t}_k}
\newcommand{\hsk}{\widehat{s}_k}
\newcommand{\hxk}{\widehat{x}_k}
\newcommand{\hyk}{\widehat{y}_k}
\newcommand{\qp}{(q,p,\widehat{q},\widehat{p})}
\newcommand{\xd}{x^{\delta}}
\newcommand{\xdp}{x^{\delta\prime}}
\newcommand{\fdt}{f^{\delta}_t}
\newcommand{\fdtp}{f^{\delta\prime}_t}
\newcommand{\bdt}{b^{\delta}_t}
\newcommand{\bdtp}{b^{\delta\prime}_t}
\newcommand{\sdt}{\sigma^{\delta}_t}
\newcommand{\sdtp}{\sigma^{\delta\prime}_t}
\newcommand{\prox}{\textrm{Prox}}

\newcommand{\diag}{\textrm{diag}}
\newcommand{\parentheses}[1]{\left(#1\right)}
\newcommand{\sqbra}[1]{\left[#1\right]}
\newcommand{\curlybra}[1]{\left\{#1\right\}}

\newcommand{\norm}[1]{\left\|#1\right\|}
\newcommand{\inner}[2]{\left\langle#1,\,#2\right\rangle}
\newcommand{\fd}[3]{\mathrm{D}^{#1}_{#2} #3}
\newcommand{\pfd}[3]{\left(\mathrm{D}^{#1}_{#2} #3\right)}

\newcommand{\rom}[1]{\text{\uppercase\expandafter{\romannumeral #1\relax}}}
\newcommand{\inttx}[1]{\int_0^T \int_{\RR^d} #1 \,\rd x \,\rd t}
\DeclareMathOperator{\Tr}{Tr}
\newcommand{\BAR}[1]{\overline{#1}}
\newcommand{\NORM}[1]{\lVert#1\rVert}
\newcommand{\R}{\mathbb{R}}
\newcommand{\batch}{\mathrm{batch}}
\newcommand{\LMC}{\mathrm{LMC}}

\title{Learning Mean-Field Games through Mean-Field Actor-Critic Flow}
\author{
Mo Zhou\thanks{Department of Mathematics, University of California, Los Angeles, CA 90095, {\em mozhou366@math.ucla.edu}.}
\and 
Haosheng Zhou\thanks{Department of Statistics and Applied Probability, University of California, Santa Barbara, CA 93106-3110, {\em hzhou593@ucsb.edu}.}
\and
Ruimeng Hu\thanks{Department of Mathematics, and Department of Statistics and Applied Probability, University of California, Santa Barbara, CA 93106-3080, {\em rhu@ucsb.edu}.} 
}

\date{}
\begin{document}
\maketitle
\begin{abstract}
We propose the Mean-Field Actor-Critic (MFAC) flow, a continuous-time learning dynamics for solving mean-field games (MFGs), combining techniques from reinforcement learning and optimal transport. The MFAC framework jointly evolves the control (actor), value function (critic), and distribution components through coupled gradient-based updates governed by partial differential equations (PDEs). A central innovation is the Optimal Transport Geodesic Picard (OTGP) flow, which drives the distribution toward equilibrium along Wasserstein-2 geodesics. We conduct a rigorous convergence analysis using Lyapunov functionals and establish global exponential convergence of the MFAC flow under a suitable timescale. Our results highlight the algorithmic interplay among actor, critic, and distribution components. Numerical experiments illustrate the theoretical findings and demonstrate the effectiveness of the MFAC framework in computing MFG equilibria.
\end{abstract}


\tableofcontents

\section{Introduction}\label{sec:intro}
Mean-field games (MFGs), introduced independently by Lasry and Lions \cite{lasry2006jeux,lasry2006jeux2,lasry2007mean} and by Huang, Caines, and Malham\'e \cite{huang2006large,huang2007large}, provide a powerful framework for modeling strategic interactions among a large population of agents, where each agent responds to the aggregate distribution of the population rather than to individual players. Over the past decade, substantial progress has been made in the theoretical development of MFGs, including the well-posedness of equilibria under monotonicity conditions \cite{lasry2006jeux}, and the rigorous connection to McKean–Vlasov forward-backward stochastic differential equations (FBSDEs) \cite{carmona2013probabilistic} and master equations \cite{cardaliaguet2019master}. A broader exposition of the theory and its historical development can be found in \cite{cardaliaguet2010notes, bensoussan2013mean,gomes2016regularity,carmona2018probabilistic}.

From a computational perspective, solving MFGs remains challenging due to their intrinsic infinite-dimensional structure arising from the dependence on the evolving population distribution. Classical numerical approaches focus on solving the coupled Hamilton--Jacobi--Bellman (HJB) and Fokker--Planck (FP) equations directly \cite{achdou2020mean}. More recent advances leverage deep learning techniques to approximate the partial differential equation (PDE) systems \cite{ruthotto2020machine,assouli2024deep}, FBSDEs \cite{carmona2022convergence,germain2022numerical,han2024learning}, and even master equations \cite{cohen2024deep,gu2024global}. In parallel, reinforcement learning (RL)-based approaches have attracted growing attention for solving MFGs, motivated by their model-free nature, i.e., the ability to learn optimal strategies directly from observations without requiring explicit knowledge of the system dynamics \cite{guo2019learning,perrin2022generalization,angiuli2022unified,angiuli2023deep}. We refer interested readers to the recent survey \cite{lauriere2022learning}.

In this work, we propose the Mean-Field Actor-Critic (MFAC) flow, a learning-based framework for solving MFGs with general distribution dependence. We model training as a dynamical system rather than a discrete iterative scheme. Our method builds upon three foundational ideas: actor-critic methods from RL for optimizing agent-level control; optimal transport theory for evolving the population distribution; and fictitious play for driving convergence to the MFG equilibrium. 

The MFAC flow consists of three interdependent components: an \emph{actor} that updates the control policy through policy gradient informed by the critic; a \emph{critic} that evaluates the value function corresponding to the current policy; and a \emph{distribution updater} governed by a novel Optimal Transport Geodesic Picard (OTGP) flow. The OTGP flow transports the distribution along Wasserstein-2 geodesics toward the state distribution induced by the current control, serving as a continuous analogue of Picard iteration in the space of probability measures. Our contributions can be summarized as follows:
\begin{itemize}
    \item Continuous-time framework. We introduce the MFAC flow as a single timescale continuous-time learning dynamics coupling policy update, policy evaluation, and population evolution. To our knowledge, this is the \emph{first} work to embed an optimal transport-based flow into an actor-critic learning framework for MFGs.
    
    \item Theoretical guarantees. We establish global exponential convergence of the MFAC flow to the MFG equilibrium using a Lyapunov-based analysis. Our proof highlights how the interaction among the actor, critic, and distribution dynamics can be controlled using the variation of the cost and contraction arguments in the Wasserstein space.

    \item Numerical algorithm. We develop a machine learning algorithm grounded in the continuous MFAC flow. Neural networks are used to parameterize both the actor and critic. To efficiently represent high-dimensional distributions, we introduce a score network trained via score matching \cite{hyvarinen2005estimation}. The optimal transport step in the OTGP flow is computed exactly using the Hungarian algorithm (whose complexity is dimension-independent). 
    We then demonstrate the practical performance of the MFAC flow on benchmark examples, confirming its stability, scalability, and ability to recover known MFG solutions.
\end{itemize}

Our work builds upon and significantly extends recent developments in continuous learning schemes. The continuous actor-critic flow was first proposed in \cite{zhou2024solving} for standard stochastic control problems, with rigorous convergence guarantees. Extending this framework to MFGs incurs significant new challenges in both flow design and theoretical analysis. On the algorithmic side, classical Wasserstein gradient flows, widely used in generative modeling and sampling \cite{liutkus2019sliced}, cannot be directly applied due to the absence of an energy functional in general MFG settings. Our proposed OTGP flow offers a natural alternative, inspired by the construction of solutions to McKean–Vlasov dynamics, though its analysis requires the introduction of a weighted Wasserstein metric and is more technically involved. 
Theoretically, our setting generalizes the one in \cite{zhou2024solving}, which was restricted to problems on torus. In contrast, we consider MFGs on non-compact spaces (e.g., the whole Euclidean space) under weaker regularity assumptions.

Existing work on RL for MFGs has largely focused on discrete iterative schemes, e.g. Q-learning \cite{angiuli2022unified} for discrete state-action spaces and actor-critic \cite{angiuli2023deep} for continuous state-action spaces. These algorithms often require multi-scale learning rates to ensure convergence \cite{angiuli2023convergence,angiuli2024analysis}, which can be difficult to tune in practice. In contrast, our MFAC flow operates on a single timescale, improving both the simplicity of implementation and empirical efficiency.

A further computational advantage lies in our use of score functions to represent high-dimensional distributions, which avoids the need to compute the normalizing constant of the density, a major bottleneck in direct density parameterization. As a result, our approach can handle general distributional dependence in the reward and dynamics, rather than being limited to dependence on low-order moments. A closely related work is \cite{han2024learning} which also addresses general distribution-dependent MFGs using a deep learning-based method to solve the associated McKean–Vlasov FBSDEs. That approach needs auxiliary constructions to recover the equilibrium control, whereas our method provides direct access to the optimal control policy throughout training.

The rest of the paper is organized as follows. Section~\ref{sec:MFG_background} introduces the MFG problem setup and notations used throughout. In Section~\ref{sec:MFACflow}, we present the MFAC flow, detailing the dynamics of the actor, critic, and distribution components and their coupling into a unified learning framework. Section~\ref{sec:convergence} provides a theoretical analysis of the MFAC flow, with separate bounds established for each component and a main theorem establishing global exponential convergence under suitable conditions. We describe the machine learning algorithm in Section~\ref{sec:algorithm}, with a focus on score-based distribution representation and optimal transport maps generated by the OTGP flow.
In Section~\ref{sec:numerical_example}, we demonstrate the performance of our method on three representative MFG problems: a systemic risk model, an optimal execution problem, and a Cucker–Smale flocking model. We conclude this work in Section~\ref{sec:conclusion}, and all technical proofs are provided in the appendices.

\section{Preliminaries}\label{sec:MFG_background}

Throughout the paper, we use $|\cdot|$ to denote the absolute value of a scalar, the $\ell^2$ norm of a vector, the Frobenius norm of a matrix, or the square root of the square sum of a higher-order tensor, depending on the context. The notation $\norm{\cdot}_2$ refers to the $\ell^2$ operator norm (i.e. the largest singular value) of a matrix. We write $\Tr(\cdot)$ for the trace of a square matrix, $\inner{\cdot}{\cdot}_{\rho}$ for the $L^2$ inner product under a weight function $\rho$, and $\mL(\cdot)$ for the law of a random variable.
For a positive integer \(N\), let \([N] := \{1,2,\ldots,N\}\).

\subsection{Mean-field games}
Let $(\Omega, \mc{F}, \mb{F}=(\mc{F}_t)_{t \geq 0}, \PP)$ be a filtered probability space with  $\mathbb{F}$ being the filtration that supports a $n'$-dimensional Brownian motion $W$. Mean-field games (MFGs) study strategic interactions through the population distribution among infinitesimal players. Mathematically, given a flow of probability measures $\mu = (\mu_t)_{t \in [0,T]}$ for the population distribution on a finite time horizon \([0,T]\), the state process $(X_t)_{t\in[0,T]}$ of a representative player is governed by a stochastic differential equation (SDE) in \(\R^d\):
\begin{equation}\label{eq:Xt}
\ud X_t^{\mu, \alpha} = b(t, X_t^{\mu, \alpha}, \mu_t, \alpha_t) \ud t + \sigma(t, X_t^{\mu, \alpha}, \mu_t) \ud W_t, \quad X\ma_0 \sim \mu_0.
\end{equation}
The player aims to search for an admissible control process $(\alpha_t)_{t\in[0,T]}$, which takes values in $\RR^n$, that minimizes the expected cost 
\begin{equation}\label{eq:J}
 J^\mu[\alpha] :=  \EE\Big[\int_0^T f(t, X_t^{\mu, \alpha}, \mu_t, \alpha_t) \ud t + g(X_T^{\mu,\alpha}, \mu_T) \Big],
\end{equation}
given the running cost $f$ and terminal cost $g$. Here, the functions $b: [0,T] \times \RR^d \times \mc{P}^2(\RR^d) \times \RR^n \to \RR^d$, $\sigma: [0,T] \times \RR^d \times \mc{P}^2(\RR^d) \to \RR^{d\times n'}$, $f: [0,T] \times \RR^d \times \mc{P}^2(\RR^d) \times \RR^n \to \RR$, $g:\RR^d \times \mc{P}^2(\RR^d) \to \RR$ are all assumed to be measurable, and $\mc{P}^2(\RR^d)$ denotes the space of probability measures on $\RR^d$ with finite second moments.
\begin{assu}\label{assu:mu0_sigma0}
Assume the following hold.
\begin{itemize}
\item $\mu_0$ is standard Gaussian $\mathcal{N}(0,I_d)$, with density $\rho_0(x) = (2\pi)^{-d/2} \exp(-|x|^2/2)$.
\item Uniform ellipticity: the smallest eigenvalue of the matrix-valued function
\begin{equation}\label{def:D}
    D(t,x,\mu) := \frac12 \sigma(t,x,\mu) \sigma(t,x,\mu)\tp
\end{equation}
is bounded below by a constant $\sigma_0 > 0$ that does not depend on \(t,x,\mu\).
\end{itemize}
\end{assu}
The assumption of standard Gaussian initialization is imposed solely for convenience; the proposed algorithm extends without modification to arbitrary initial distributions.

\begin{defn}[Mean-field equilibrium]
A control-distribution pair $(\alpha^*, \mu^*)$ is called a mean-field equilibrium (MFE), if (i) given the measure flow $\mu^*$, $\alpha^*$ solves the optimal control problem \eqref{eq:Xt}--\eqref{eq:J}, and (ii) the marginal law of the optimal state dynamics $X_t^{\mu^*,\alpha^*}$ satisfies the consistency condition:
$$\mu^*_t = \mc{L}(X_t^{\mu^*, \alpha^*}), \quad \text{for all} \quad t \in [0,T].$$
\end{defn}

\begin{rem}
Existence and uniqueness of MFE have been widely studied in the literature, via reformulations in terms of PDE systems, forward-backward SDEs, or master equations. For a comprehensive discussion, we refer interested readers to \cite{carmona2018probabilistic}. In this paper, we assume that a unique MFE exists and denote it by $(\alpha^*, \mu^*)$.
\end{rem}

\medskip

Throughout this work, we focus on feedback controls of the form $\alpha_t = \alpha(t,X_t^{\mu, \alpha})$, where $\alpha$ is a deterministic function in $t$ and $x$. Given a fixed measure flow $\mu$ and a control function $\alpha$, the associated \textit{value function} is defined as 
\begin{equation}\label{eq:value_function}
V^{\mu,\alpha}(t, x) := \EE \Big[\int_t^T f(s, X_s^{\mu,\alpha}, \mu_s, \alpha_s) \, \,\rd s + g(X_T^{\mu,\alpha}, \mu_T) ~\Big|~ X\ma_t = x\Big],
\end{equation}
where superscripts $\mu$ and $\alpha$ in $V\ma$ emphasize the dependence on the population distribution and the control. The value function $V\ma$ satisfies a linear PDE
\begin{equation}\label{eq:value_PDE}
-\pt V^{\mu,\alpha}(t,x) + H\parentheses{t, x, \mu_t, \alpha(t,x), -\nabla_x V^{\mu,\alpha}(t,x), -\nabla^2_x V^{\mu,\alpha}(t,x)} = 0, \quad V\ma(T,x)=g(x,\mu_T),
\end{equation}
where the Hamiltonian $H: \RR \times \RR^d \times \mathcal{P}_2(\RR^d) \times \RR^n \times \RR^d \times \RR^{d \times d} \to \RR$ is defined as
$$H(t, x, \mu, \alpha, p, P) := \frac12 \Tr \left( P \sigma(t, x, \mu) \sigma(t, x, \mu)\tp \right) + b(t, x, \mu, \alpha)\tp p - f(t, x, \mu, \alpha).$$
The density $\rho^{\mu,\alpha}(t,x)$ of $X_t^{\mu,\alpha}$ satisfies the FP equation (recall $D$ defined in \eqref{def:D})
\begin{equation}\label{eq:FokkerPlanck}
\pt \rho^{\mu,\alpha}(t,x) + \nx\cdot\parentheses{ b(t,x,\mu_t,\alpha(t,x)) \rho\ma(t,x)} = \sum_{i,j=1}^d \partial_{x_i} \partial_{x_j} \sqbra{D_{ij}(t,x,\mu_t) \rho^{\mu,\alpha}(t,x)}, \quad \rho(0,x)=\rho_0(x).
\end{equation}

For fixed $\mu$, the problem \eqref{eq:Xt}--\eqref{eq:J} reduces to a classical stochastic control problem. Let $\alpha\ms$ be the optimal control in this case, where the superscript $\mu$ emphasizes the dependence of $\alpha\ms$ on the given flow of measure $\mu$. We denote the associated value function under this control by $V\ms := V^{\mu,\alpha\ms}$. Then, by the dynamic programming principle, $V\ms$ satisfies the HJB equation (cf. \cite[Ch.~2-4]{yong2012stochastic})
\begin{equation}\label{eq:HJB}
-\pt V\ms(t,x) + \sup_{\alpha\in\RR^n} H\parentheses{t, x, \mu_t, \alpha, -\nabla_x V\ms(t,x), -\nabla^2_x V\ms(t,x)} = 0, \quad V\ms(T,x)=g(x,\mu_T),
\end{equation}
and, for any $(t,x) \in [0,T] \times \RR^d$, $\alpha\ms(t,x)$ maximizes the function
$$\alpha \mapsto H\parentheses{t, x, \mu_t, \alpha, -\nabla_x V\ms(t,x), -\nabla^2_x V\ms(t,x)}.$$

\subsection{Notations}
\begin{defn}[Wasserstein-2 distance for measure flows]
\label{defn:flow_W_metric}
Let $\mu = (\mu_t)_{t \in [0,T]}$ and $\nu = (\nu_t)_{t \in [0,T]}$ be two flows of probability measures with finite second moments. We define the \textit{flow Wasserstein-2 distance} between $\mu$ and $\nu$ as
$$\mW_2(\mu,\nu)^2 := \int_0^T W_2(\mu_t,\nu_t)^2 \,\rd t,$$
where $W_2(\cdot,\cdot)$ is the standard Wasserstein-2 distance between two probability measures on $\RR^d$.
\end{defn}
When a probability measure is absolutely continuous with respect to the Lebesgue measure, we will not distinguish between the measure itself and its Radon-Nikodym derivative (i.e., its density function). For example, although the Wasserstein distance is formally defined between probability measures, we may write $W_2(\rho_1(\cdot), \rho_2(\cdot))$ to denote the Wasserstein distance between the underlying measures associated with density functions $\rho_1$ and $\rho_2$. Similarly, we may write $\mu_t(x)$ to denote the density of $\mu_t$ when it exists. For a time-varying density function $\rho(t,x)$, we often use the shorthand notation $\rho_t := \rho(t,\cdot)$ for convenience.

\smallskip
\noindent\textbf{Weighted norms. } Given a function $V_0: \RR^d \to \RR$, we define the weighted $L^2$ norm
$$\norm{V_0(x)}_{\rho_0}^2 := \int_{\RR^d} |V_0(x)|^2 \rho_0(x) \,\rd x,$$
where the subscript specifies the weight function. Similarly, for $V:[0,T]\times\RR^d \to \RR$ and given a measure-control pair $(\mu, \alpha)$, we define
$$\norm{V(t,x)}_{\mu,\alpha}^2 \equiv \norm{V(t,x)}_{\rho\ma}^2 := \inttx{ \abs{V(t,x)}^2\rho\ma(t,x)}.$$

\noindent\textbf{Functional derivatives.}
We use the symbol $\rD^\cdot_\cdot$ to denote the functional derivative, where the subscript indicates the argument with respect to which the derivative is taken, and the superscript specifies the weight used for the inner product. For instance, the functional derivative of $J^\mu[\alpha]$ with respect to $\alpha$, under a given weight $\rho$, is denoted by $\fd{\rho}{\alpha}{J^\mu[\alpha]}$. To simplify notation, we write $\fd{\rho\ma}{\alpha}{}$ as $\fd{\mao}{\alpha}{}$.

By definition, for any controls $\alpha$, $\alpha'$ and any flows of measures $\mu$, $\mu'$,
\begin{align*}
 \quad \dfrac{\rd}{\rd \ve} J^\mu[\alpha + \ve \phi] \Big|_{\ve=0} = \inner{\fd{\mu',\alpha'}{\alpha}{J^\mu[\alpha]}}{\phi}_{\mu',\alpha'}  = \inttx{ \pfd{\mu',\alpha'}{\alpha}{J^\mu[\alpha]}(t,x) \,\, \phi(t,x) \, \rho^{\mu',\alpha'}(t,x)}
\end{align*}
for any smooth and $\rho\map$-square integrable $\phi$. Consequently, for any pair $(\mu', \alpha')$, we have the identity
\begin{equation*}
\pfd{\mu',\alpha'}{\alpha}{J^\mu[\alpha]}(t,x) \, \rho^{\mu',\alpha'}(t,x) = \parentheses{\fd{\mu,\alpha}{\alpha}{J^\mu[\alpha]}}(t,x) \, \rho^{\mu,\alpha}(t,x), \quad \forall (t,x) \in [0,T] \times \RR^d.
\end{equation*}
This holds because the first variation is geometry-independent, while the functional derivative depends on the geometry.

\section{The mean-field actor-critic flow}\label{sec:MFACflow}
In this section, we introduce the mean-field actor-critic (MFAC) flow, a learning framework for solving MFGs with general distributional dependencies.
Inspired by the actor-critic framework in RL \cite{sutton1998reinforcement}, the MFAC flow couples an actor flow, which improves the control based on policy gradient updates, with a critic flow that evaluates the value function~\eqref{eq:value_function}.
Building on geometric insights from optimal transport, we incorporate a novel distribution flow based on Wasserstein geodesics. 
Different from discrete learning schemes in the previous literature, the MFAC flow models the continuous learning dynamics through PDEs, eliminating the introduction of stochastic approximation and significantly facilitating convergence analysis. We denote by $\tau$ the continuous \emph{learning time} of the flow, which should be distinguished from the physical time variable $t$ used in the MFG.

\subsection{Actor: policy gradient flow for the control}\label{sec:PGactor}
The policy gradient theorem \cite{silver2014deterministic} is widely used for updating the actor via gradient-based methods, especially when policies are parameterized by neural networks or other function approximators. To this end, we first characterize the functional derivative of the objective \eqref{eq:J} with respect to the control function.

\begin{prop}[Policy gradient theorem]\label{prop:policy_gradient}
Under regularity conditions specified in Section~\ref{sec:convergence}, the derivative of $J^\mu[\alpha]$ with respect to $\alpha$ is 
\begin{equation}\label{eq:policy_gradient0}
\pfd{\mu,\alpha}{\alpha}{J^\mu[\alpha]}(t,x) = -\na H(t,x,\mu_t,\alpha(t,x),-\nx V\ma(t,x),-\nx^2 V\ma(t,x)).
\end{equation}
\end{prop}
The proof is at the beginning of Appendix~\ref{sec:proof_actor}. If the diffusion coefficient $\sigma$ is free of control $\alpha$, as it is in our setting, $\na H$ does not depend on the Hessian term  $-\nx^2V\ma$ and the derivative simplifies to:
\begin{equation}\label{eq:policy_gradient}
\pfd{\mu,\alpha}{\alpha}{J^\mu[\alpha]}(t,x) = -\na H(t,x,\mu_t,\alpha(t,x),-\nx V\ma(t,x)).
\end{equation}
We then consider updating the control via the gradient flow (with $\tau$ being the learning time):
\begin{equation}\label{eq:PG_flow0}
\pta \alpha^\tau(t,x) := -\pfd{\mu,\alpha}{\alpha}{J^\mu[\alpha]}(t,x) = \na H(t,x,\mu_t,\alpha^\tau(t,x), -\nx V^{\mu,\alpha^\tau}(t,x)).
\end{equation}

This gradient flow raises two challenges. Firstly, it requires instantaneous evaluation of $-\nx V^{\mu,\alpha^\tau}(t,x)$ at each $\tau$, which is nontrivial in practice. We address this in Section~\ref{sec:shooting_critic}. Secondly, the population distribution $\mu$ may not be the mean-field distribution and must also be updated dynamically. We denote the evolving flow by $\mu^\tau = (\mu^\tau_t)_{t\in[0,T]}$ and develop its update mechanism in Section~\ref{sec:OTgeodesicflow}.

\subsection{Critic: a shooting method for the value function}\label{sec:shooting_critic}

We now discuss how to compute $V\ma$ and its gradient $\nx V\ma$ for a given measure flow $\mu$ and control $\alpha$. We parametrize $V\ma(0,\cdot)$ and $\nx V\ma(\cdot,\cdot)$ using two functions $\mV_0$ and $\mG$, respectively. These are trained by minimizing the critic loss $\mL_c$:
\begin{equation}\label{eq:critic_loss}
\mL_c := \frac12 \EE\Big[\Big(\mV_0(X_0\ma) - \int_0^T f(t,X\ma_t,\mu_t,\alpha_t)\,\rd t + \int_0^T \mG(t,X\ma_t)\tp \sigma(t,X\ma_t,\mu_t) \,\rd W_t - g(X\ma_T,\mu_T)\Big)^2\Big],
\end{equation}
where $\alpha_t = \alpha(t, X\ma_t)$, and the subscript $c$ indicates the loss for the critic component.

This formulation is based on a shooting method \cite{han2018solving}. We apply It\^o's lemma to $V\ma(t,X\ma_t)$ and obtain
\begin{equation}\label{eq:BSDE_trick}
\begin{aligned}
& \quad \rd V\ma(t,X\ma_t) \\
& = \sqbra{\pt V\ma(t,X\ma_t) + b(t,X\ma_t,\mu_t,\alpha_t)\tp \nx V\ma(t,X\ma_t) + \Tr\parentheses{D(t,X\ma_t,\mu_t)\tp \nx^2 V\ma(t,X\ma_t)}}\rd t \\
& \quad + \nx V\ma(t,X\ma_t)\tp \sigma(t,X\ma_t,\mu_t) \,\rd W_t \\
& =- f(t,X\ma_t,\mu_t,\alpha_t)\,\rd t + \nx V\ma(t,X\ma_t)\tp \sigma(t,X\ma_t,\mu_t) \,\rd W_t,
\end{aligned}
\end{equation}
where the second equality follows from \eqref{eq:value_PDE}. 
Consequently,
\begin{equation}\label{eq:g_Ito}
g(X_T\ma,\mu_T) = V\ma(0,X_0\ma) - \int_0^T f(t,X\ma_t,\mu_t,\alpha_t)\,\rd t + \int_0^T \nx V\ma(t,X\ma_t)\tp \sigma(t,X\ma_t,\mu_t) \,\rd W_t,
\end{equation}
and the critic loss \eqref{eq:critic_loss} serves as the residual for the consistency condition of the value function.
The next proposition characterizes $\mL_c$.

\begin{prop}\label{prop:critic_loss}
The critic loss $\mL_c$ can be decomposed into two orthogonal terms:
\begin{equation}\label{eq:critic_loss_2terms}
\begin{aligned}
\mL_c &= \frac12 \int_{\RR^d} \parentheses{\mV_0(x) - V\ma(0,x)}^2 \rho_0(x) \,\rd x \\
& \quad + \frac12 \inttx{ \abs{\sigma(t,x,\mu_t)\tp \parentheses{ \mG(t,x) - \nx V\ma(t,x)}}^2\rho\ma(t,x)}.
\end{aligned}
\end{equation}
The derivatives of $\mL_c$ with respect to $\mV_0$ and $\mG$ are
\begin{equation}\label{eq:critic_derivative}
\pfd{\rho_0}{\mV_0}{\mL_c}(x) = \mV_0(x) - V\ma(0,x), \qquad 
\pfd{\mao}{\mG}{\mL_c}(t,x) = 2D(t,x,\mu_t) \parentheses{ \mG(t,x) - \nx V\ma(t,x)}.
\end{equation}
\end{prop}

A detailed proof is provided in Appendix~\ref{sec:proof_critic}. We remark that $\mV_0$ approximates $V\ma$ at $t = 0$, and is therefore weighted against the initial density $\rho_0(x)$. In contrast, $\mG$ depends on both $t$ and $x$ and is thus weighted by the density $\rho\ma(t,x)$.

With these explicit derivatives, we consider the critic flow (for fixed $\mu$ and $\alpha$):
\begin{equation}\label{eq:critic_flow0}
\begin{aligned}
\pta \mV_0^\tau(x) &:= - \pfd{\rho_0}{\mV_0}{\mL_c}(x) = V\ma(0,x) - \mV_0^\tau(x),\\
\pta \mG^\tau(t,x) &:= -\pfd{\mao}{\mG}{\mL_c}(t,x) = 2D(t,x,\mu_t) \parentheses{\nx V\ma(t,x) - \mG^\tau(t,x)}.
\end{aligned}
\end{equation}
This formulation offers several advantages: $\mV_0^\tau$ and $\mG^\tau$ evolve toward their true counterparts $V\ma(0,\cdot)$ and $\nx V\ma$, even though these targets are never computed explicitly. The updates require only simulations of $X\ma_t$, evaluation of the loss $\mL_c$ in \eqref{eq:critic_loss}, and computing its gradient, making it amenable to sampling-based training. Moreover, the decomposition in \eqref{eq:critic_loss_2terms} has a natural interpretation: the first term is the weighted $L^2$ error of $\mV_0$, while the second term is equivalent to the weighted $L^2$ error for $\mG$ (recall $\sigma$ is uniformly elliptic). This decomposition naturally guarantees both consistency and stability of the critic loss.  

\subsection{Distribution: optimal transport geodesic Picard flow}\label{sec:OTgeodesicflow}

A classical approach for learning the mean-field equilibria is \textit{fictitious play} \cite{brown1949some,brown1951iterative}. In this method, one first computes the optimal state density $\rho\ms$ corresponding to a given distribution flow $\mu$, then updates the distribution by setting $\mu \leftarrow \rho\ms$ (with an abuse of notation between measures and densities). A new optimal control problem is then solved under this updated measure. This iterative procedure is in the spirit of a Picard fixed-point iteration, whose convergence properties have been studied in \cite{cardaliaguet2017learning,yu2024convergence}.

We extend this idea to a continuous-time learning dynamic. Let $\mu^\tau$ and $\alpha^\tau$ be the current estimates of the distribution and the control. Following the idea of Picard iteration, for each physical time $t \in [0,T]$, we evolve $\mu^\tau_t$ along the Wasserstein-2 geodesic to $\rho\mat_t$. Mathematically, let $\varphi^\tau_t(\cdot)$ denote the Kantorovich potential \cite[Definition~1.12]{santambrogio2015optimal} for the optimal transport from $\mu^\tau_t$ to $\rho\mat_t$ under the squared Euclidean distance. We define the \textit{optimal transport geodesic Picard (OTGP) flow} as
\begin{equation}\label{eq:distribution_flow0}
\pta \mu^\tau_t(x) := \nx\cdot \parentheses{\mu^\tau_t(x) \, \nx\varphi^\tau_t(x)}, \quad \mu^\tau_0 = \rho_0.
\end{equation}
By definition of the Kantorovich potential, the map $T^\tau_t(x) := x - \nx \varphi^\tau_t(x)$ is the optimal transport from $\mu^\tau_t$ to $\rho\mat_t$, with $-\nx \varphi^\tau_t(x)$ being the associated optimal velocity field. The tangent vector $\pta \mu^\tau_t$ points in the direction of $\rho\mat_t$ along the Wasserstein-$2$ geodesic. We emphasize that the target $\rho\mat_t$ itself depends on $\tau$, so the OTGP flow is \emph{not} a standard Wasserstein geodesic flow.
 
\begin{rem}
    We stress again that for fixed $\tau$, a flow of distributions $\mu^\tau_t$ refers to the temporal evolution in the physical time $t$, while the OTGP flow describes evolution in the learning time $\tau$. In practice, parameterizing high-dimensional densities $\mu^\tau_t$ with neural networks is challenging due to the intractability of the normalizing constant. In Section~\ref{sec:algorithm}, we discuss this issue using a score-matching approach from generative modeling to avoid explicit density parameterization.
\end{rem}

\subsection{The full mean-field actor-critic flow}
Having defined the actor~\eqref{eq:PG_flow0}, critic~\eqref{eq:critic_flow0}, and distribution flows~\eqref{eq:distribution_flow0} separately, we now combine them into the MFAC flow. We introduce scaling parameters $\beta_a$, $\beta_c$, and $\beta_\mu$ to control the relative speeds of the actor, critic, and distribution components, respectively.  

In the actor flow, the gradient of the true value function $\nx V\mat$ is replaced by its estimation $\mG^\tau$. In the critic flow, the value function $V\mat$ itself evolves with the learning time $\tau$. Incorporating these elements, the full MFAC flow is defined as

\begingroup
\mathtoolsset{showonlyrefs=false}
\begin{subequations}\label{eq:MFAC_flow}
\begin{align}
\pta \alpha^\tau(t,x) &:= \beta_a \na H\big(t,x,\mu_t^\tau,\alpha^\tau(t,x), -\mG^\tau(t,x)\big) \label{eq:actor_flow}\\
\pta \mV_0^\tau(x) &:= \beta_c\parentheses{V\mat(0,x) - \mV_0^\tau(x)} \label{eq:V0_flow}\\ 
\pta \mG^\tau(t,x) &:= \beta_c\,  2D(t,x,\mu_t^\tau) \parentheses{\nx V\mat(t,x) - \mG^\tau(t,x)} \label{eq:G_flow}\\
\pta \mu^\tau_t(x) &:= \beta_\mu\nx\cdot \parentheses{\mu^\tau_t(x) \, \nx\varphi^\tau_t(x)}. \label{eq:OTGP_flow}
\end{align}
\end{subequations}
\endgroup
In the next section, we present a convergence analysis of the MFAC flow.

\section{The convergence analysis}\label{sec:convergence}

In this section, we present the convergence analysis of the MFAC flow. We begin by stating the technical assumptions used throughout. Unless otherwise specified, we assume Assumption~\ref{assu:mu0_sigma0} holds.

We first define the classes of admissible controls and distribution flows:
\begin{align*}
\mA & := \big\{\alpha:[0,T]\times\RR^d \to \RR^n ~|~~ \alpha \text{ is twice differentiable in }x\in\R^d, \,\, |\alpha(t,0)| \le K, \,\, |\nx \alpha(t,x)|\le K,  \\
&|\nx^2 \alpha(t,x)| \le K,\,\, |\alpha(t,x)-\alpha(s,x)| \le K(1+|x|)\,\,|t-s|,\,\,|\nx\alpha(t,x)-\nx\alpha(s,x)| \le K|t-s|,\big\}, \\
\mM &:= \big\{ (\mu_t)_{t\in[0,T]} \in \mP^2(\RR^d)^{[0,T]} ~|~ \mu_0 = \rho_0,~~ W_2(\mu_t,\delta_0) \le K,~~ W_2(\mu_t,\mu_s)^2 \le K|t-s| \big\}, 
\end{align*}
where $K>0$ is an absolute constant and $\delta_0$ denotes the Dirac mass at the origin.

Under Assumption~\ref{assu:basic} stated below,  if $\mu\in\mM$ and $\alpha\in\mA$, the density $\rho\ma(t,\cdot)$ of the state process satisfies an Aronson-type bound (see \cite{menozzi2021density}):
\begin{equation}\label{eq:Aronson}
c_l \exp(-C_l|x|^2) \le \rho\ma(t,x) \le C_r \exp(-c_r |x|^2), \quad \forall (t,x) \in [0,T] \times \RR^d
\end{equation}
for constants $c_l,C_l,c_r,C_r > 0$ depending only on $\sigma_0$, $d$, $K$ and $T$. In addition, we assume logarithmic Aronson bounds $\abs{\nx \log \rho\mat(t,x)} \le C (1+|x|)$ and $\abs{\nx^2 \log \rho\mat(t,x)} \le C (1+|x|^2)$, which will be used to prove a technical lemma in Appendix~\ref{sec:actor_distribution_update}. A similar bound was established in \cite[Theorem B]{sheu1991some}.

To ensure integrability and control on tails, we define the function class:
\begin{equation}\label{eq:central_class}
\begin{aligned}
\mC := \Big\{ F(t,x) ~\big|~& 
\int_{\RR^d} (1+|x|^3)\,|F(t,x)|^2 \rho(t,x)\,\rd x 
\le K \int_{\RR^d} |F(t,x)|^2 \rho(t,x)\,\rd x, \\
&\forall t \in[0,T], \quad \forall \rho \text{ satisfying the Aronson-type bound \eqref{eq:Aronson}} \Big\}.
\end{aligned}
\end{equation}
Here $F$ may be a scalar- or vector-valued function. The class $\mC$ contains functions focusing on regions of the state space that are frequently visited. This condition holds for many practical parameterizations, including polynomials and neural networks with suitable activation functions, including \texttt{sigmoid} and \texttt{tanh}. On compact domains, this condition is not needed (see \cite{zhou2024solving}).

\begin{assu}\label{assu:basic}
The functions $b$, $\sigma$, $f$ and $g$  are differentiable in $(x,\alpha)$, with classical derivatives, and satisfy the bounds:
\begin{align*}
& |b(t,x,\mu,\alpha)| \le K \big(1 + |x| + W_2(\mu,\delta_0) + |\alpha| \big), \quad |\nxa b|,~|\nxa^2 b| \le K \\
& |\sigma(t,x,\mu)|,~|\nx \sigma| ,~|\nx^2 \sigma|\le K,\\
& |f(t,x,\mu,\alpha)| \le K \big(1 + |x|^2 + W_2(\mu,\delta_0)^2 + |\alpha|^2 \big), \quad |\nxa f| \le K \big(1 + |x| + W_2(\mu,\delta_0) + |\alpha| \big) \\
& \qquad |\nxa^2 f|,~~|\nxa^3 f| \le K,\\
& |g(x,\mu)| \le K \big(1 + |x|^2 + W_2(\mu,\delta_0)^2 \big), 
\quad |\nx g| \le K \big(1 + |x| + W_2(\mu,\delta_0)\big), \quad |\nx^2 g| \le K \\
& |b(t,x,\mu,\alpha)-b(s,x,\nu,\alpha)| \le K \big[ (1+|x|+W_2(\mu,\delta_0)\vee W_2(\nu,\delta_0)+|\alpha|)|t-s|^{1/2} + W_2(\mu,\nu)\big], \\
& |\sigma(t,x,\mu)-\sigma(s,x,\nu)| \le K \big(|t-s| + W_2(\mu,\nu)\big), \\
& |f(t,x,\mu,\alpha) - f(s,x,\mu,\alpha)| \le K \big(1+|x|^2+|\alpha|^2+W_2(\mu,\delta_0)^2\big)|t-s|^{1/2}, \\
& |f(t,x,\mu,\alpha) - f(t,x,\nu,\alpha)| \le K \big(1+|x|+|\alpha|+ W_2(\mu,\delta_0)\vee W_2(\nu,\delta_0)\big)\, W_2(\mu,\nu). 
\end{align*}
Here, $\mu,\nu \in \mP^2(\RR^d)$, and $\nxa$ represents the gradient with respect to both $x$ and $\alpha$. 
In addition, we assume $\nxa f$, $\nxa^2 f$, $\nxa b$, $\nxa^2 b$, $\nx \sigma$, $\nx^2 \sigma$, $g$, $\nx g$ are all $K-$Lipschitz in $\mu$ with respect to $W_2(\cdot,\cdot)$.
\end{assu}

The derivative bounds above imply Lipschitz continuity. For instance, $|\nx b| \le K$ implies $|b(t,x,\mu,\alpha)-b(t,x',\mu,\alpha)| \le K |x-x'|$.

\begin{assu}\label{assu:Hconcave}
The Hamiltonian $H$ is $\lam_H$-strongly concave in $\alpha$, i.e.,
$$\alpha \mapsto H(t,x,\mu_t,\alpha,-\nx V(t,x), -\nx^2 V(t,x))$$
is $\lam_H$-strongly concave.
\end{assu}

 In the linear-quadratic (LQ) case where $f(t,x,\mu,\alpha) = \frac12 |x|^2 + \frac12 |\alpha|^2$ and $b(t,x,\mu,\alpha) = \alpha$, the Hamiltonian takes the form
$$H(t,x,\mu_t,\alpha,p,P) = -\frac12 |x|^2 - \frac12 |\alpha|^2 + p\tp \alpha + \Tr(P D(t,x,\mu_t)),$$
which is strongly concave in $\alpha$ with $\lam_H = 1$.

\begin{assu}\label{assu:flow_in_class}
The parametrized functions $\alpha^\tau, \alpha\mts \in \mA$, $\mu^\tau \in \mM$, and $\pta \alpha^\tau, \alpha^\tau - \alpha\mts \in \mC$. The approximation $\mG^\tau$ is $K$-Lipschitz in $x$ with $|\mG^\tau(t,0)| \le K$.
\end{assu}

 These conditions guarantee the regularity of the actor, critic, and distribution flows.

\subsection{Convergence of the actor}\label{sec:convergence_actor}
For the actor, we define the Lyapunov function as
\begin{equation}\label{eq:Lyapunov_actor}
\mL_a^\tau := J^{\mu^\tau}[\alpha^\tau] - J^{\mu^\tau}[\alpha^{\mu^\tau,*}],
\end{equation}
which measures the suboptimality of the current control $\alpha^\tau$ under the distribution $\mu^\tau$. By definition, $\mL_a^\tau \ge 0$, with equality if and only if $\alpha^\tau$ is optimal for $\mu^\tau$.
\begin{thm}[Actor convergence]\label{thm:actor_convergence}
Let Assumptions~\ref{assu:basic}-\ref{assu:flow_in_class} hold. Under the MFAC flow \eqref{eq:MFAC_flow}, the actor Lyapunov function $\mL_a^\tau$ satisfies
\begin{equation}\label{eq:actor_improvement}
\begin{aligned}
\pta \mL^\tau_a  &\le - c_a \beta_a \mL_a^\tau - \frac12 \beta_a \norm{\na H(t,x,\mu_t,\alpha^\tau(t,x), -\mG^\tau(t,x))}^2_\mato \\
&\quad  + \frac12 \beta_a K^2 \big\|\nx V\mat - \mG^\tau\big\|_\mato^2  + C_a \beta_\mu \mL_a^\tau,
\end{aligned}
\end{equation}
where $C_a, c_a > 0$ are constants independent of $\beta_a$, $\beta_c$, and $\beta_\mu$.
\end{thm}

The first term $-c_a \beta_a \mL_a^\tau$ shows exponential decay of the cost gap $\mL_a^\tau$, in the absence of errors and distribution updates. The second term further decreases the Lyapunov function and will be used to offset the positive contributions from the critic and distribution updates. The third term captures the error due to approximating $\nx V\mat$ via $\mG^\tau$ in the actor flow. The term $C_a \beta_\mu \mL_a^\tau$ addresses the dependence of $\mL^\tau_a$ on the evolving distribution $\mu^\tau$, contributing a positive term proportional to the distribution update speed $\beta_\mu$.

\subsection{Convergence of the critic}\label{sec:convergence_critic}
For the critic, we define the Lyapunov function analogously to \eqref{eq:critic_loss}
\begin{equation}\label{eq:Lyapunov_critic}
\begin{aligned}
\mL_c^\tau &:= \frac12 \EE \bigg[ \Big( \mV_0^\tau(X_0\mat) - \int_0^T f(t,X\mat_t,\mu_t^\tau,\alpha_t^\tau)\,\rd t  \\
& \qquad + \int_0^T \mG^\tau(t,X\mat_t)\tp \sigma(t,X\mat_t,\mu^\tau_t) \rd W_t - g(X\mat_T,\mu^\tau_T)\Big)^2\bigg],
\end{aligned}
\end{equation}
where $\alpha^\tau_t = \alpha^\tau(t,X_t\mat)$. By Proposition~\ref{prop:critic_loss},
\begin{equation}\label{eq:Lyapunov_critic2}
\begin{aligned}
\mL_c^\tau &= \frac12 \int_{\RR^d} \parentheses{\mV_0^\tau(x) - V\mat(0,x)}^2 \rho_0(x) \,\rd x \\
& \qquad + \frac12 \inttx{ \abs{\sigma(t,x,\mu_t^\tau)\tp \parentheses{ \mG^\tau(t,x) - \nx V\mat(t,x)}}^2\rho\mat(t,x)}.
\end{aligned}
\end{equation}
\begin{thm}[Critic convergence]\label{thm:critic_convergence}
Let Assumptions~\ref{assu:basic}-\ref{assu:flow_in_class} hold. Under the MFAC flow \eqref{eq:MFAC_flow}, the critic Lyapunov function $\mL_c^\tau$ satisfies
\begin{equation}\label{eq:critic_improvement}
\begin{aligned}
\pta \mL^\tau_c \le -c_c\beta_c \mL_c^\tau + \dfrac{C_c}{\beta_c} \Big[ \beta_\mu^2 \mW_2\parentheses{\mu^\tau, \rho\mat}^2 + \beta_\mu^2 W_2\parentheses{\mu^\tau_T, \rho\mat_T}^2 \\
+ \beta_a^2 \norm{\na H(t,x,\mu_t,\alpha^\tau(t,x), -\mG^\tau(t,x))}^2_\mato \Big],
\end{aligned}
\end{equation}
where $C_c, c_c>0$ are constants independent of $\beta_a$, $\beta_c$, and $\beta_\mu$.
\end{thm}
The term $-c_c \beta_c \mL_c^\tau$ indicates the exponential decay of the critic loss under fixed distribution-control pairs. However, both $\mu^\tau$ and $\alpha^\tau$ evolve with $\tau$, leading to variation in $V\mat$. This contributes to the other terms weighted by $\beta_\mu^2$ and $\beta_a^2$.

\subsection{Convergence of the distribution}\label{sec:convergence_distribution}
To aid the convergence analysis (see Lemma~\ref{lem:controaction_FPI}), we define a weighted Wasserstein-2 metric with \(\beta>0\):
\begin{equation}\label{eq:weighted_distance}
d_\beta(\mu,\nu)^2 := \int_0^T e^{-2\beta t}\, W_2(\mu_t,\nu_t)^2 \,\rd t, \quad  \mu = (\mu_t)_{t \in [0,T]}, \quad \nu = (\nu_t)_{t \in [0,T]}.
\end{equation}
This is equivalent to $\mW_2$ since $e^{-\beta T} \mW_2(\mu,\nu) \le d_\beta(\mu,\nu) \le \mW_2(\mu,\nu)$. In the sequel, we set $\beta = 34K^2 + \frac{51}{2}K$ and $\lam_T = \min\{\frac{1}{4C_T}, e^{-2\beta T}\}$, where $C_T$ is a constant depending only on $d$, $T$, and $K$ (see \eqref{eq:distribution_term5} in Appendix~\ref{sec:proof_distribution}).

We now define the Lyapunov function for the distribution as
$$\mL^\tau_\mu := \frac{1}{2} d_\beta(\mu^\tau, \rho\mat)^2 + \frac12 \lam_T W_2(\mu^\tau_T, \rho\mat_T)^2,$$
which penalizes the discrepancy between $\mu^\tau$ and its one-step Picard update $\rho\mat$. The additional term with weight $\lambda_T$ is included to control the terminal error that has arisen in the critic estimate (cf.~\eqref{eq:critic_improvement}).
\begin{thm}[Distribution convergence]\label{thm:distribution_convergence}
Let Assumptions~\ref{assu:basic}-\ref{assu:flow_in_class} hold. Under the MFAC flow \eqref{eq:MFAC_flow}, the distribution Lyapunov function $\mL_\mu^\tau$ satisfies
\begin{equation}\label{eq:distribution_improvement}
\pta \mL^\tau_\mu \le - c_\mu \beta_\mu \mL^\tau_\mu + C_\mu \dfrac{\beta_a^2}{\beta_\mu} \norm{\na H(t,x,\mu_t,\alpha^\tau(t,x), -\mG^\tau(t,x))}_\mato^2,
\end{equation}
where $C_\mu, c_\mu > 0$ are constants independent of $\beta_a$, $\beta_c$, and $\beta_\mu$.
\end{thm}
The first term, $-c_\mu \beta_\mu \mL^\tau_\mu$, shows that the Lyapunov function for the distribution decays exponentially when the control is held fixed. The second term arises because the control $\alpha^\tau$ is evolving with $\tau$.

\begin{rem}[OTGP for McKean--Vlasov SDEs]
The OTGP flow also provides a method to solve FP equations associated with McKean--Vlasov SDEs. When the control $\alpha$ is fixed (i.e., $\beta_a = 0$), Theorem~\ref{thm:distribution_convergence} implies 
\[
\pta \mL^\tau_\mu \le -c_\mu \beta_\mu \mL^\tau_\mu,
\]
showing that $\mu^\tau$ converges exponentially fast to the solution of the McKean--Vlasov SDE. In Lemma~\ref{lem:controaction_FPI}, we prove that the Picard map $\mu \mapsto \rho\ma$ is a contraction under the metric $d_\beta(\cdot,\cdot)$, with the fixed point corresponding to the density of the McKean--Vlasov SDE for a fixed control. The OTGP flow can thus be interpreted as a continuous-time analogue of the Picard iteration.
\end{rem}

\subsection{Main result: convergence of the MFAC flow}
We now combine the convergence results from Sections~\ref{sec:convergence_actor}--\ref{sec:convergence_distribution} to establish global convergence of the MFAC flow. To this end, we define the total Lyapunov function as
\begin{equation}\label{eq:Lyapunov_total}
\mL^\tau_{\text{total}} = \mL^\tau_a + \mL^\tau_c + \lam_\mu \mL^\tau_\mu,
\end{equation}
where $\lam_\mu = \beta_\mu / (4\beta_a C_\mu) > 0$ weights the distribution component. The update speeds $\beta_c$, $\beta_a$, and $\beta_\mu$ are chosen to satisfy
\begin{equation}\label{eq:speed_ratio}
\dfrac{\beta_a}{\beta_c} \le \min\curlybra{\dfrac{\sigma_0 c_c}{K^2}, \dfrac{1}{4C_c}, \dfrac{\lam_Tc_\mu}{16\, C_cC_\mu}}, \quad \dfrac{\beta_\mu}{\beta_a} \le \dfrac{c_a}{2C_a}, \quad \dfrac{\beta_\mu}{\beta_c} \le \dfrac{\lam_T \lam_\mu c_\mu}{4C_c}.
\end{equation}
In practice, these conditions are met by choosing $\beta_c$ sufficiently large relative to $\beta_a$, and $\beta_\mu$ sufficiently small relative to $\beta_a$. The last condition in \eqref{eq:speed_ratio} is automatically satisfied with our choice of $\lam_\mu$. With this setup, we obtain the main convergence result.
\begin{thm}[Convergence of MFAC flow]\label{thm:MFAC_convergence}
Let Assumptions~\ref{assu:basic}-\ref{assu:flow_in_class} hold. Then under the MFAC flow \eqref{eq:MFAC_flow} with parameters satisfying \eqref{eq:speed_ratio}, the total Lyapunov function \eqref{eq:Lyapunov_total} satisfies
\begin{equation}\label{eq:total_Lyapunov_improvement}
\pta \mL^\tau_{\text{total}} \le - c_\mL \mL^\tau_{\text{total}}, \quad \text{where} \quad  c_\mL := \frac12 \min\{c_a\beta_a, c_c\beta_c, c_\mu \beta_\mu \} > 0.
\end{equation}
\end{thm}
\begin{proof}
With \eqref{eq:speed_ratio}, we can verify that $\frac12 c_c \beta_c \ge \dfrac{K^2}{2\sigma_0}\beta_a$, $\frac12 \lam_\mu c_\mu \beta_\mu \ge \dfrac{2C_c}{\lam_T} \dfrac{\beta_\mu^2}{\beta_c}$, $\frac12 c_a \beta_a \ge C_a \beta_\mu$, \\
and $\frac12 \beta_a = \frac14 \beta_a + \frac14 \beta_a \ge C_c \dfrac{\beta_a^2}{\beta_c} + \lam_\mu C_\mu \dfrac{\beta_a^2}{\beta_\mu}$.
Then, combining the results in Theorems~\ref{thm:actor_convergence}--\ref{thm:distribution_convergence}, and using the fact that $\mL_c^\tau \ge \sigma_0 \|\nx V\mat - \mG^\tau\|^2_\mato$, we obtain
\begin{align*}
\pta \mL^\tau_{\text{total}} & = \pta \parentheses{\mL^\tau_a + \mL^\tau_c + \lam_\mu \mL^\tau_\mu} \\
& \le  - (c_a \beta_a - C_a \beta_\mu) \mL_a^\tau - \frac12 \beta_a \norm{\na H(t,x,\mu_t,\alpha^\tau(t,x), -\mG^\tau(t,x))}^2_\mato + \dfrac{K^2}{2\sigma_0}\beta_a \mL^\tau_c \\
& \qquad -c_c\beta_c \mL_c^\tau + \dfrac{2C_c}{\lam_T} \dfrac{\beta_\mu^2}{\beta_c} \mL^\tau_\mu + C_c \dfrac{\beta_a^2}{\beta_c} \norm{\na H(t,x,\mu_t,\alpha^\tau(t,x), -\mG^\tau(t,x))}^2_\mato \\
& \qquad + \lam_\mu \parentheses{- c_\mu \beta_\mu \mL^\tau_\mu + C_\mu \dfrac{\beta_a^2}{\beta_\mu} \norm{\na H(t,x,\mu_t,\alpha^\tau(t,x), -\mG^\tau(t,x))}^2_\mato} \\
& \le -\frac12 \parentheses{c_a \beta_a \mL^\tau_a + c_c\beta_c \mL_c^\tau + \lam_\mu c_\mu \beta_\mu \mL^\tau_\mu } \le -c_\mL \mL^\tau_{\text{total}}.
\end{align*}
\end{proof}
Theorem~\ref{thm:MFAC_convergence} informs that each of the three Lyapunov functions $\mL_a^\tau$, $\mL_c^\tau$, and $\mL_\mu^\tau$ decays exponentially to $0$. Overall, the theorem establishes global exponential convergence of the MFAC flow to the mean-field equilibrium. The proof relies on a delicate balance between actor, critic, and distribution updates. The critic converges rapidly for sufficiently large $\beta_c$, ensuring accurate approximation of the value function gradient. The actor then improves the policy exponentially fast, provided that $\beta_a$ is neither too large relative to $\beta_c$ nor too small relative to $\beta_\mu$. The distribution converges under the OTGP flow. Together, these conditions guarantee that the combined system is stable and that the total Lyapunov function decreases monotonically at rate $c_\mL > 0$.

Theorem~\ref{thm:MFAC_convergence} further implies that convergence holds when the actor, critic, and distribution are updated on a single timescale. This motivates the use of a single-timescale algorithm numerically, which is more efficient than multi-timescale approaches \cite{chen2022single}.

\section{Numerical algorithm}\label{sec:algorithm}

With the convergence of the MFAC flow~\eqref{eq:MFAC_flow} established in Section~\ref{sec:convergence}, we discretize and approximate the continuous learning dynamics, yielding a deep reinforcement learning algorithm that effectively solves MFGs.
Different from most existing methods, we borrow techniques from generative modeling and optimal transport, facilitating a widely applicable distributional parameterization that solves general MFGs. In this section, we introduce the details for numerical implementation and flow approximation. A complete numerical algorithm is summarized in Algorithm~\ref{alg:MFAC}. Numerical results are presented in Section~\ref{sec:numerical_example}.

\medskip
\noindent\textbf{Time discretization.}
In the numerical implementation, both the physical time \(t \in [0,T]\) and the learning time \(\tau\) are discretized.
The physical horizon \([0,T]\) is partitioned into \(N_T\) subintervals of equal lengths \(h := T/N_T\), with grid points \(\Delta := \{jh:j\in\{0,1,\ldots,N_T - 1\}\}\).
The learning horizon is discretized with stepsize \(\Delta \tau\). We denote by \(k\) the index of the current training iteration and truncate the learning horizon at \(k_{\mathrm{end}} \Delta \tau\), resulting in a total of \(k_{\mathrm{end}}\) iterations. In what follows, we use \(\tau\) and \(k\) interchangeably, with the relation \(\tau = k \Delta \tau\).

\medskip
\noindent\textbf{Neural network parameterization.}
To capture time inhomogeneity, independent neural networks are used at each physical time step \(t\in\Delta\).
The feedback control function $\alpha^\tau$, the initial value function $\mV_0^\tau$ and the state gradient of the value function $\mG^\tau$ are parameterized respectively by the following neural networks:
$$\mA(t,x;\theta_a^\tau)\in \R^n, ~~ \mV_0(x;\theta_c^\tau)\in \R, ~~ \mG(t,x;\theta_c^\tau)\in \R^d, \quad  \forall (t,x) \in \Delta\times\R^d,$$
where \(\theta^\tau_a\) and \(\theta^\tau_c\) denote the actor and critic network parameters at learning time \(\tau\).
For distributions \(\mu^\tau\), we parameterize the score function \(s_{\mu^\tau_t}(x) := \nabla_x \log\mu^\tau_t(x)\) using a  score network $\mS(t,x;\theta^\tau_s) \in\R^d, \ \forall (t,x) \in \Delta\times \R^d$,
where \(\theta^\tau_s\) denotes the score network parameters.

\medskip
\noindent\textbf{SDE simulation.} 
All SDEs are simulated forward in time on the grid \(\Delta\) using the Euler-Maruyama scheme, producing \(N_\batch\) independent sample paths.
Given a flow of measures \((\tilde{\mu}^k_t)_{t\in\Delta}\), the state process $X_t$ defined in \eqref{eq:Xt} is approximated by:
\begin{equation}\label{eqn:FSDE_simu}
\tilde{X}_{t+h}^{k,m} = \tilde{X}_{t}^{k,m} + b(t, \tilde{X}_{t}^{k,m}, \tilde{\mu}_{t}^k, \tilde{\alpha}^{k,m}_t) \,h + \sigma(t, \tilde{X}_{t}^{k,m}, \tilde{\mu}_{t}^k) \,\sqrt{h}\,\xi^{k,m}_t,\ \forall t\in\Delta,\ m\in [N_{\batch}],\ k\in[k_{\mathrm{end}}],
\end{equation}
where \(\tilde{X}_{t}^{k,m}\) denotes the \(m\)-th simulation path during the \(k\)-th training iteration, and \(\xi^{k,m}_t\overset{\mathrm{i.i.d.}}{\sim} \mathcal{N}(0,1)\). The control \(\tilde{\alpha}^{k,m}_t:= \mA(t,\tilde{X}^{k,m}_t;\theta^\tau_a)\) is computed from the actor network at the current state and time.
In subsequent discussions, we introduce the construction of \((\tilde{\mu}^k_t)\) based on score networks.

\medskip
\noindent\textbf{Langevin Monte Carlo (LMC).}
To sample from the distribution $\mu_t^\tau$, we simulate the associated overdamped Langevin diffusion:
\[
\rd L_u = \tfrac{1}{2}s_{\mu^\tau_t}(L_u)\,\rd u + \rd B_u, 
\]
where $B$ is a standard Brownian motion. Under standard ergodicity assumptions, the law of $L_u$ converges to the stationary distribution $\pi=\mu^\tau_t$ as $u \to \infty$, providing approximate samples from $\mu^\tau_t$ \cite{erdogdu2021convergence}.

In our algorithm, LMC generates random samples associated with the score network $\mS$, which are then used to construct empirical measures for the mean-field interaction terms. For each $t \in \Delta$, we simulate $N_{\text{batch}}$ independent paths on the grid $\Delta^{\text{LMC}} := \{jh^{\text{LMC}}: j=0,1,\dots,N_T^{\text{LMC}}-1\}$ with step size $h^{\text{LMC}} = T^{\text{LMC}}/N_T^{\text{LMC}}$:
\begin{equation}\label{eqn:LMC}
    L_{u+h^\LMC}^{k,m,t} = L^{k,m,t}_u + \frac{1}{2}\mS(t,L^{k,m,t}_u;\theta_s^\tau)\,h^\LMC + \sqrt{h^\LMC}\xi^{\LMC,k,m,t}_u,\ \forall u\in\Delta^{\text{LMC}},\ m\in [N_{\text{batch}}],
\end{equation}
where \(\xi^{\LMC,k,m,t}_u\overset{\mathrm{i.i.d.}}{\sim}\mathcal{N}(0,1)\) are independent of \(\xi^{k,m}_t\) (cf. \eqref{eqn:FSDE_simu}).
With a sufficiently large \(T^\LMC\), the terminal values \(L^{k,m,t}_{T^\LMC}\) approximate $\mu_t^\tau$ via their empirical measure
\begin{equation}
    \label{eqn:LMC_empirical}
    \tilde{\mu}^k_t := \frac{1}{N_{\text{batch}}}\sum_{m\in[N_\batch]}\delta_{L_{T^{\text{LMC}}}^{k,m,t}},\ \forall t\in\Delta,
\end{equation}
where \(\delta_x\) denotes a Dirac measure centered at \(x\). The mean-field interaction terms in \eqref{eqn:FSDE_simu} are thus evaluated at such empirical measures.

\medskip
\noindent\textbf{The distribution flow.}
To discretize the OTGP flow over a small time interval $[\tau, \tau + \Delta \tau]$, we adopt a particle-based interpretation: each particle $x \sim \mu^\tau_t$ moves along the velocity $-\beta_\mu\nabla_x \varphi^\tau_t(x)$. This gives the approximation:
\begin{equation*}
\mu^{\tau+\Delta\tau}_t \approx [\mathrm{id}-\Delta\tau\beta_\mu \nabla_x\varphi^\tau_t]_\#\mu^\tau_t = [\Delta\tau \beta_\mu T^\tau_t + (1-\Delta\tau\beta_\mu)\mathrm{id}]_\#\mu^\tau_t,
\end{equation*}
where \(\mathrm{id}\) denotes the identity map and \(T^\tau_t(x) = x - \nx \varphi^\tau_t(x)\). This update lies on the Wasserstein-2 geodesic between \(\mu^\tau_t\) and \(\rho\mat_t\), and can be understood as a measure-valued Krasnosel’skii–Mann iteration along the Wasserstein-$2$ geodesic. 

Numerically, to construct synthetic samples from $\mu_t^{\tau+\Delta\tau}$, we approximate the optimal transport map $T_t^\tau$ between samples $\{L_{T^{\text{LMC}}}^{k,m,t}\}_{m\in[N_\batch]}\sim \mu_t^\tau$ and $\{\tilde{X}^{k,m}_t\}_{m\in[N_\batch]}\sim \rho_t^{\mato}$. $T_t^\tau$ are computed using the Hungarian algorithm in $\mO(N_\text{batch}^3)$ operations, and the updated samples are
\begin{equation}\label{eqn:synthetic}
Q^{k+1,m}_t := \Delta\tau\beta_\mu T^\tau_t(L_{T^{\text{LMC}}}^{k,m,t}) + (1 - \Delta\tau\beta_\mu)L_{T^{\text{LMC}}}^{k,m,t}.
\end{equation}

\noindent\textbf{Score matching.}
A key advantage of score-based parameterization is its data-driven learnability: the score can be estimated directly from samples without evaluating the underlying density. This idea, known as \emph{score matching}, was introduced in \cite{hyvarinen2005estimation} and has become a foundational tool in modern generative modeling \cite{song2019generative}. 

For each \(t \in \Delta\), given synthetic samples \(\{Q_t^{k+1,m}\}_{m\in[N_\batch]}\) from \(\mu_t^{\tau+\Delta\tau}\) and the score network \(\mS(t,\cdot;\theta_s^\tau)\), we update the parameters to \(\theta_s^{\tau+\Delta\tau}\) so that \(\mS(t,\cdot;\theta_s^{\tau+\Delta\tau})\) approximates the score function of \(\mu_t^{\tau+\Delta\tau}\). A natural objective is to minimize \(\tfrac{1}{2}\mathbb{E}_{Y\sim \mu^{\tau+\Delta\tau}_t} |\mS(t,Y;\theta_s) - s_{\mu^{\tau+\Delta\tau}_t}(Y)|^2,\) which is equivalent, by \cite[Theorem~1]{hyvarinen2005estimation}, to minimizing 
\[
\EE_{Y\sim \mu_t^{\tau+\Delta\tau}} \Big[ \nabla_x \cdot \mS(t,Y;\theta_s) + \tfrac12 |\mS(t,Y;\theta_s)|^2 \Big].
\]
Approximating the expectation with Monte Carlo samples leads to the score-matching loss
\begin{equation}
    \label{eqn:loss_score}
    \mathscr{L}_s^\tau(\theta_s) := \frac{1}{N_T}\sum_{t\in\Delta}\frac{1}{N_{\text{batch}}} \sum_{m\in[N_\batch]} \Big[\nabla_x \cdot \mS(t,Q^{k+1,m}_t;\theta_s) + \tfrac{1}{2}|\mS(t,Q^{k+1,m}_t;\theta_s)|^2\Big].
\end{equation}
The divergence term is computed via automatic differentiation, and the parameters \(\theta_s\) are updated using standard first-order optimizers such as Adam.

\medskip
\noindent\textbf{The critic flow.} The discretized shooting loss~\eqref{eq:critic_loss} for the value function is
\begin{equation}\label{eqn:loss_critic}
\begin{aligned}
\mathscr{L}^\tau_c(\theta_c) := \frac{1}{N_\batch} \sum_{m\in[N_\batch]} \Big[ \mV_0(\tilde{X}_{0}^{k,m};\theta_c) - \sum_{t\in\Delta} f(t, \tilde{X}_{t}^{k,m}, \tilde{\mu}^k_{t}, \mA(t, \tilde{X}_{t}^{k,m};\theta_a^\tau))\, h \\
+ \sum_{t\in\Delta} \mG(t, \tilde{X}_{t}^{k,m};\theta_c)\tp \sigma(t, \tilde{X}_{t}^{k,m}, \tilde{\mu}^k_{t}) \,\sqrt{h}\,\xi^{k,m}_t - g(\tilde{X}_{T}^{k,m}, \tilde{\mu}^k_T)\Big]^2,
\end{aligned}
\end{equation}
where  \(\xi^{k,m}_t\) are the same Brownian increments used in the state dynamics~\eqref{eqn:FSDE_simu}.

\medskip
\noindent\textbf{The actor flow.} 
Discretizing the actor flow~\eqref{eq:actor_flow} yields
$$\alpha^{\tau+\Delta\tau}(t,x) \approx \alpha^\tau(t,x) + \beta_a\dta\, \na H(t,x,\mu^\tau_t, \alpha^\tau(t,x), -\mG^\tau(t,x)).$$
Replacing the control and value gradient terms with neural network counterparts gives the actor loss:
\begin{equation}\label{eqn:loss_actor}
\begin{aligned}
\mathscr{L}_a(\theta_a) &= \sum_{t\in\Delta}\frac{1}{N_\batch}\sum_{m\in[N_\batch]}\Big[\mA(t,\chi^{k,m}_t;\theta_a) - \mA(t,\chi^{k,m}_t;\theta_a^\tau) \\
& \quad- \beta_a\dta\, \nabla_\alpha H(t,\chi^{k,m}_t,\tilde{\mu}^k_t,\mA(t,\chi^{k,m}_t;\theta_a^\tau),-\mG(t,\chi^{k,m}_t;\theta_c^\tau))\Big]^2.
\end{aligned}
\end{equation}
Notable, \(\chi^{k,m}_t\) are \(\mathrm{i.i.d.}\) samples uniformly drawn from \(C_{\chi}^{k,t}\subset\R^d\) using Latin hypercube sampling \cite{stein1987large}, independent of the state trajectories \(\tilde{X}^{k,m}_t\).
The sampling region \(C_\chi^{k,t}\) is chosen as a hypercube centered at the empirical mean of \(\{\tilde{X}_t^{k,m}\}_{m\in[N_\batch]}\), with side lengths equal to \(\pm 3\) standard deviations in each coordinate.

\begin{algorithm}[htpb]
\renewcommand{\algorithmicrequire}{\textbf{Input:}}
\renewcommand{\algorithmicensure}{\textbf{Output:}}
\caption{MFAC: a deep reinforcement learning algorithm solving MFGs}
\begin{algorithmic}[1]
\REQUIRE Actor, critic, and score networks \(\mA(t,\cdot),\ \mV_0(\cdot),\ \mG(t,\cdot),\ \mS(t,\cdot),\ \forall t\in\Delta\)
\STATE Initialize network parameters \(\theta_a\), \(\theta_c\), \(\theta_s\) and synthetic samples \(\{Q^{1,m}_t\}_{m\in[N_\batch],\,t\in\Delta}\) 
\FOR{\(k = 0\) to \(k_{\text{end}} - 1\)}
\STATE \(\tau = k\dta\)
\STATE Update \(\theta_s\) for \(N_s\) epochs using the score-matching loss~\eqref{eqn:loss_score}.
\STATE Construct the flow of empirical measures $\{\tilde{\mu}^k_t\}_{t \in \Delta}$ via Langevin Monte Carlo~\eqref{eqn:LMC}--\eqref{eqn:LMC_empirical}.
\STATE Simulate state trajectories $\{\tilde{X}_{t}^{k,m}\}_{t \in \Delta}$ via the Euler scheme~\eqref{eqn:FSDE_simu}.
\STATE Construct synthetic samples $\{Q^{k+1,m}_t\}_{t \in \Delta}$ via optimal transport~\eqref{eqn:synthetic}.
\STATE Update \(\theta_c\) for \(N_c\) epochs using the critic loss~\eqref{eqn:loss_critic}.
\STATE Update \(\theta_a\) for \(N_a\) epochs using the actor loss~\eqref{eqn:loss_actor}.
\ENDFOR
\ENSURE Trained networks approximating the mean-field equilibrium
\end{algorithmic}
\label{alg:MFAC}
\end{algorithm}

\section{Numerical experiments}\label{sec:numerical_example}

In this section, we evaluate MFAC (Algorithm~\ref{alg:MFAC}) on three MFG models: the systemic risk model (Section~\ref{sec:SR}), the optimal execution problem (Section~\ref{sec:Trader}), and the Cucker--Smale flocking model (Section~\ref{sec:Flocking}). These examples range from semi-analytically tractable cases to complex models without analytical solutions, allowing us to assess MFAC under varied levels of difficulty and distributional dependence. All experiments are implemented in \texttt{PyTorch} and run on an Nvidia GeForce RTX 2080 Ti GPU. The choice of hyperparameters are listed in Appendix~\ref{app:hyper}.

\smallskip
\noindent\textbf{Evaluation metrics.} Performance is measured using the relative error in value (REV) and the relative mean square error (RMSE), based on \(N_{\text{test}}=25000\) trajectories. Let \((\hat{X}^m_t,\hat{\alpha}^m_t,\hat{M}_t)\) denote the baseline equilibrium state, control, and population mean, and \((\tilde{X}^m_t,\tilde{\alpha}^m_t,\tilde{M}_t)\) the MFAC counterparts generated by score networks, which contain coupled effects of \(\mA\) and \(\mS\). To separately evaluate the actor and score, we also simulate \((\check{X}^m_t,\check{\alpha}^m_t,\check{M}_t)\) based on empirical measures \(\check{\mu}_t := \frac{1}{N_{\mathrm{test}}}\sum_{m\in[N_{\mathrm{test}}]}\delta_{\check{X}^m_t}\), without involving score networks. Corresponding expected costs are denoted by \(\hat{J}, \tilde{J}, \check{J}\).  

The pathwise RMSEs for equilibrium states and controls are defined as follows:
\begin{equation}
    \mathrm{RMSE}_X := \sqrt{\frac{\sum_{t\in\Delta,m\in[N_{\mathrm{test}}]}(\hat{X}^m_t - \check{X}^m_t)^2}{\sum_{t\in\Delta,m\in[N_{\mathrm{test}}]} (\hat{X}^m_t)^2}}, \quad \mathrm{RMSE}_\alpha := \sqrt{\frac{\sum_{t\in\Delta,m\in[N_{\mathrm{test}}]}(\hat{\alpha}^m_t - \check{\alpha}^m_t)^2}{\sum_{t\in\Delta,m\in[N_{\mathrm{test}}]} (\hat{\alpha}^m_t)^2}},
\end{equation}
The RMSE for population mean and the REV are defined as
\begin{equation}
    \mathrm{RMSE}_M := \sqrt{\frac{\sum_{t\in\Delta}(\hat{M}_t - \check{M}_t)^2}{\sum_{t\in\Delta} (\hat{M}_t)^2}},\quad \mathrm{REV} := \Big|\frac{\hat{J} - \check{J}}{\hat{J}}\Big|.
\end{equation}
The metric \(\mathrm{RMSE}_M\) is particularly informative when mean-field interactions depend only on the population mean (as in Sections~\ref{sec:SR} and~\ref{sec:Trader}), while REV offers a value-based summary of overall performance.

\subsection{Systemic risk model}\label{sec:SR}

We begin with a linear-quadratic (LQ) model of interbank borrowing and lending among infinitely many identical banks \cite{carmona2015mean}. Each bank controls its borrowing or lending rate from the central bank, and is penalized for deviations from the population average. We focus on the one-dimensional case \(d=n=n'=1\).

\smallskip
\noindent\textbf{Model setup.}
The log-monetary reserve \(X_t\) of a representative bank evolves as:
\begin{equation}\label{eqn:Xt-ex1}
    \ud X_t = [a(\BAR{\mu_t} - X_t) + \alpha_t] \ud t + \sigma \ud W_t,\ X_0\sim \mu_0,
\end{equation}
where \(\BAR{\mu_t}\) denotes the mean of the measure \(\mu_t\).
The agent aims to minimize the cost~\eqref{eq:J} with
\begin{equation*}
    f(t, x,\mu, \alpha) = \tfrac{1}{2} \alpha^2 - q\alpha(\BAR{\mu} - x) + \tfrac{1}{2} \varepsilon(\BAR{\mu} - x)^2,\quad g(x, \mu) = \tfrac{1}{2} c(x - \BAR{\mu})^2.
\end{equation*}
We assume \(a,q,c\geq 0,\ \sigma>0,\ q^2\leq\varepsilon\) for well-posedness. The exact solution is presented in Appendix~\ref{app:baseline_SR}.

\smallskip
\noindent\textbf{Numerical results.}
We adopt the following model parameters:
\begin{equation}
    T = 1.0,\ a = 0.1,\ \sigma = 0.5,\ q = 0.5,\ \varepsilon = 1.0,\ c = 1.0,\ \mu_0 = \mathcal{N}(1,1).
\end{equation}
The evaluation metrics are reported as follows:
\begin{equation}
    \mathrm{REV} = 0.048\%,\ \mathrm{RMSE}_X = 0.15\%,\ \mathrm{RMSE}_\alpha = 2.52\%,\ \mathrm{RMSE}_M = 0.24\%.
\end{equation}
These results indicate accurate approximation accuracy of MFAC. The overall training takes 18 minutes.

Figures~\ref{fig:SR_Grad_Control}--\ref{fig:SR_Density_V0} compare baseline and MFAC approximations of value gradients, controls, and population measures. The cyan histograms closely follow the baseline densities, demonstrating that the MFAC flow accurately recovers the equilibrium distribution. Within the support of these distributions, MFAC approximations track the baseline solutions well, showing the representational power of the actor and critic networks.

\begin{figure}[!ht]
    \centering
    \includegraphics[width=1.0\linewidth]{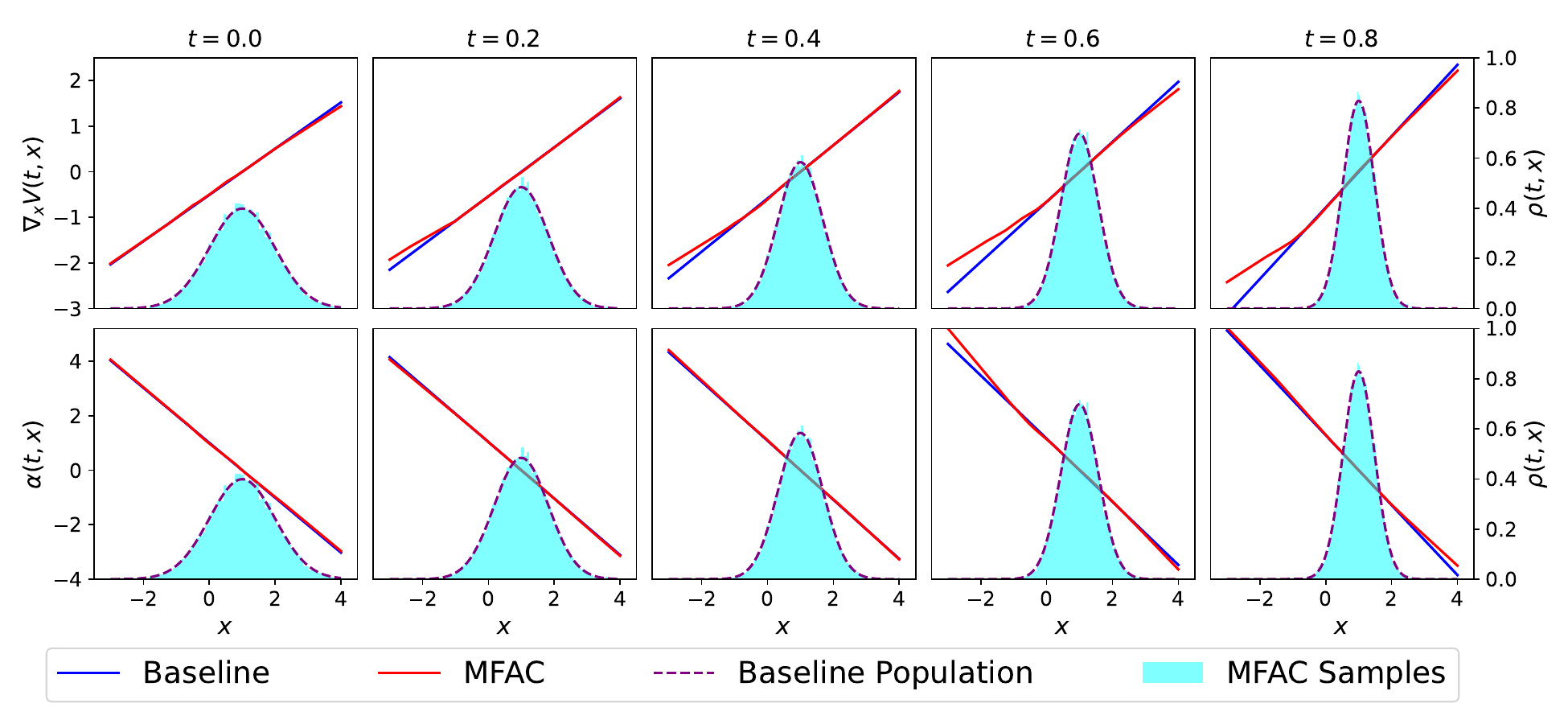}
    \caption{Comparisons of value function gradients (top), equilibrium controls (bottom), and population measures in the systemic risk model (cf. Section~\ref{sec:SR}).
    Five time snapshots are shown. Blue solid lines: baseline solutions; red solid lines: MFAC approximations; purple dashed lines: baseline densities; cyan histograms: empirical distributions from \(5000\) sample paths of \(\check{X}^m_t\).}
    \label{fig:SR_Grad_Control}
\end{figure}

The left panel of Figure~\ref{fig:SR_Density_V0} shows the evolution of empirical densities \(\tilde{\mu}_t\), reconstructed via kernel density estimation from LMC samples generated using trained score networks.
These curves closely match the baseline densities, demonstrating the effectiveness of score matching: even when mean-field interactions depend only on the mean, the score network captures full distributional features.

\begin{figure}[!ht]
    \centering
    \begin{minipage}[b]{0.4\textwidth}
        \centering
        \includegraphics[trim=50 0 10 30, clip,width=\linewidth]{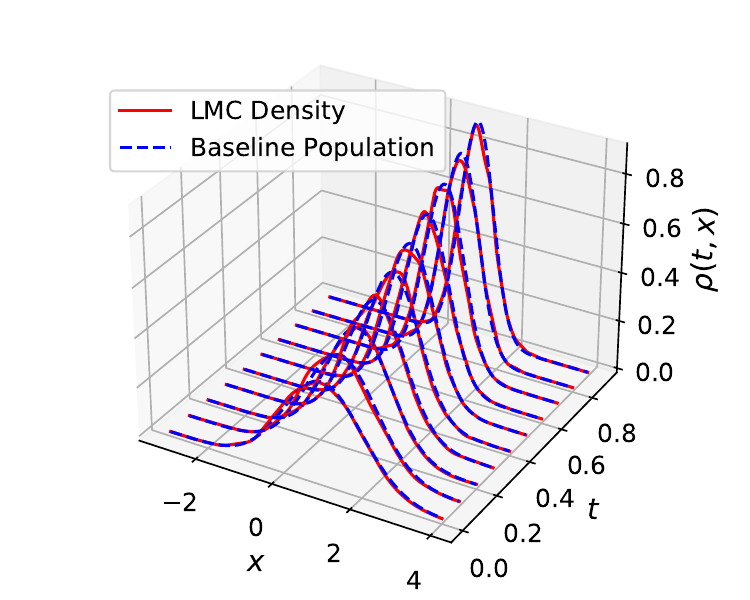}
    \end{minipage}
    \begin{minipage}[b]{0.4\textwidth}
        \centering
        \includegraphics[trim=0 30 0 0, clip,width=\linewidth]{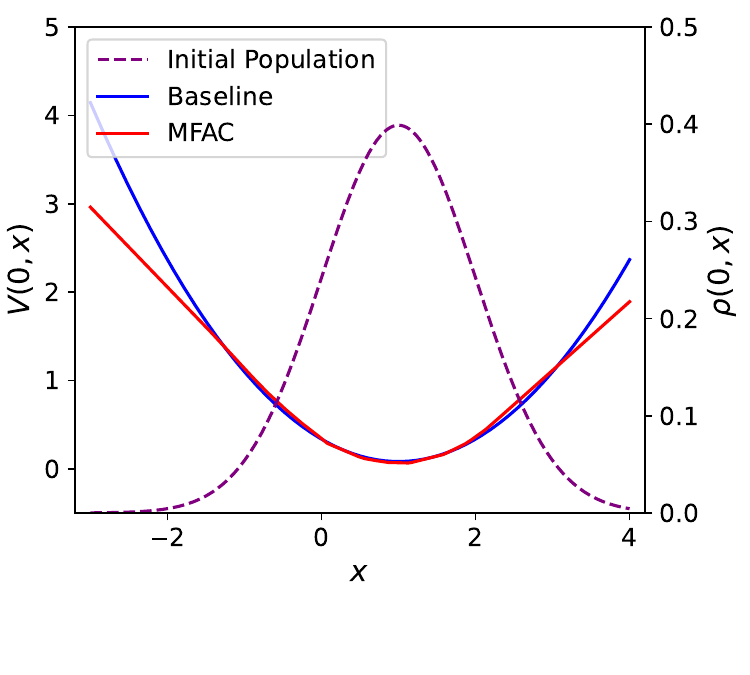}
    \end{minipage}    
    \caption{Equilibrium population measures (left) and initial value functions (right) in the systemic risk model (cf. Section~\ref{sec:SR}). Left: blue dashed lines denote baseline densities; red solid lines show kernel density estimations of \(\tilde{\mu}_t\), computed from \(5000\) LMC samples. Right: blue solid lines show the baseline value function; red solid lines show the MFAC approximation; purple dashed lines plot the initial density $\rho_0$.}
    \label{fig:SR_Density_V0}
\end{figure}

To better understand the impact of \(\beta_\mu\), we conduct additional experiments with fixed model and training parameters, setting \(\Delta\tau = 0.5\), and varying \(\beta_\mu\) across six values in \([0,2]\). As shown in Figure~\ref{fig:SR_beta_error}, very small \(\beta_\mu\) (e.g., near zero) significantly downgrades the performance, while values above \(0.5\) achieve similar convergence. We therefore set \(\beta_\mu = 1.5\) throughout. 

\begin{figure}[!ht]
    \centering
    \includegraphics[width=0.8\linewidth]{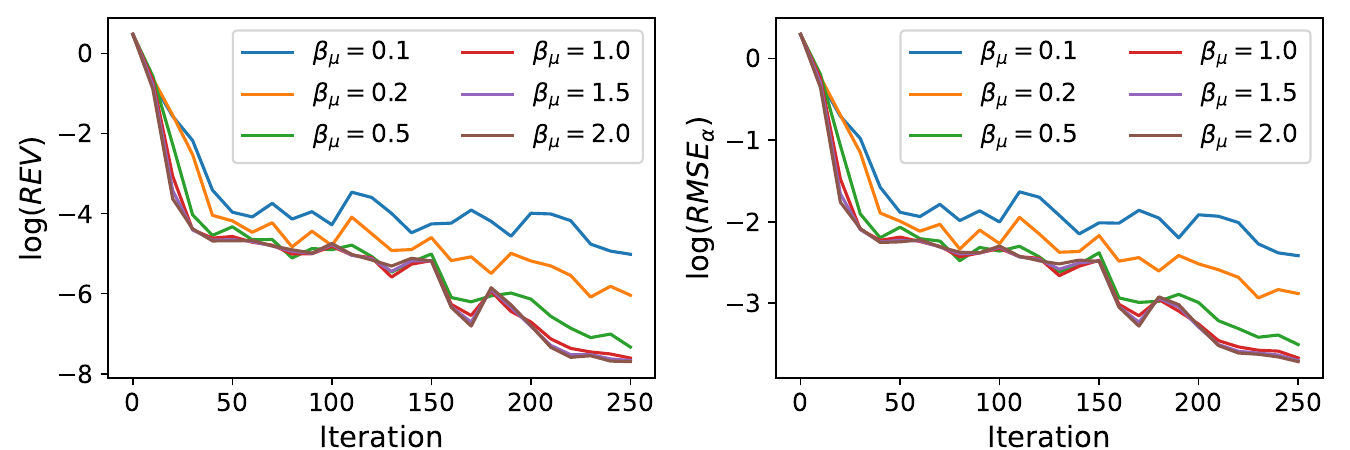}
    \caption{Log-error curves in the systemic risk model (cf. Section~\ref{sec:SR}) across different \(\beta_\mu\).  Errors are recorded every \(10\) iterations.}
    \label{fig:SR_beta_error}
\end{figure}

Figure~\ref{fig:SR_lyapunov} plots the evolution of Lyapunov functions \(\mL_a^\tau\)~\eqref{eq:Lyapunov_actor}, $\mL_c^\tau$~\eqref{eq:Lyapunov_critic} and $\frac{1}{2}\mW_2(\mu^\tau,\rho\mat)^2$ (cf. Definition~\ref{defn:flow_W_metric}) over the training time \(\tau\). During early iterations, the logarithmic values are roughly straight lines, demonstrating exponential rates of convergence and providing numerical evidence for the convergence guarantees of MFAC established in Section~\ref{sec:convergence}.

\begin{figure}[!ht]
    \centering
    \includegraphics[width=0.5\linewidth]{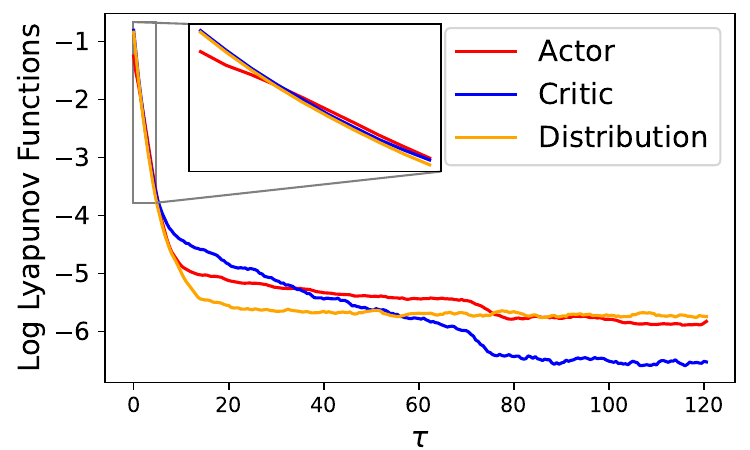}
    \caption{Evolution of Lyapunov functions in the systemic risk model (cf. Section~\ref{sec:SR}). Red: actor~\eqref{eq:Lyapunov_actor}; blue: critic~\eqref{eq:Lyapunov_critic}; orange: distribution term $\frac{1}{2}\mW_2(\mu^\tau,\rho\mat)$. Values are averaged over 10 independent runs and smoothed with a moving average (window size 10).}
    \label{fig:SR_lyapunov}
\end{figure}

\subsection{Optimal execution}\label{sec:Trader}
An important variant of MFGs is the \emph{extended MFG}, where the distribution dependence lies on the action space rather than the state space. Although MFAC is presented in the standard setting, it naturally extends to this formulation with minimal modifications.

We consider a high-frequency trading game of optimal execution with a large population of symmetric traders \cite{angiulia2023reinforcement}.
Each trader controls its trading rate on the market to balance trading execution cost, inventory risk, and price impact.  The interaction is through the mean of trading rates, thus a extended MFG. Here, We consider the one-dimensional case \(d = n = n' = 1\).


\smallskip
\noindent\textbf{Model setup.}
The inventory $X_t$ of a representative trader evolves as:
\begin{equation}\label{eqn:Xt-ex3}
    \ud X_t = \alpha_t \ud t + \sigma \ud W_t,\ X_0\sim \mu_0.
\end{equation}
The trader aims to liquidate its position $X_0$ while minimizing the associated cost:
\begin{equation}
    f(t, x, \mu, \alpha) = \tfrac{1}{2} c_\alpha \alpha^2 + \tfrac{1}{2} c_X x^2 - \gamma x \BAR{\mu},\quad g(x,\mu) = \tfrac{1}{2} c_g x^2,
\end{equation}
where \(\mu\) a measure on the action space \(\R^n\). We assume \(c_\alpha,c_X,\gamma,c_g,\sigma>0\). Derivations of the baseline equilibrium are provided in Appendix~\ref{app:baseline_Trader}.

\smallskip
\noindent\textbf{Numerical results.} The following model parameters are used:
\begin{equation}
    T = 1.0,\ a = 0.1,\ \sigma = 0.5,\  c_\alpha = 0.5,\ c_X = 1.0,\ c_g = 1.0,\ \gamma = 1.0,\ \mu_0 = \mathcal{N}(1,1).
\end{equation}
Evaluation metrics are reported below:
\begin{equation}
    \mathrm{REV} = 1.50\%,\ \mathrm{RMSE}_X = 2.57\%,\ \mathrm{RMSE}_\alpha = 3.70\%,\ \mathrm{RMSE}_M = 4.31\%,
\end{equation}
demonstrating the strong approximation performance of MFAC in the extended MFG setting. Total training time is approximately 20 minutes.

\begin{figure}[!ht]
    \centering
    \includegraphics[width=1.0\linewidth]{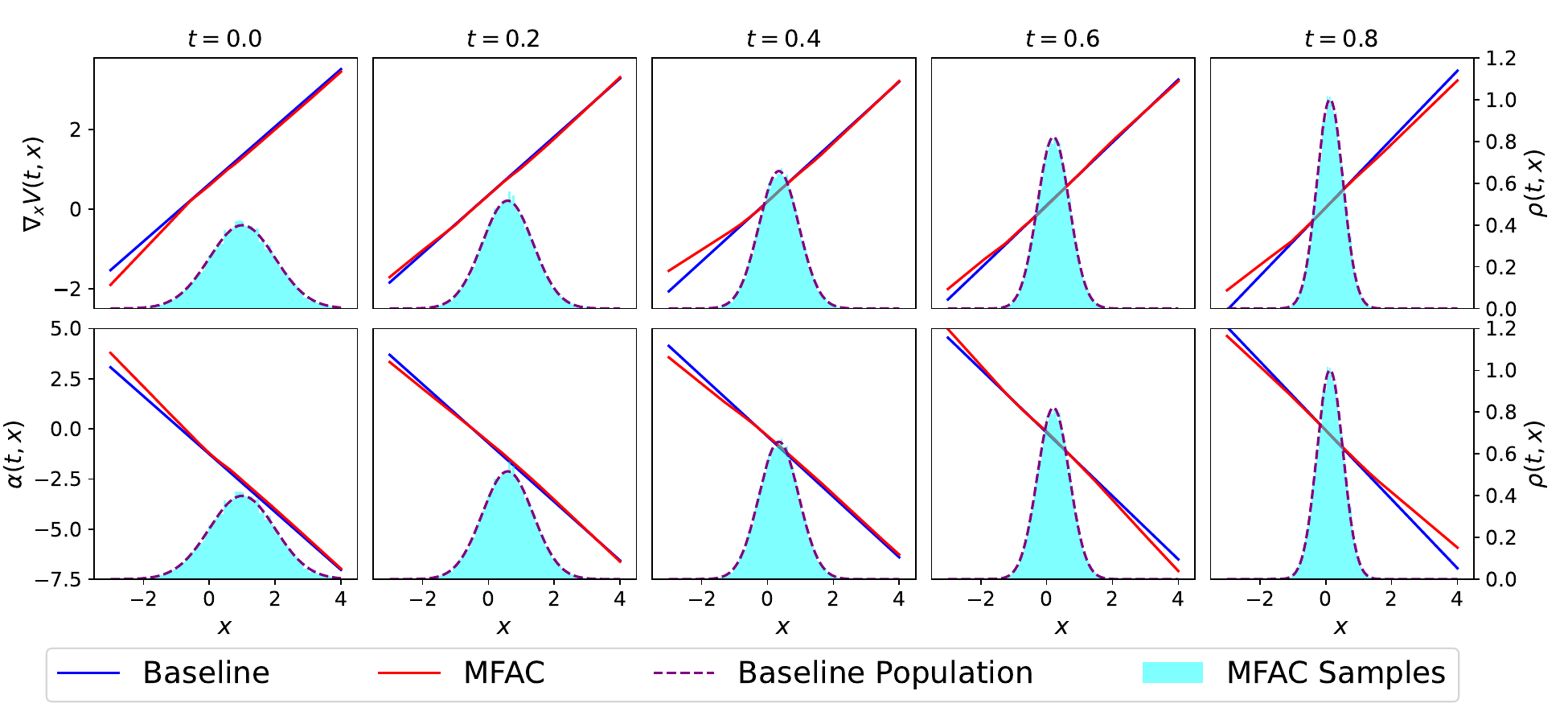}
    \caption{Comparisons of value function gradients (top), equilibrium controls (bottom), and population measures in the optimal execution problem (cf. Section~\ref{sec:Trader}). 
    Five time snapshots are shown. Blue solid lines: baseline solutions; red solid lines: MFAC approximations; purple dashed lines: baseline densities of control; cyan histograms: empirical distributions from 5000 sample paths of \(\check{X}^m_t\).}
    \label{fig:Trader_Grad_u}
\end{figure}

\begin{figure}[!ht]
    \centering
    \begin{minipage}{0.4\textwidth}
        \centering
        \includegraphics[trim=50 0 10 30, clip,width=\linewidth]{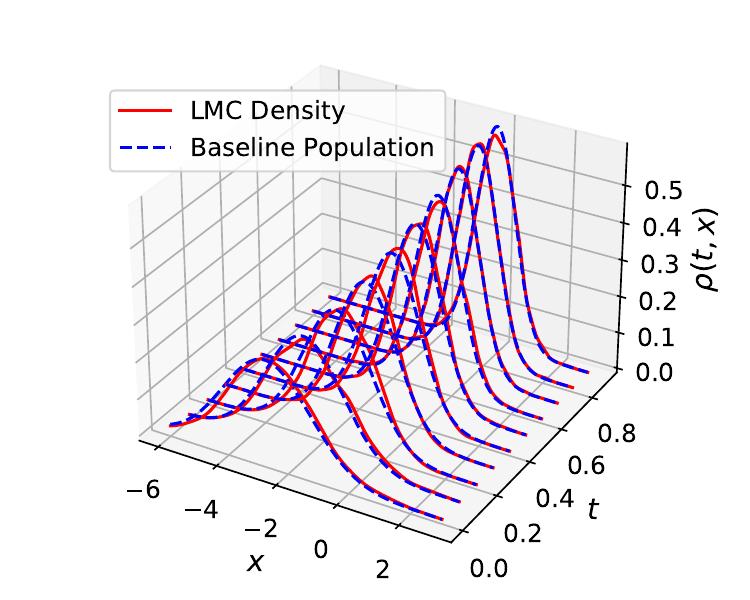}
    \end{minipage}
    \begin{minipage}{0.4\textwidth}
        \centering
        \includegraphics[trim=0 30 0 0, clip,width=\linewidth]{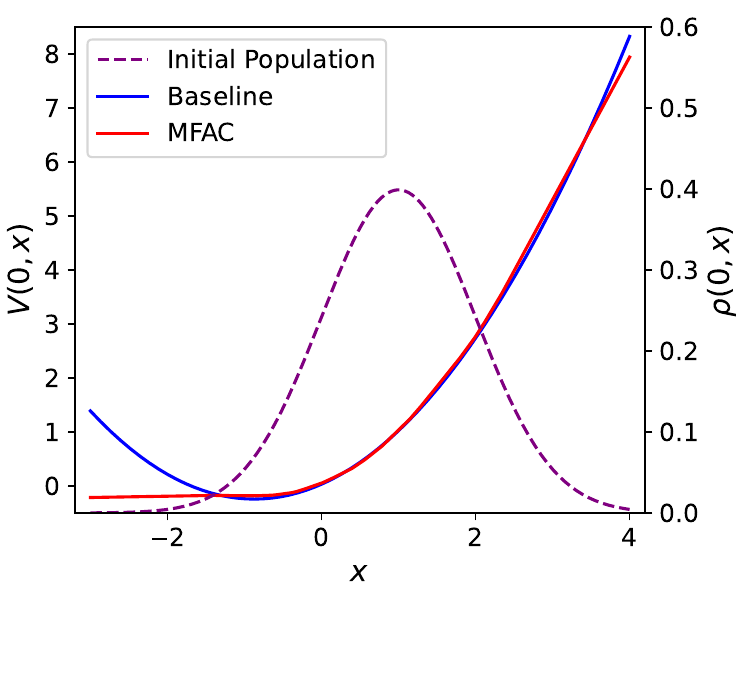}
    \end{minipage}
    \caption{Equilibrium measures of the control (left) and initial value functions (right) in the optimal execution problem (cf. Section~\ref{sec:Trader}). Left: blue dashed lines represent baseline densities; red solid lines stand for kernel density estimations of \(\tilde{\mu}_t\), computed from \(5000\) LMC samples. Right: blue solid lines show the baseline value function; red solid lines show the MFAC approximation; purple dashed lines plot the initial density $\rho_0$.}
    \label{fig:Trader_Density_V0}
\end{figure}

Figures~\ref{fig:Trader_Grad_u}--\ref{fig:Trader_Density_V0} compare the baseline and MFAC approximations of controls, value gradients, and distribution of controls. The left panel of Figure~\ref{fig:Trader_Density_V0} shows the evolution of \(\tilde{\mu}_t\), obtained from LMC samples with trained score networks. Results are qualitatively consistent with Section~\ref{sec:SR}, confirming that MFAC generalizes well to extended MFGs.

\subsection{Cucker--Smale flocking model}\label{sec:Flocking}

We consider a mean-field game modeling bird flocking behavior in three dimensions \cite{carmona2018probabilistic}, where each agent (bird) controls its acceleration to stay with the flock while minimizing energy expenditure.
We consider the multi-dimensional case, i.e., \(d = 6,\ n = n' = 3\).

\smallskip
\noindent\textbf{Model setup.}
The state variable $x = (s,v) \in \RR^6$ of a representative agent consists of  position \(S_t\) and velocity \(V_t\), evolving according to:
\begin{equation}
    \ud S_t = V_t\ud t,\quad 
    \ud V_t = \alpha_t\ud t + C\ud W_t
    ,\ (S_0,V_0)\sim \mu_0,
\end{equation}
where \(C\in\mathbb{R}^{3\times 3}\) is a constant matrix.
Each individual aims to minimize its expected cost~\eqref{eq:J}, with running and terminal costs given by:
\begin{equation}
    f(t,x,\mu,\alpha) = \NORM{\alpha}_R^2 + \Big\lVert\int_{\R^3\times \R^3}w(|s-s'|)\,(v'-v)\ud \mu(s',v')\Big\rVert_Q^2,\quad  g\equiv 0.
\end{equation}
Here \(R,Q\in\mathbb{S}^{3\times 3}\) are positive semi-definite, and the weight function is defined as \(w:\R^3\ni s\to {(1 + |s|^2)^{-\beta}}\in\R_+\), for some \(\beta\geq 0\).  \(\NORM{x}_A^2 := x\transpose Ax\) denotes the vector norm induced by a positive semi-definite matrix \(A\). 

Unlike the systemic risk and optimal execution models, the flocking game admits no semi-explicit solution for \(\beta>0\), and its mean-field interactions are through the entire distribution. We adopt the results in \cite{han2024learning} as the baseline for comparison.

\smallskip
\noindent\textbf{Numerical results.}
We set the model parameters as follows:
\begin{equation}
    T = 1.0,\quad C = 0.1I_3,\quad R = 0.5I_3,\quad Q = I_3,\quad \beta = 0.2,\quad \mu_0 = \mathcal{N}(\mathbf{0}_3,I_3)\otimes \mathcal{N}(\mathbf{1}_3,I_3),
\end{equation}
where \(\mathbf{0}_d\in\R^d\) (resp. \(\mathbf{1}_d\)) denotes $d$-dimensional zero (resp. one) vector, and \(\otimes\) denotes the measure product. The overall training procedure takes \(3\) hours.

\begin{figure}[!ht]
    \centering
    \includegraphics[width=1.0\linewidth]{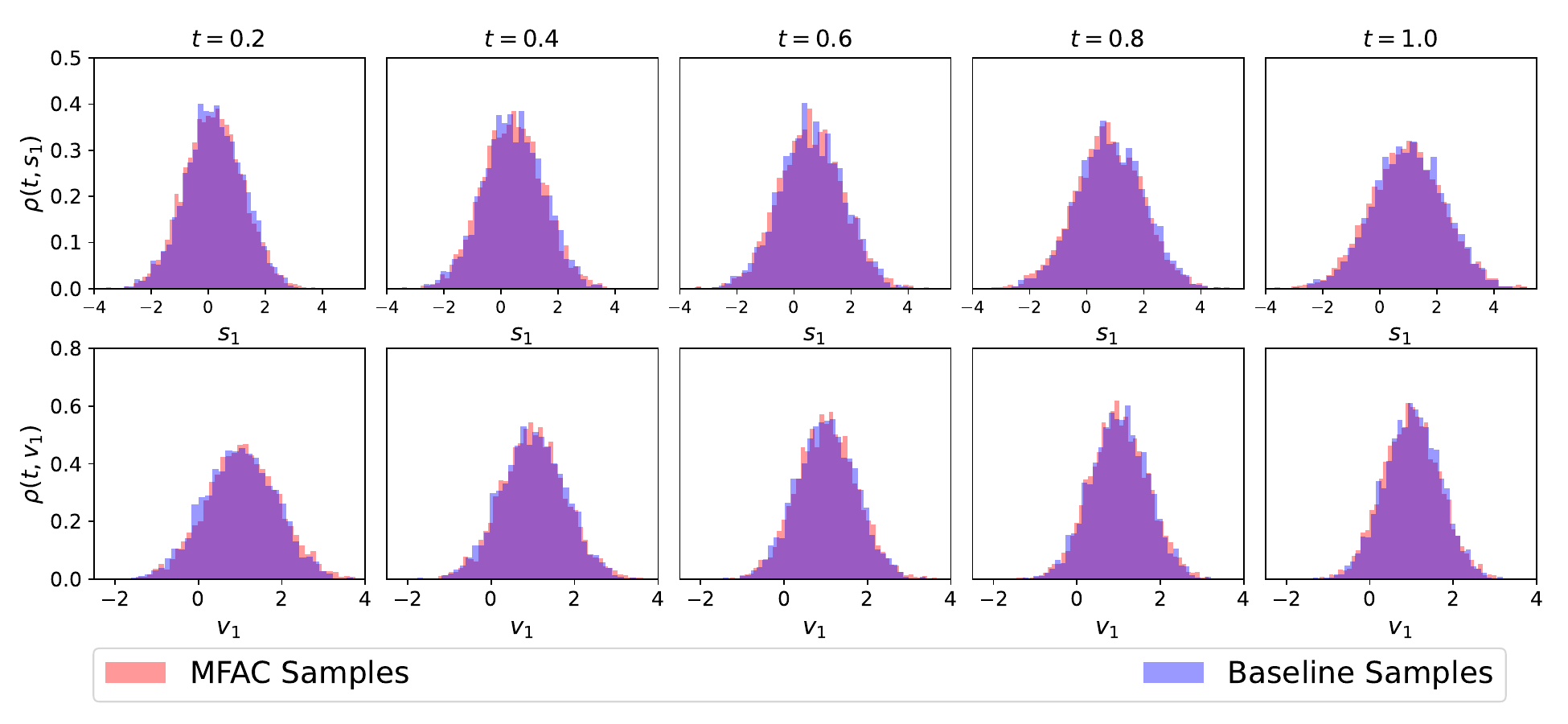}
    \caption{Comparisons of equilibrium measures in the flocking model.
    Five time snapshots are shown.
    Blue histograms represent the baseline solution from \cite{han2024learning}; red histograms are plotted based on \(5000\) sample paths of \(\check{X}^m_t\). For clarity, only the first component of equilibrium position and velocity is shown.}
    \label{fig:Flocking_Hist_0.2}
\end{figure}

\begin{figure}[!ht]
    \centering
    \includegraphics[width=0.8\linewidth]{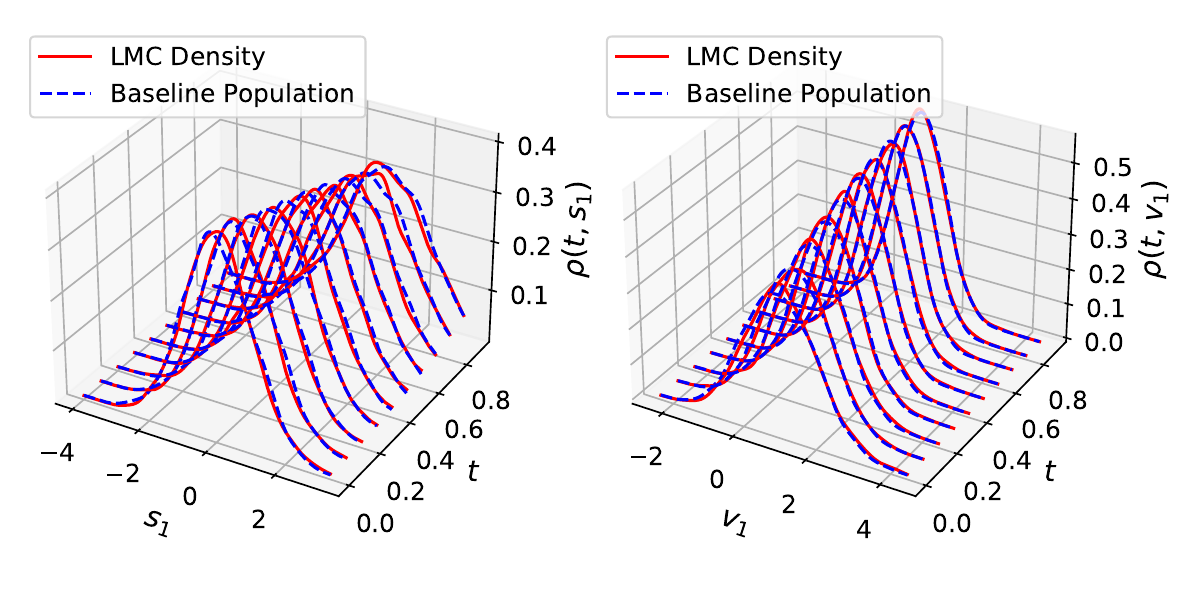}
    \caption{Comparisons of equilibrium density in the flocking model.
    Blue dashed lines show the baseline from \cite{han2024learning}; red solid lines show kernel density estimations of \(\tilde{\mu}_t\), computed from \(5000\) LMC samples.
    The first components of equilibrium position and velocity are shown.}
    \label{fig:Flocking_Density_0.2}
\end{figure}

Figures~\ref{fig:Flocking_Hist_0.2}--\ref{fig:Flocking_Density_0.2}  compare baseline and MFAC results. Figure~\ref{fig:Flocking_Hist_0.2} (resp. Figure~\ref{fig:Flocking_Density_0.2}) evaluates the actor networks (resp. score networks) by simulating sample paths of $\check X_t^m$ (resp. kernel density estimates of $\tilde \mu_t$.
The conclusions are qualitatively the same as those presented in Sections~\ref{sec:SR}-\ref{sec:Trader}. These results demonstrate the robustness of MFAC in handling high-dimensional MFGs with nontrivial distributional dependencies. For further experiments with varying $\beta$, see Appendix~\ref{app:numerics}.

\section{Conclusion}\label{sec:conclusion}

In this work, we proposed the Mean-Field Actor–Critic (MFAC) flow, a continuous-time learning framework for solving MFGs by combining policy gradient methods, value-based updates, and OTGP flow. Theoretically, we established the exponential convergence for MFAC using Lyapunov functionals under suitable timescale conditions. On the computational side, we discretized MFAC into a practical deep reinforcement learning algorithm, using neural network parameterizations and score matching techniques. Numerical experiments on systemic risk, optimal execution, and flocking models confirmed the effectiveness of the framework. Overall, MFAC offers both theoretical insights and practical algorithms for learning equilibria in MFGs. Future directions include extending the framework to MFGs with common noises and to mean-field control problems, relaxing the technical assumptions, and exploring scalable implementations for large-scale multi-agent systems.

\section*{Acknowledgment}

The authors thank Alan Raydan for sharing the code associated with \cite{angiuli2023deep}, which informed parts of our implementation.
The authors thank Jiequn Han for making available the code from \cite{han2024learning}; parts of this code were adapted for use in our experiments. The authors thank Daniel Lacker for the discussion on MFGs.

M. Zhou is supported by NSF 2208272 and AFOSR YIP award No. FA9550-23-1-0087. R. Hu acknowledges partial support from the ONR grant under \#N00014-24-1-2432, the Simons Foundation (MP-TSM-00002783), and the NSF grant DMS-2420988.

\bibliographystyle{plain}
\bibliography{ref}

\clearpage
\appendix
\section*{Appendix}
 
We provide technical lemmas and proofs used throughout the paper in the Appendix. 
Without specification, $C$ and $c$ denote generic positive constants depending only on $d$, $K$, $T$, $\sigma_0$, $\lam_H$, but independent of the speeds $\beta_a$, $\beta_c$, $\beta_\mu$ and the selections of \(\tau\geq0, t\in[0,T],\ x\in\R^d,\ \alpha\in\mA,\ \mu\in\mM\). Their values may vary from line to line. $C$ is potentially large while $c$ is potentially small.

\section{Lemmas}\label{sec:lemmas}
This section presents several auxiliary lemmas for the main results. Unless otherwise stated, we assume Assumptions \ref{assu:mu0_sigma0}, \ref{assu:basic}, and \ref{assu:Hconcave} hold.

\subsection{Stochastic Gr\"onwall's inequalities}\label{sec:Gronwall}

We present several versions of stochastic Gr\"onwall's inequalities. Since the proof strategies are identical across cases, we state the results collectively and present a unified proof.

\begin{lem}\label{lem:Gronwall_square}
For $\alpha \in \mA$, $\mu \in \mM$, let $x_t:= X^{\mu,\alpha}_t$ be the state process under $(\mu,\alpha)$. Then
\begin{equation}\label{eq:Gronwall_square}
\EE\big[|x_t|^2 ~|~ x_0\big] \le C (1 + |x_0|^2),\ \forall t\in[0,T].
\end{equation}

For $\alpha' \in \mA$, $\mu' \in \mM$, let $x'_t:= X^{\mu',\alpha'}_t$ be the state process under $(\mu',\alpha')$ driven by the same Brownian motion as $x_t$. 

If $\alpha = \alpha'$,
\begin{equation}\label{eq:Gronwall_diff_square}
\EE\big[|x_t - x'_t|^2 ~|~ \mF_0\big] \le C (|x_0 - x'_0|^2 + \mW_2(\mu,\mu')^2),\ \forall t\in[0,T].
\end{equation}

If $\alpha = \alpha'$ and $\mu=\mu'$, 
\begin{equation}\label{eq:Gronwall_xsq_diffxsq100}
\EE\Big[\parentheses{1+|x_t|^2 + |x'_t|^2} \abs{x_t-x'_t}^2 ~\Big|~ \mF_0\Big] \le C  (1+|x_0|^2+|x_0'|^2)|x_0-x_0'|^2,\ \forall t\in[0,T].
\end{equation}

If $\alpha - \alpha' \in \mC$ and $x_0 \overset{\mathrm{a.s.}}{=} x'_0 \sim \rho_0$,
\begin{equation}\label{eq:Gronwall_xsq_diffxsq011}
\EE\Big[\parentheses{1+|x_t|^2 + |x'_t|^2} \abs{x_t-x'_t}^2 \Big] \le C \parentheses{\mW_2(\mu,\mu')^2 + \norm{\alpha - \alpha'}^2_\mao},\ \forall t\in[0,T].
\end{equation}
\end{lem}

\begin{cor}\label{cor:Gronwall_Wasserstein}
For any $\alpha,\alpha' \in \mA$ and $\mu, \mu' \in \mM$,
\begin{equation}\label{eq:Gronwall_Wasserstein}
W_2(\rho\ma_t,\rho\map_t)^2,\,\mW_2(\rho\ma,\rho\map)^2, \, d_\beta(\rho\ma,\rho\map)^2  \le C \parentheses{ \norm{\alpha-\alpha'}^2_\mao + \mW_2(\mu,\mu')^2},\ \forall t\in[0,T].
\end{equation}
\end{cor}

\begin{lem}\label{lem:Gronwall_second_order}
For any $\alpha \in \mA$, $\mu \in \mM$, let $x^\pm_t$ and $x_t$ be three state processes under $(\mu,\alpha)$ driven by the same Brownian motion with different initial conditions $x_0^\pm$ and $x_0$, then
\begin{align}\label{eq:Gronwall_xsq_diffx4}
&\EE\sqbra{\parentheses{1+|x^+_t|^2 + |x^-_t|^2} \abs{x^+_t-x^-_t}^4  ~\Big|~ \mF_0} \le C  (1+|x_0^+|^2+|x_0^-|^2)|x_0^+-x_0^-|^4,\\
&  \EE\sqbra{\parentheses{1+|x^+_t + x^-_t|^2 + |x_t^2|} \abs{x^+_t+x^-_t-2x_t}^2 ~\Big|~ \mF_0}  \label{eq:Gronwall_xsq_x3terms}
\\
& \qquad \qquad \qquad \le C \,  \parentheses{1+|x_0^++x_0^-|^2 + |x_0|^2} \parentheses{ |x_0^+ + x_0^- - 2x_0|^2 + |x_0^+ - x_0^-|^4 },\ \forall t\in[0,T].
\end{align}
\end{lem}

\begin{lem}\label{lem:Gronwall_4terms}
For any $\alpha \in \mA$, $\mu,\mu' \in \mM$, let $x^1_t,x^2_t$ be two state processes under $(\mu,\alpha)$, and $x^{1\prime}_t,x^{2\prime}_t$ be two state processes under $(\mu',\alpha)$, with all four processes driven by the same Brownian motion.
The initial conditions satisfy $x^1_0 \overset{\mathrm{a.s.}}{=} x^{1\prime}_0$, $x^2_0 \overset{\mathrm{a.s.}}{=} x^{2\prime}_0$. 
Then, 
\begin{equation}\label{eq:Gronwall_4terms}
\EE\Big[\big| (x^2_t-x^1_t) - (x^{2\prime}_t - x^{1\prime}_t)\big|^2 ~\big|~ \mF_0\Big] \le C |x^2_0-x^1_0|^2 \,\mW_2(\mu,\mu')^2,\ \forall t\in[0,T].
\end{equation}
\end{lem}

\begin{rem}[Extension to general initial times]\label{rem:Gronwall}
All results above remain valid when the state dynamics are initialized at an intermediate time $s \in [0,T]$ instead of time $0$.  For example, we can extend \eqref{eq:Gronwall_square} to
$$\EE\big[|x_t|^2 ~|~ x_s\big] \le C (1 + |x_s|^2),\ \forall 0\leq s\leq t\leq T.$$
The proofs for all the results (including general initial times) follow the same strategy:
\begin{enumerate}
\item Apply It\^o's lemma to the quantity of interest and take expectations on both sides;
\item Bound the expectation of the drift term using given conditions;
\item Apply the classical Gr\"onwall's inequality.
\end{enumerate}
\end{rem}

\begin{proof}
We prove all the lemmas in this section by following the three-step strategy outlined in Remark~\ref{rem:Gronwall}. 
As an illustration, we show \eqref{eq:Gronwall_square} in full details.

Let $b_t:=b(t,x_t,\mu_t,\alpha(t,x_t))$ and $\sigma_t := \sigma(t,x_t,\mu_t)$, so that $\rd x_t = b_t \,\rd t + \sigma_t \,\rd W_t$. 
Applying It\^o's lemma to $|x_t|^2$ yields
$\rd |x_t|^2 = [2x_t\tp b_t + |\sigma_t|^2] \,\rd t + 2 x_t\tp \sigma_t \,\rd W_t.$
By \cite[Theorem~5.2.1]{oksendal2003stochastic}, \(\EE \int_0^T |x_t|^2\ud t<\infty\), which justifies \(\int_0^t x_s\tp \sigma_s \,\rd W_s\) being a martingale, due to the boundedness of \(\sigma_t\).
Integrating both sides and taking expectations conditional on $x_0$ yield
$\pt \EE\sqbra{|x_t|^2 ~|~ x_0} = \EE\sqbra{2x_t\tp b_t + |\sigma_t|^2 ~|~ x_0}.$

In the second step, we bound this expectation. By Assumption~\ref{assu:basic}, $|\sigma_t| \le K$ and
$$|b_t| \le K(1 + |x_t| + W_2(\mu_t,\delta_0) + |\alpha(t,x_t)|) \le C(1+|x_t|),$$
where the second inequality follows from $\mu \in \mM$ and $\alpha \in \mA$. Therefore, we obtain  
$$\pt \EE\sqbra{|x_t|^2 ~|~ x_0} \le \EE\sqbra{2|x_t| C(1+|x_t|) + K^2 ~|~ x_0} \le C \parentheses{1+\EE\sqbra{|x_t|^2 ~|~ x_0}}.$$

In the third step, applying Gr\"onwall's inequality to $\EE[|x_t|^2~|~x_0]$ concludes the proof of \eqref{eq:Gronwall_square}.

\medskip

\textbf{For the proofs of the remaining results, we only outline key steps below, with the key difference lying in bounding drift and diffusion terms in Step 2 of Remark~\ref{rem:Gronwall}.}

Let  $b'_t := b(t,x'_t,\mu'_t,\alpha'(t,x'_t))$ and $\sigma_t' := \sigma(t,x'_t,\mu'_t)$, so that $\rd x'_t = b'_t \,\rd t + \sigma'_t \,\rd W_t$.
For the function $b\ma(t,x):= b(t,x,\mu_t,\alpha(t,x))$, by Assumption~\ref{assu:basic} and $\alpha \in \mA$,
\begin{equation}\label{eq:bma_Lip}
|\nx b\ma(t,x)| \le |\nx b| + |\na b|\, |\nx \alpha| \le K + K^2,
\end{equation}
implying that $b\ma$ is Lipschitz in $x$.

For \eqref{eq:Gronwall_diff_square}, since $\alpha=\alpha'$, applying the Lipschitz property of $b\ma$ (cf. \eqref{eq:bma_Lip}) and $\sigma$ yields
$$|b_t - b'_t|,\,|\sigma_t - \sigma'_t| \le (K+K^2) (|x_t-x'_t| + W_2(\mu_t, \mu'_t)).$$

For \eqref{eq:Gronwall_xsq_diffxsq100}, since $\alpha = \alpha'$ and $\mu=\mu'$,
$|b_t - b'_t|,\,|\sigma_t - \sigma'_t| \le (K+K^2) |x_t-x'_t|.$

For \eqref{eq:Gronwall_xsq_diffxsq011}, using $|\alpha'(t,x_t')| \le |\alpha'(t,x_t)| + K|x_t-x'_t|$, we obtain
\begin{align*}
&\pt \EE \sqbra{|x_t|^2 \, |x_t - x'_t|^2 } \le C \, \EE\sqbra{ (1+|x_t|^2) \, |x_t - x'_t|^2 + W_2(\mu_t,\mu'_t)^2 + |\alpha(t,x_t) - \alpha'(t,x_t)|^2}, \\
&\pt \EE \sqbra{|x_t'|^2 \, |x_t - x'_t|^2 } \le C \, \EE\sqbra{ (1+|x_t'|^2) \, |x_t - x'_t|^2 + W_2(\mu_t,\mu'_t)^2 + |\alpha(t,x_t) - \alpha'(t,x_t)|^2}.
\end{align*}

For \eqref{eq:Gronwall_Wasserstein}, the proof is based on
\begin{equation}
W_2(\rho\ma_t,\rho\map_t)^2 \le \EE\big[\abs{x_t-x'_t}^2 \big] \le C \parentheses{\mW_2(\mu,\mu')^2 + \norm{\alpha - \alpha'}^2_\mao},
\end{equation}
where the first inequality follows from the synchronous coupling $x_t \sim \rho\ma_t$, $x'_t \sim \rho\map_t$, and the second inequality is a modified version of \eqref{eq:Gronwall_xsq_diffxsq011} (without the moment term $|x_t|^2$, hence not requiring $\alpha-\alpha' \in \mC$). 
Bounds for $\mW_2(\rho\ma,\rho\map)^2$ and $d_\beta(\rho\ma,\rho\map)^2$ directly follow.

For \eqref{eq:Gronwall_xsq_diffx4} and \eqref{eq:Gronwall_xsq_x3terms}, we apply the mean value theorem in Step 2. Setting $b^\pm_t := b(t,x^\pm_t,\mu_t,\alpha(t,x^\pm_t))$, $\sigma^\pm_t := \sigma(t,x^\pm_t,\mu_t)$, so that $\rd x_t^\pm = b^\pm_t \,\rd t + \sigma_t^\pm \rd W_t$, we get
$$\abs{b^+_t + b^-_t - 2b_t}, \abs{\sigma^+_t + \sigma^-_t - 2\sigma_t} \le C \parentheses{|x^+_t - x^-_t|^2 + |x^+_t + x^-_t - 2 x_t|}.$$
See \eqref{eq:f_3terms_bound} for a similar derivation of this inequality.

For \eqref{eq:Gronwall_4terms}, note that
$$\pt \EE\Big[\big| (x^2_t-x^1_t)\tp(x^1_t - x^{1\prime}_t)\big|^2 ~\big|~ \mF_0\Big] \le C \, \EE\Big[\big| (x^2_t-x^1_t)\tp(x^1_t - x^{1\prime}_t)\big|^2 ~\big|~ \mF_0\Big] + C |x^2_0-x^1_0|^2 W_2(\mu_t,\mu_t')^2,$$
which implies 
$\EE\Big[\big| (x^2_t-x^1_t)\tp(x^1_t - x^{1\prime}_t)\big|^2 ~\big|~ \mF_0\Big] \le C |x^2_0-x^1_0|^2 \,\mW_2(\mu,\mu')^2.$
Based on this, we show that
\begin{align*}
& \quad \pt \EE\Big[\big| (x^2_t-x^1_t) - (x^{2\prime}_t - x^{1\prime}_t)\big|^2 ~\big|~ \mF_0\Big] \\
& \le C \, \EE\Big[\big| (x^2_t-x^1_t)-(x^{2\prime}_t - x^{1\prime}_t) \big|^2 + |x^2_t-x^1_t|^2 \parentheses{|x^1_t - x^{1\prime}_t|^2 + |x^2_t - x^{2\prime}_t|^2 + W_2(\mu_t,\mu_t')^2} ~\big|~ \mF_0\Big] \\
& \le C \, \EE\Big[\big| (x^2_t-x^1_t)-(x^{2\prime}_t - x^{1\prime}_t) \big|^2 ~\big|~ \mF_0\Big]  + C \, |x^2_0-x^1_0|^2 W_2(\mu_t,\mu_t')^2,
\end{align*}
where the first inequality is based on the mean value theorem.

This concludes the proofs of all the lemmas.
\end{proof}

\subsection{Performance difference lemma}
The performance difference lemma \cite[Section~4.1]{kakade2002approximately} is a fundamental result in reinforcement learning (RL), as it quantitatively relates the performance gap between two policies. In the context of stochastic control and mean-field games (MFGs), an analogous performance difference lemma also holds. It provides a rigorous way to compare the value functions associated with different controls or policies, which forms the basis for the convergence guarantees.

\begin{lem}[Performance difference]\label{lem:performance_difference}
Let $\alpha, \alpha' \in \mA$ and $\mu, \mu' \in \mM$. Let $x_t = X\ma_t$ be the state process under $(\mu, \alpha)$. Then 
\begin{equation}\label{eq:performance_difference}
\begin{aligned}
V\ma(0,x_0) - V\map(0,x_0) = \EE \Big[ \int_0^T \Big( H\big(t,x_t,\mu'_t,\alpha'(t,x_t), -\nx V\map(t,x_t), -\nx^2 V\map(t,x_t)\big) \\
 - H\big(t,x_t,\mu_t,\alpha(t,x_t), -\nx V\map(t,x_t), -\nx^2 V\map(t,x_t)\big) \Big) \,\rd t + g(x_T,\mu_T) - g(x_T, \mu'_T) ~\Big|~ x_0 \Big].
\end{aligned}
\end{equation}
\end{lem}

\begin{rem}
Unlike the analysis in stochastic Gr\"onwall's inequalities, only the state process under $(\mu,\alpha)$ appears in \eqref{eq:performance_difference}. 
After taking expectation with respect to $x_0 \sim \rho_0$, the left-hand side of \eqref{eq:performance_difference} becomes $J^{\mu}[\alpha] - J^{\mu'}[\alpha']$.

Additionally, the lemma extends to any initial time $s \in [0,T]$, i.e.,
\begin{equation}\label{eq:performance_difference_general}
\begin{aligned}
V\ma(s,x_s) - V\map(s,x_s) = \EE \Big[ \int_s^T \Big( H\big(t,x_t,\mu'_t,\alpha'(t,x_t), -\nx V\map(t,x_t), -\nx^2 V\map(t,x_t)\big) \\
 - H\big(t,x_t,\mu_t,\alpha(t,x_t), -\nx V\map(t,x_t), -\nx^2 V\map(t,x_t)\big) \Big) \,\rd t + g(x_T,\mu_T) - g(x_T, \mu'_T) ~\Big|~ x_s \Big],
\end{aligned}
\end{equation}
and its proof remains identical.
\end{rem}

\begin{proof}
Define $f_t := f(t,x_t,\mu_t,\alpha(t,x_t))$ and $f'_t := f(t,x_t,\mu'_t,\alpha'(t,x_t))$. Similarly, $b_t := b(t,x_t,\mu_t,\alpha(t,x_t))$, $b'_t := b(t,x_t,\mu'_t,\alpha'(t,x_t))$, $\sigma_t := \sigma(t,x_t,\mu_t)$, $\sigma'_t := \sigma(t,x_t,\mu'_t)$. Denote by $\mL := \mL\ma$ and $\mL':=\mL\map$ the infinitesimal generators associated with \((\mu,\alpha)\) and \((\mu',\alpha')\) respectively. By It\^o's lemma, 
\begin{align*}
& g(x_T,\mu_T) = V\ma(0,x_0) + \int_0^T (\pt + \mL)V\ma(t,x_t) \, \rd t + \int_0^T \nx V\ma(t,x_t)\tp \sigma_t \,\rd W_t, \\
& g(x_T,\mu_T') = V\map(0,x_0) + \int_0^T (\pt + \mL)V\map(t,x_t) \, \rd t + \int_0^T \nx V\map(t,x_t)\tp \sigma_t \,\rd W_t.
\end{align*}
Therefore,
\begin{equation}\label{eq:g_V0}
\EE\sqbra{g(x_T,\mu_T) - V\ma(0,x_0) ~\big|~ x_0} = \EE \Big[ \int_0^T (\pt + \mL)V\ma(t,x_t) \, \rd t  ~\big|~ x_0 \Big] = - \EE \Big[ \int_0^T f_t \, \rd t  ~\big|~ x_0 \Big],
\end{equation}
\begin{equation}\label{eq:g_V0_prime}
\begin{aligned}
& \quad \EE\sqbra{g(x_T,\mu_T') - V\map(0,x_0) ~\big|~ x_0} = \EE \Big[ \int_0^T (\pt + \mL)V\map(t,x_t) \, \rd t  ~\Big|~ x_0 \Big] \\
& =\EE \Big[ \int_0^T \parentheses{(\mL-\mL')V\map(t,x_t) - f'_t} \, \rd t  ~\Big|~ x_0 \Big],
\end{aligned}
\end{equation}
where we used the PDEs \eqref{eq:value_PDE} satisfied by $V\ma$ and $V\map$. Subtracting \eqref{eq:g_V0} from \eqref{eq:g_V0_prime} yields
\begin{align*}
& \quad \EE\sqbra{\big(V\ma(0,x_0) - V\map(0,x_0)\big) -\big(g(x_T,\mu_T) - g(x_T,\mu_T')\big)  ~\big|~ x_0} \\
& = \EE \Big[ \int_0^T \parentheses{(\mL-\mL')V\map(t,x_t) +f_t- f'_t} \, \rd t  ~\Big|~ x_0 \Big] \\
& = \EE \Big[ \int_0^T \Big( H(t,x_t,\mu'_t,\alpha'(t,x_t), -\nx V\map(t,x_t), -\nx^2 V\map(t,x_t))  \\
&  \quad \quad \quad \quad - H(t,x_t,\mu_t,\alpha(t,x_t), -\nx V\map(t,x_t), -\nx^2 V\map(t,x_t)) \Big) \, \rd t  ~\Big|~ x_0 \Big],
\end{align*}
which concludes the proof.
\end{proof}

An important corollary of this lemma is an explicit characterization of the \emph{cost gap}, the difference between the cost under a given control and the optimal cost under the same distribution. Specifically, by taking $\mu'=\mu$ and $\alpha' = \alpha^{\mu,*}$ in Lemma~\ref{lem:performance_difference}, where $\alpha^{\mu,*}$ denotes the optimal control associated with a given measure flow $\mu$, the lemma provides an explicit expression for this cost gap.

\begin{lem}[Cost gap]\label{lem:cost_gap}
For any $\mu \in \mM$ and $\alpha \in \mA$, let $x_t = X\ma_t$ be the state process under $(\mu, \alpha)$, and denote $\alpha_s := \alpha(s,x_s)$, $\alpha^*_s := \alpha^{\mu,*}(s,x_s)$. Then, 
\begin{equation}\label{eq:cost_gap}
\begin{aligned}
J^\mu[\alpha] - J^\mu[\alpha^{\mu,*}] = & -\EE \Big[ \int_0^T \int_0^1 \int_0^u (\alpha_s - \alpha^*_s)\tp \\
& \qquad \na^2 H\parentheses{s,x_s,\mu_s,\alpha^*_s + v(\alpha_s - \alpha^*_s), -\nx V\ms(s,x_s)} (\alpha_s - \alpha^*_s) \,\rd v\,\rd u\,\rd s \Big].
\end{aligned}
\end{equation}
\end{lem}
\begin{proof}
By Lemma~\ref{lem:performance_difference}, we have
\begin{equation}\label{eq:cost_diff}
\begin{aligned}
& \quad J^\mu[\alpha] - J^\mu[\alpha\ms]  = \EE \Big[ \int_0^T \Big( H\big(s,x_s,\mu_s,\alpha\ms(s,x_s), -\nx V\ms(s,x_s), -\nx^2 V\ms(s,x_s)\big) \\
& \quad \quad \quad \quad - H\big(s,x_s,\mu_s,\alpha(s,x_s), -\nx V\ms(s,x_s), -\nx^2 V\ms(s,x_s)\big) \Big) \,\rd s \Big].
\end{aligned}
\end{equation}
For fixed $s$ and $x_s$, denote by \(H(\alpha)\) the mapping
$\alpha \mapsto H\big(s,x_s,\mu_s,\alpha, -\nx V\ms(s,x_s), -\nx^2 V\ms(s,x_s)\big).$
By Assumption~\ref{assu:Hconcave}, $H(\alpha)$ is $\lam_H$-strongly concave, and attains its maximum at $\alpha^* = \alpha\ms(s,x_s)$. Therefore, $\na H(\alpha^*) = 0$ and by standard calculus,
$$H(\alpha^*) - H(\alpha) = - \int_0^1 \int_0^u (\alpha-\alpha^*)\tp \,\na^2 H(\alpha^* + v(\alpha-\alpha^*)) \,(\alpha-\alpha^*) \, \rd v\,\rd u.$$
Substituting this identity into \eqref{eq:cost_diff} concludes the proof.
\end{proof}

By definition, the left-hand side of \eqref{eq:cost_gap} is always non-negative. Since $\na^2 H$ is negative definite by strong concavity, the right-hand side remains non-negative, serving as a sanity check. This lemma quantifies the cost gap between the current control \(\alpha\) and the optimal one \(\alpha^{\mu,*}\) under a fixed measure flow.

An analogous result holds for the value function, as stated below. The proof follows the same argument as that in Lemma~\ref{lem:cost_gap}, and is thus omitted.

\begin{lem}[Value function gap]\label{lem:value_gap}
With the same assumptions and notations as those in Lemma~\ref{lem:cost_gap}, 
\begin{equation}\label{eq:value_gap}
\begin{aligned}
V\ma(t,x) - & V\ms(t,x) = -\EE \Big[ \int_t^T \int_0^1 \int_0^u (\alpha_s - \alpha^*_s)\tp \\
& \na^2 H\parentheses{s,x_s,\mu_s,\alpha^*_s + v(\alpha_s - \alpha^*_s), -\nx V\ms(s,x_s)} (\alpha_s - \alpha^*_s) \,\rd v\,\rd u\,\rd s ~\Big|~ x_t = x \Big].
\end{aligned}
\end{equation}
\end{lem}

By the $\lam_H$-strong concavity of the Hamiltonian in $\alpha$, i.e., $\na^2 H \le - \lam_H I$, Lemma~\ref{lem:cost_gap} implies the following.

\begin{lem}[Landscape of the cost]\label{lem:modulus}
For any $\mu \in \mM$ and $\alpha \in \mA$,
$$J^{\mu}[\alpha] - J^{\mu}[\alpha^{\mu,*}] \ge \frac12 \lam_H \norm{\alpha - \alpha^{\mu,*}}^2_\mao.$$
\end{lem}

This result is called the modulus-of-continuity condition, i.e., $\norm{\alpha - \alpha^{\mu,*}}_\mao \le \omega(J^{\mu}[\alpha] - J^{\mu}[\alpha^{\mu,*}])$ for some function \(\omega:\R\to\R\).
Unlike previous literature \cite{zhou2024solving,zhou2025policy}, where the modulus-of-continuity was introduced as an assumption, here we rigorously prove the result and explicitly identify $\omega(\cdot)$ as the square-root function.

Next, we derive an upper bound of the cost gap, showing that the gap is at most quadratic in $\norm{\alpha - \alpha\ms}_\mao$.

\begin{lem}[Quadratic upper bound]\label{lem:quadratic_growth}
For any $\mu \in \mM$ and $\alpha \in \mA$, if $\alpha - \alpha\ms \in \mC$, then
\begin{equation}\label{eq:quadratic_growth}
J^{\mu}[\alpha] - J^{\mu}[\alpha^{\mu,*}] \le C \norm{\alpha - \alpha^{\mu,*}}^2_\mao\,.
\end{equation}
\end{lem}

\begin{proof}
In \eqref{eq:cost_gap}, denote $\alpha_s^v := \alpha^*_s + v(\alpha_s - \alpha^*_s)$. The Hessian term satisfies
\begin{align*}
&  \abs{\na^2 H\parentheses{s,x_s,\mu_s,\alpha^v_s, -\nx V\ms(s,x_s)}} \\
& \quad \le \abs{\na^2 b(s,x_s,\mu_s,\alpha^v_s)} \, \abs{\nx V\ms(s,x_s)} + \abs{\na^2 f(s,x_s,\mu_s,\alpha^v_s)} \le C (1 + |x_s|),
\end{align*}
where the last inequality follows from Lemma~\ref{lem:value_growth}. Using this bound, \eqref{eq:cost_gap} provides
\begin{align*}
 J^\mu[\alpha] - J^\mu[\alpha^{\mu,*}] 
& \le C \, \EE \Big[ \int_0^T  (1+|x_s|)\abs{\alpha_s - \alpha^*_s}^2 \,\rd s \Big] \le C \, \EE \Big[ \int_0^T \abs{\alpha_s - \alpha^*_s}^2 \,\rd s \Big] = C \norm{\alpha - \alpha^{\mu,*}}^2_\mao,
\end{align*}
where the second inequality follows from $\alpha-\alpha\ms \in \mC$.
\end{proof}

\subsection{Growth condition for the value function}
In this section, we quantify how the value function and its derivatives grow with respect to $|x|$.

\begin{lem}[Bounds for the value function and its derivatives]\label{lem:value_growth}
For any $\alpha \in \mA$ and $\mu \in \mM$,
\begin{equation}\label{eq:value_growth}
\begin{aligned}
&|V\ma(t,x)|,\,|\pt V\ma(t,x)|\le C(1+|x|^2),\quad |\nx V\ma(t,x)|\le C(1+|x|),\\
& |\nx^2 V\ma(t,x)|\le C(1+|x|),\ \forall (t,x)\in[0,T]\times\R^d.
\end{aligned}
\end{equation}
\end{lem}
\begin{proof}
Fix time $t_0 \in [0,T]$ and $x\in\RR^d$. Let $x_t$ be the state process under $(\mu, \alpha)$ with initial condition $x_{t_0} = x$. Define $f_t := f(t,x_t,\mu_t,\alpha(t,x_t))$, $b_t := b(t,x_t,\mu_t,\alpha(t,x_t))$, and $\sigma_t := \sigma(t,x_t,\mu_t)$, so that $\rd x_t = b_t \,\rd t + \sigma_t \rd W_t$.

\medskip
\noindent \emph{Step 1}. Prove $|V\ma(t_0,x)| \le C(1+|x|^2)$. By the definition of value function \eqref{eq:value_function}, 
\begin{align*}
 \abs{V\ma(t_0,x)} & = \bigg|\,\EE \Big[ \int_{t_0}^T f_t \,\rd t + g(x_T,\mu_T)\Big] \bigg| \le \EE \Big[ \int_{t_0}^T |f_t| \,\rd t + |g(x_T,\mu_T)|\Big] \\
& \le \EE \Big[ \int_{t_0}^T K\parentheses{1 + |x_t|^2 + W_2(\mu_t,\delta_0)^2 + |\alpha(t,x_t)|^2} \,\rd t + K(1+|x_T|^2 + W_2(\mu_T,\delta_0)^2) \Big] \\
& \le C \, \EE \Big[ \int_{t_0}^T (1+|x_t|^2) \,\rd t + 1+|x_T|^2 \Big] \le C (1 + |x|^2),
\end{align*}
where the second inequality is based on Assumption~\ref{assu:basic}, and the last follows from Gr\"onwall's inequality~\eqref{eq:Gronwall_square}.

\medskip
\noindent \emph{Step 2}. Prove $|\nx V\ma(t_0,x)| \le C(1+|x|)$. Let $x'_t$ be another state process under \((\mu,\alpha)\), driven by the same Brownian motion as \(x_t\), satisfying $\rd x'_t = b'_t \,\rd t + \sigma_t' \,\rd W_t$,  where $x'_{t_0} = x' \in \RR^d$, $b'_t:=b(t,x'_t,\mu_t,\alpha(t,x'_t))$, $\sigma_t':=\sigma(t,x'_t,\mu_t)$, and $f'_t := f(t,x'_t,\mu_t,\alpha(t,x'_t))$. 
Then
\begin{align*}
& \quad \abs{V\ma(t_0,x) - V\ma(t_0,x')}  \le \EE \Big[ \int_{t_0}^T |f_t-f'_t| \,\rd t + |g(x_T,\mu_T)-g(x_T',\mu_T)|\Big]  \\
& \le C \, \EE \Big[ \int_{t_0}^T (1+|x_t|\vee|x'_t|)|x_t-x_t'| \,\rd t + (1+|x_T|\vee|x'_T|) |x_T-x_T'|\Big] \le C (1+|x|\vee|x'|)|x-x'|,
\end{align*}
where the second inequality is based on Assumption~\ref{assu:basic}, and the third follows from~\eqref{eq:Gronwall_xsq_diffxsq100}. 
Setting $x' \to x$ in
$ \abs{V\ma(t_0,x) - V\ma(t_0,x')} / |x-x'| \le C (1+|x|\vee|x'|)$ concludes the proof.

We remark that, unlike standard H\"older estimation for parabolic equations (see \cite[Section 4.5]{ladyvzenskaja1988linear} for example), which guarantees 
H\"older differentiability of $V\ma$ without providing an explicit growth rate, we prove linear growth of the gradient $\nx V\ma$ in $|x|$.

\medskip
\noindent \emph{Step 3}. Prove $|\nx^2 V\ma(t_0,x)| \le C(1+|x|)$. Denote by $f\ma$ the mapping 
$(t,x) \mapsto f(t,x,\mu_t,\alpha(t,x)).$
By Assumption~\ref{assu:basic},
\begin{equation}\label{eq:nx_fma_bound}
|\nx f\ma(t,x)| \le |\nx f| + |\na f| \, |\nx \alpha| \le C(1+|x|),
\end{equation}
\begin{equation}\label{eq:nx2_fma_bound}
\begin{aligned}
  |\nx^2 f\ma(t,x)| & \le |\nx^2 f| + 2 |\nx\na f| \, |\nx a| + |\na^2 f| \, |\nx \alpha|^2 + |\na f| \, |\nx^2 \alpha| \\
& \le K + 2K^2 + K^3 + K^2(1 + |x| + W_2(\mu_t,\delta_0) + |\alpha(t,x)|) \le C(1+|x|).
\end{aligned}
\end{equation}

Take $e \in\RR^d$ as an arbitrary unit vector ($|e|=1$) and $\delta \in (0,1)$. Define $x^- := x-\delta e, \,x^+ := x + \delta e$. 
Let $x^+_t, x^-_t$ be the state processes under $(\mu,\alpha)$ starting at $x^+_{t_0} = x^+$, $x^-_{t_0} = x^-$. The two processes satisfy $\rd x_t^+ = b^+_t \,\rd t + \sigma_t^+ \rd W_t$, $\rd x_t^- = b^-_t \,\rd t + \sigma_t^- \rd W_t$, where $b^+_t := b(t,x^+_t,\mu_t,\alpha(t,x^+_t))$, $b^-_t := b(t,x^-_t,\mu_t,\alpha(t,x^-_t))$, $\sigma^+_t := \sigma(t,x^+_t,\mu_t)$, $\sigma^-_t := \sigma(t,x^-_t,\mu_t)$. Similar to \emph{Step 2}, $x^+_t$ and $x^-_t$ are driven by the same Brownian motion as $x_t$. 
Denote $f^\pm_t := f\ma(t,x^\pm)$ and $\bxt= \tfrac12(x^+_t + x^-_t)$, so that $\bar{x}_{t_0} = x$. 

We focus on estimating $f^+_t+f^-_t-2f_t$. By the mean value theorem, there exists \(\xi\in[-1,1]\), such that
$$h(1)+h(-1)-2h(0) = \int_0^1 (h'(s)-h'(-s)) \,\rd s = \int_0^1 \int_{-s}^s h''(\tau) \,\rd \tau \,\rd s = h''(\xi),$$
for a twice differentiable scalar function $h$. Applying this argument to $s \mapsto f\ma(t,\bxt + \frac12 s (x^+_t - x^-_t))$ yields
\begin{equation}\label{eq:f_central_diff}
\begin{aligned}
 \quad \abs{f\ma(t,x^+_t) + f\ma(t,x^-_t) - 2 f\ma(t,\bxt)} & = \frac14 \Big| (x^+_t-x^-_t)\tp \nx^2 f\ma(t,\xi_t)  (x^+_t-x^-_t) \Big| \\
& \hspace{-2em}\le C (1+|\xi_t|) \, |x^+_t-x^-_t|^2 \le C (1 + |x^+_t| + |x^-_t|) |x^+_t-x^-_t|^2,
\end{aligned}
\end{equation}
where $\xi_t$ lies between \(x^-_t\) and \(x^+_t\) and the inequality follows from \eqref{eq:nx2_fma_bound}. Using \eqref{eq:nx_fma_bound}, we get
\begin{equation}\label{eq:f_xbarx_diff}
|f\ma(t,\bxt) - f\ma(t,x_t)| \le C (1+|\bxt| + |x_t|)|\bxt - x_t| \le C (1+|\bxt| + |x_t|)|x^+_t + x^-_t - 2x_t|.
\end{equation}
Combining \eqref{eq:f_central_diff} and \eqref{eq:f_xbarx_diff} yields
\begin{equation}\label{eq:f_3terms_bound}
\begin{aligned}
& \quad \abs{\EE[f^+_t+f^-_t-2f_t]} = \abs{ \EE\sqbra{ f\ma(t,x^+_t) + f\ma(t,x^-_t) - 2f\ma(t,x_t) } } \\
& \le  \EE\sqbra{ \abs{f\ma(t,x^+_t) + f\ma(t,x^-_t) - 2f\ma(t,\bxt) } + 2\abs{f\ma(t,\bxt) - f\ma(t,x_t)} }  \\
& \le C \, \EE \sqbra{ (1 + |x^+_t| + |x^-_t|) |x^+_t-x^-_t|^2 + (1+|\bxt| + |x_t|)|x^+_t + x^-_t - 2x_t|} \\
& \le C (1 + |x|) \parentheses{|x^+_{t_0} - x^-_{t_0}|^2 + |x^+_{t_0} + x^-_{t_0} - 2x_{t_0}|} = 4C (1+|x|) \delta^2,
\end{aligned}
\end{equation}
where we use Gr\"onwall's inequalities \eqref{eq:Gronwall_xsq_diffx4} and \eqref{eq:Gronwall_xsq_x3terms}. Similarly, we can show
\begin{equation}\label{eq:g_3terms_bound}
\abs{ \EE\sqbra{ g(x^+_T,\mu_T) + g(x^-_T,\mu_T) - 2g(x_T,\mu_T) } } \le C (1+|x|) \delta^2.
\end{equation}
Combining \eqref{eq:f_3terms_bound} and \eqref{eq:g_3terms_bound} provides
\begin{align*}
& \quad \abs{V\ma(t_0,x^+) + V\ma(t_0,x^-) - 2V\ma(t_0,x)} \\
& \le \EE \Big[ \int_{t_0}^T \abs{f^+_t + f^-_t - 2 f_t} \, \rd t + \abs{g(x^+_T,\mu_T) + g(x^-_T,\mu_T) - 2g(x_T,\mu_T)} \Big]  \le C (1 + |x|) \delta^2.
\end{align*}
Setting $\delta \to 0$ 
yields
$\abs{e\tp \nx^2 V\ma(t_0,x) e} \le C(1+|x|).$
Since $e$ is an arbitrary unit vector and all matrix norms are equivalent, \(\abs{\nx^2 V\ma(t_0,x)} \le C(1+|x|)\).

\medskip
\noindent \emph{Step 4}. Prove $|\pt V\ma(t,x)| \le C(1+|x|^2)$. 
Applying previously proved conclusions to the PDE \eqref{eq:value_PDE} yields
\begin{align*}
& \quad |\pt V\ma(t,x)| = \abs{ \Tr\parentheses{D(t,x,\mu_t) \,\nx^2 V\ma(t,x)} + b(t,x,\mu_t,\alpha(t,x))\tp \nx V\ma(t,x) + f(t,x,\mu_t,\alpha(t,x)) } \\
& \le C \abs{\nx^2 V\ma(t,x)} + C (1+|x|) \abs{\nx V\ma(t,x)} + C (1 + |x|^2) \le C (1 + |x|^2),
\end{align*}
which concludes the proof.
\end{proof}

\subsection{Lipschitz condition for the value function}
In this section, we show that the value function satisfies a Lipschitz-type stability condition with respect to both the distribution $\mu$ and the control $\alpha$.

\begin{lem}[Lipschitz condition for value function]\label{lem:Value_Lipschitz}
For any $\alpha,\alpha' \in \mA$ such that $\alpha - \alpha' \in \mC$, and any $\mu, \mu' \in \mM$, let $x_t := X\ma_t$ and $x'_t:=X\map_t$ be two state processes under $(\mu,\alpha)$ and $(\mu',\alpha')$, starting from the same initial condition $x_0\in\R^d$, driven by the same Brownian motion. Denote $f_t := f(t,x_t,\mu_t,\alpha(t,x_t))$, $f_t' := f(t,x_t',\mu_t',\alpha'(t,x_t'))$ and define $b_t$, $b'_t$, $\sigma_t$, $\sigma'_t$ similarly. 
The following three bounds hold:
\begin{align}
\norm{V\ma(0,\cdot) - V\map(0,\cdot)}_{\rho_0}^2 &\le C \parentheses{ \mW_2(\mu,\mu')^2 + W_2(\mu_T,\mu'_T)^2 + \norm{\alpha - \alpha'}^2_\mao}, \label{eq:Lip_V0}\\
\EE \Big[ \int_0^T  \abs{\sigma_t'^{\top}p'_t - \sigma_t^{\top}p_t}^2\,\rd t\Big] & \le C \parentheses{ \mW_2(\mu,\mu')^2 + W_2(\mu_T,\mu'_T)^2 + \norm{\alpha - \alpha'}^2_\mao}, \label{eq:Lip_sigma_gradV} \\
\norm{\nx V\ma - \nx V\map}_{\mao}^2 &\le C \parentheses{ \mW_2(\mu,\mu')^2 + W_2(\mu_T,\mu'_T)^2 + \norm{\alpha - \alpha'}^2_\mao}, \label{eq:Lip_gradV}
\end{align}
where processes $p_t := \nx V\ma(t,x_t)$ and $p'_t := \nx V\map(t,x'_t)$.
\end{lem}

\begin{proof}
The two state processes follow $\rd x_t = b_t\,\rd t + \sigma_t \,\rd W_t$ and $\rd x_t' = b_t'\,\rd t + \sigma_t' \,\rd W_t$.

\medskip
\noindent\emph{Step 1. Proof of \eqref{eq:Lip_V0}.} By the definition of the value function \eqref{eq:value_function},
\begin{equation}\label{eq:V0x_diff}
\abs{V\ma(0,x_0) - V\map(0,x_0)} \le \EE \Big[ \int_0^T |f_t - f'_t| \,\rd t + |g(x_T,\mu_T)-g(x_T',\mu_T')| ~\big|~ x_0 \Big].
\end{equation}
Using the Lipschitz property of $f$ (cf. Assumption~\ref{assu:basic}),
\begin{equation}\label{eq:ft_diff}
\begin{aligned}
& \quad |f_t-f_t'| = |f(t,x_t,\mu_t,\alpha(t,x_t)) - f(t,x_t',\mu_t',\alpha'(t,x_t'))| \\
& \le K \parentheses{1+|x_t|\vee|x'_t| + W_2(\mu_t,\delta_0)\vee W_2(\mu'_t,\delta_0) + |\alpha(t,x_t)|\vee|\alpha'(t,x_t')|}\\
& \quad \cdot \parentheses{ |x_t-x_t'| + W_2(\mu_t,\mu_t') + |\alpha(t,x_t)-\alpha'(t,x_t')|} \\
& \le C (1+|x_t|\vee|x'_t|) \parentheses{|x_t-x_t'| + W_2(\mu_t,\mu_t') + |\alpha(t,x_t)-\alpha'(t,x_t)|}.
\end{aligned}
\end{equation}
Similarly, we can show that
\begin{equation}\label{eq:gT_diff}
|g(x_T,\mu_T)-g(x_T',\mu_T')| \le C (1+|x_T|\vee|x'_T|) \parentheses{|x_T-x_T'| + W_2(\mu_T,\mu_T')}.
\end{equation}
Plugging \eqref{eq:ft_diff} and \eqref{eq:gT_diff} into \eqref{eq:V0x_diff} yields
\begin{equation}\label{eq:V0x_diff_bound}
\begin{aligned}
& \quad \abs{V\ma(0,x_0) - V\map(0,x_0)} \le C \, \EE \Big[ \int_0^T  (1+|x_t|\vee|x'_t|) \parentheses{|x_t-x_t'| + W_2(\mu_t,\mu_t') + |\alpha(t,x_t)-\alpha'(t,x_t)|} \,\rd t \\
& \quad\quad \quad + (1+|x_T|\vee|x'_T|) \parentheses{|x_T-x_T'| + W_2(\mu_T,\mu_T')} ~\big|~ x_0 \Big].
\end{aligned}
\end{equation}
Therefore,
\begin{equation}\label{eq:normV0_diff_temp}
\begin{aligned}
& \quad \norm{V\ma(0,\cdot) - V\map(0,\cdot)}_{\rho_0}^2 = \EE\sqbra{|V\ma(0,x_0) - V\map(0,x_0)|^2} \\
& \le C \, \EE \Big[ \int_0^T  (1+|x_t|\vee|x'_t|)^2 \parentheses{|x_t-x_t'| + W_2(\mu_t,\mu_t') + |\alpha(t,x_t)-\alpha'(t,x_t)|}^2 \,\rd t \\
& \quad\quad \quad + (1+|x_T|\vee|x'_T|)^2 \parentheses{|x_T-x_T'| + W_2(\mu_T,\mu_T')}^2  \Big] \\
& \le C \Big\{ \EE \Big[ \int_0^T \parentheses{  (1+|x_t|^2+|x'_t|^2) |x_t-x_t'|^2 + (1+|x_t|^2 + |x_t-x'_t|^2) |\alpha(t,x_t)-\alpha'(t,x_t)|^2} \,\rd t \\
& \quad\quad  + (1+|x_T|^2+|x'_T|^2)|x_T-x_T'|^2 \Big] + \mW_2(\mu,\mu')^2 + W_2(\mu_T,\mu'_T)^2 \Big\} \\
& \le C \Big\{ \EE \Big[ \int_0^T (|x_t|^2 + |x_t-x'_t|^2) |\alpha(t,x_t)-\alpha'(t,x_t)|^2 \,\rd t \Big] + \mW_2(\mu,\mu')^2 +\\
&\quad\quad W_2(\mu_T,\mu'_T)^2 + \norm{\alpha - \alpha'}^2_\mao \Big\}.
\end{aligned}
\end{equation}
Here, in the first inequality, we use~\eqref{eq:V0x_diff_bound}, apply Cauchy-Schwarz and H\"older's inequality, then use the tower property. 
In the second inequality, we use~\eqref{eq:Gronwall_square}, triangle inequality, and Cauchy-Schwarz. 
The last inequality follows from Gr\"onwall's inequality \eqref{eq:Gronwall_xsq_diffxsq011}.

Since $\alpha-\alpha' \in \mC$,
\begin{equation}\label{eq:x2_alphadiff_central}
\EE \Big[ \int_0^T |x_t|^2 |\alpha(t,x_t)-\alpha'(t,x_t)|^2 \,\rd t \Big] \le C \, \EE \Big[ \int_0^T |\alpha(t,x_t)-\alpha'(t,x_t)|^2 \,\rd t \Big] = C \norm{\alpha - \alpha'}^2_\mao.
\end{equation}
Since $\alpha,\alpha' \in \mA$,
\begin{equation}\label{eq:x2diff_alphadiff}
\begin{aligned}
\EE \Big[ \int_0^T |x_t-x'_t|^2 \, |\alpha(t,x_t)-\alpha'(t,x_t)|^2 \,\rd t \Big] & \le C \, \EE \Big[ \int_0^T (1 + |x_t|^2) |x_t-x'_t|^2 \,\rd t \Big] \\
& \le C \parentheses{\norm{\alpha - \alpha'}^2_\mao + \mW_2(\mu,\mu')^2},
\end{aligned}
\end{equation}
where we use the linear growth of the control in the first inequality, followed by Gr\"onwall's inequality~\eqref{eq:Gronwall_xsq_diffxsq011}. Substituting \eqref{eq:x2_alphadiff_central} and \eqref{eq:x2diff_alphadiff} into \eqref{eq:normV0_diff_temp} concludes the proof.

\medskip
\noindent\emph{Step 2. Proof of \eqref{eq:Lip_sigma_gradV}.} Let  $V_t := V\ma(t,x_t)$ and $V_t' := V\map(t,x'_t)$. 
Applying It\^o's lemma and using the PDE \eqref{eq:value_PDE} yield (cf. derivation of equation~\eqref{eq:BSDE_trick})
$$\rd V_t = -f_t \,\rd t + p_t\tp \sigma_t \,\rd W_t, \quad \rd V_t' = -f_t' \,\rd t + p_t^{\prime\top} \sigma_t' \,\rd W_t.$$
Subtracting two equations and integrating from $0$ to $T$ yield
\begin{equation}\label{eq:integrade_diff_Vt}
(g(x_T,\mu_T) - g(x'_T,\mu_T')) - (V_0-V'_0) = -\int_0^T (f_t-f'_t)\,\rd t + \int_0^T (p_t\tp \sigma_t -p_t^{\prime\top} \sigma_t') \,\rd W_t.
\end{equation}
Based on It\^o's isometry, we conclude the proof by noticing
\begin{align*}
& \quad \EE \Big[ \int_0^T  \abs{\sigma_t^{\top}p_t - \sigma_t'^{\top}p'_t}^2\,\rd t\Big]  = \EE\Big[ \Big(\int_0^T (p_t\tp \sigma_t - p_t^{\prime\top} \sigma_t') \,\rd W_t \Big)^2\Big] \\
& = \EE\Big[ \Big( (g(x_T,\mu_T) - g(x'_T,\mu_T')) - (V_0-V'_0) + \int_0^T (f_t-f'_t)\,\rd t \Big)^2\Big] \\
& \le 3 \EE\Big[ (g(x_T,\mu_T) - g(x'_T,\mu_T'))^2 + (V_0-V'_0)^2 + T \int_0^T (f_t-f'_t)^2 \,\rd t \Big] \\
& \le C \parentheses{ \mW_2(\mu,\mu')^2 + W_2(\mu_T,\mu'_T)^2 + \norm{\alpha - \alpha'}^2_\mao},
\end{align*}
where the last inequality follows from estimations \eqref{eq:gT_diff}, \eqref{eq:normV0_diff_temp}, \eqref{eq:ft_diff} previously derived in \emph{Step 1}.

\medskip
\noindent\emph{Step 3. Proof of \eqref{eq:Lip_gradV}.} Note that
\begin{equation}\label{eq:sigma_gradV_diff_norm_split}
\begin{aligned}
& \quad \norm{\sigma(t,x,\mu_t)\tp (\nx V\ma(t,x) - \nx V\map(t,x))}_\mao^2 \\
& = \EE \Big[ \int_0^T \big| \sigma_t\tp \nx V\ma(t,x_t) - \sigma_t\tp \nx V\map(t,x_t) \big|^2 \,\rd t \Big] \\
& \le 3 \EE \Big[ \int_0^T \Big( \big| \sigma_t\tp \nx V\ma(t,x_t) - \sigma_t^{\prime\top} \nx V\map(t,x'_t) \big|^2\\
& \quad \quad + \big| \sigma_t^{\prime\top} \nx V\map(t,x'_t) - \sigma_t\tp \nx V\map(t,x_t') \big|^2 \\
& \quad \quad + \big| \sigma_t\tp \nx V\map(t,x_t') - \sigma_t\tp \nx V\map(t,x_t) \big|^2 \Big) \,\rd t \Big].
\end{aligned}
\end{equation}
We estimate each of the three terms on the right-hand side of~\eqref{eq:sigma_gradV_diff_norm_split}. 
The first term is bounded in \emph{Step 2}.

For the second and third terms, we apply Lemma~\ref{lem:value_growth}. The second term reads
\begin{equation}\label{eq:sigma_gradV_diff_norm_term2}
\begin{aligned}
& \quad \EE \Big[ \int_0^T \big| \sigma_t^{\prime\top} \nx V\map(t,x'_t) - \sigma_t\tp \nx V\map(t,x_t') \big|^2 \ud t\Big] \le C \,\EE \Big[ \int_0^T \big|\sigma_t' - \sigma_t \big|^2 (1+|x_t'|)^2 \ud t\Big] \\
& \le C \,\EE \Big[ \int_0^T \parentheses{|x_t-x_t'| + W_2(\mu_t,\mu_t')}^2 \, (1+|x_t'|^2) \,\rd t \Big]  \\
& \le C \,\EE \Big[ \int_0^T (1+|x_t'|^2)(|x_t-x_t'|^2 + W_2(\mu_t,\mu_t')^2)  \,\rd t \Big] \le C \parentheses{ \mW_2(\mu,\mu')^2 + \norm{\alpha - \alpha'}^2_\mao},
\end{aligned}
\end{equation}
where the last inequality follows from Gr\"onwall's inequalities~\eqref{eq:Gronwall_square} and~\eqref{eq:Gronwall_xsq_diffxsq011}. 

The third term in \eqref{eq:sigma_gradV_diff_norm_split} can be bounded as follows: 
\begin{equation}\label{eq:sigma_gradV_diff_norm_term3}
\begin{aligned}
& \quad \EE \Big[ \int_0^T \big| \sigma_t\tp \nx V\map(t,x_t') - \sigma_t\tp \nx V\map(t,x_t) \big|^2 \ud t\Big]  \\
& \le C\EE \Big[ \int_0^T \big|\nx V\map(t,x_t') - \nx V\map(t,x_t) \big|^2 \,\rd t \Big]  \\
& \le C \, \EE \Big[ \int_0^T (1+|x_t|^2) |x_t-x_t'|^2 \,\rd t \Big] \le C \parentheses{ \mW_2(\mu,\mu')^2 + \norm{\alpha - \alpha'}^2_\mao},
\end{aligned}
\end{equation}
where the last inequality follows from 
Gr\"onwall's inequality~\eqref{eq:Gronwall_xsq_diffxsq011}. 

Plugging bounds~\eqref{eq:Lip_sigma_gradV}, \eqref{eq:sigma_gradV_diff_norm_term2} and \eqref{eq:sigma_gradV_diff_norm_term3} into \eqref{eq:sigma_gradV_diff_norm_split} yields
\begin{equation}
\norm{\sigma(t,x,\mu_t)\tp (\nx V\ma(t,x) - \nx V\map(t,x))}_\mao^2 \le  C \parentheses{ \mW_2(\mu,\mu')^2 + W_2(\mu_T,\mu'_T)^2 + \norm{\alpha - \alpha'}^2_\mao}.
\end{equation}
Since $D$ satisfies uniform ellipticity (cf. Assumption~\ref{assu:mu0_sigma0}), we have 
$$\norm{\sigma(t,x,\mu_t)\tp (\nx V\ma(t,x) - \nx V\map(t,x))}_\mao^2 \ge 2\sigma_0 \norm{\nx V\ma(t,x) - \nx V\map(t,x)}_\mao^2,$$
which concludes the proof.
\end{proof}

\subsection{Properties for OTGP flow}
In this section, we establish several key properties of the OTGP flow defined in \eqref{eq:OTGP_flow}. We first show that the Picard iteration $\mu \mapsto \rho\ma$ is a contraction under the metric $d_\beta$ for a properly chosen $\beta = 34K^2 + \frac{51}{2}K$.

\begin{lem}[Contraction for Picard iteration]\label{lem:controaction_FPI}
For any $\mu,\nu \in \mM$ and $\alpha \in \mA$,
\begin{equation}\label{eq:contraction_FPI}
d_\beta(\rho\ma, \rho^{\nu,\alpha}) \le \kappa d_\beta(\mu, \nu),
\end{equation}
where $\kappa := [(4K^2+3K) / (2\beta-(4K^2+3K))]^{\frac12} = \frac14<1$, provided that $\beta = 34K^2 + \frac{51}{2}K$.
\end{lem}

\begin{proof}
Let $x^\mu_t$, $x^\nu_t$ be two state processes under $(\mu,\alpha)$ and $(\nu,\alpha)$ respectively, starting from the same initial condition $x^\mu_0 \overset{\mathrm{a.s.}}{=} x^\nu_0 \sim \rho_0$, driven by the same Brownian motion. Their dynamics are
\begin{align*}
\rd x^\mu_t &= b(t,x^\mu_t, \mu_t,\alpha(t,x^\mu_t)) \, \rd t + \sigma(t,x^\mu_t, \mu_t) \,\rd W_t =: b^\mu_t \,\rd t + \sigma^\mu_t \,\rd W_t,\\
\rd x^\nu_t &= b(t,x^\nu_t, \nu_t,\alpha(t,x^\nu_t)) \, \rd t + \sigma(t,x^\nu_t, \nu_t) \,\rd W_t =: b^\nu_t \,\rd t + \sigma^\nu_t \,\rd W_t\,.
\end{align*}
By definition, $x^\mu_t \sim \rho\ma_t$ and $x^\nu_t \sim \rho^{\nu,\alpha}_t$. 
Since $\mu,\nu\in\mM$ and $\alpha\in\mA$, by Assumption~\ref{assu:basic},
\begin{equation}\label{eq:b_sigma_Lip}
\begin{aligned}
&|b^\mu_t - b^\nu_t| \le K \parentheses{ |x^\mu_t - x^\nu_t| + W_2(\mu_t,\nu_t) + |\alpha(t,x^\mu_t) - \alpha(t,x^\nu_t)|} \le K(1+K)  |x^\mu_t - x^\nu_t| + K W_2(\mu_t,\nu_t),\\
&|\sigma^\mu_t - \sigma^\nu_t| \le K \parentheses{ |x^\mu_t - x^\nu_t| + W_2(\mu_t,\nu_t) }.
\end{aligned}
\end{equation}
By It\^o's lemma,
$$\rd |x^\mu_t - x^\nu_t|^2 = \sqbra{ 2(x^\mu_t - x^\nu_t)\tp (b^\mu_t - b^\nu_t) + |\sigma^\mu_t - \sigma^\nu_t|^2} \,\rd t + 2 (x^\mu_t - x^\nu_t)\tp(\sigma^\mu_t - \sigma^\nu_t)\,\rd W_t.$$
Denote $D_t := \EE \sqbra{|x^\mu_t - x^\nu_t|^2}$, so that $D_0=0$ and
\begin{align*}
& \quad \pt D_t = \EE \sqbra{ 2(x^\mu_t - x^\nu_t)\tp (b^\mu_t - b^\nu_t) + |\sigma^\mu_t - \sigma^\nu_t|^2} \\
& \le \EE\sqbra{2|x^\mu_t - x^\nu_t| \parentheses{K(1+K)  |x^\mu_t - x^\nu_t| + K W_2(\mu_t,\nu_t)} + K^2 \parentheses{ |x^\mu_t - x^\nu_t| + W_2(\mu_t,\nu_t) }^2 } \\
& \le (4K^2+3K) \EE\sqbra{ |x^\mu_t - x^\nu_t|^2 + W_2(\mu_t,\nu_t)^2 } = (4K^2+3K)(D_t + W_2(\mu_t,\nu_t)^2).
\end{align*}
By Gr\"onwall's inequality, 
\(D_t\leq (4K^2+3K)\int_0^t e^{(4K^2+3K)(t-s)}W_2(\mu_s,\nu_s)^2\ud s\).
Therefore,
\begin{align*}
& \quad d_\beta(\rho\ma,\rho^{\nu,\alpha})^2 = \int_0^T e^{-2\beta t} \, W_2(\rho\ma_t,\rho^{\nu,\alpha}_t)^2 \,\rd t \le \int_0^T e^{-2\beta t} D_t \,\rd t \\
& \le \int_0^T e^{-2\beta t} (4K^2+3K)\int_0^t e^{(4K^2+3K)(t-s)} \, W_2(\mu_s,\nu_s)^2 \,\rd s \, \rd t \\
& = (4K^2+3K)\int_0^T \Big(\int_s^T e^{(4K^2 + 3K - 2\beta)  t}  \, \rd t \Big) \, e^{-(4K^2+3K)s} \, W_2(\mu_s,\nu_s)^2 \,\rd s \\
& \le \dfrac{4K^2+3K}{2\beta - (4K^2+3K)} \int_0^T  e^{(4K^2+3K - 2\beta) s} \, e^{-(4K^2+3K)s} \, W_2(\mu_s,\nu_s)^2 \,\rd s \\
& = \dfrac{4K^2+3K}{2\beta - (4K^2+3K)} \int_0^T  e^{-2\beta s} W_2(\mu_s,\nu_s)^2 \,\rd s = \kappa^2 d_\beta(\mu,\nu)^2.
\end{align*}
This concludes the proof.
\end{proof}

Next, we quantify the rate at which $\mu^\tau$ moves away from itself towards \(\rho\mat_t\).
This result is a direct corollary of \cite[Theorem 7.2.2]{ambrosio2005gradient}.
\begin{lem}\label{lem:constant_speed}
The OTGP flow \eqref{eq:OTGP_flow} satisfies
\begin{equation}\label{eq:constant_speed}
\dfrac{\rd}{\rd\tau} W_2(\mu^\tau_t, \nu_t)\Big|_{\nu_t = \mu_t^\tau} = \beta_\mu W_2(\mu^\tau_t, \rho\mat_t), \ \forall t \in [0,T].
\end{equation}
\end{lem}


\subsection{Moreau envelope}\label{sec:Moreau}
We introduce several properties of the Moreau envelope in this section, which will be used later in the proof of Lemma~\ref{lem:technical} in Section~\ref{sec:technical}.

Let $V \in C_{\text{loc}}^{1,2}([0,T] \times \RR^d)$ and $\io \in (0,1)$. Define
\begin{equation}\label{eq:Moreau_def}
\begin{aligned}
V_\io(t,x) & := \inf_{(s,y) \in [0,T] \times \RR^d} \Big[ V(s,y) + \frac{1}{2\io} \parentheses{ |t-s|^2 + |x-y|^2 } \Big], \\
V^\io(t,x) & := \sup_{(s,y) \in [0,T] \times \RR^d} \Big[ V(s,y) - \frac{1}{2\io} \parentheses{ |t-s|^2 + |x-y|^2 } \Big],
\end{aligned}
\end{equation}
and denote the proximal operators by
\begin{equation}\label{eq:proximal_operator}
\begin{aligned}
\prox_\io[V](t,x) &:= \argmin_{(s,y) \in [0,T] \times \RR^d} \Big[ V(s,y) + \frac{1}{2\io} \parentheses{ |t-s|^2 + |x-y|^2 } \Big], \\
\prox^\io[V](t,x) & := \argmax_{(s,y) \in [0,T] \times \RR^d} \Big[ V(s,y) - \frac{1}{2\io} \parentheses{ |t-s|^2 + |x-y|^2 } \Big].
\end{aligned}
\end{equation}
When the minimizer or maximizer is not unique, the proximal operator returns a set of values. 

\begin{lem}[Moreau envelope]\label{lem:Moreau}
Let $R > 1$ and \(B_{R-1}\) be the closed ball in \(\R^d\) of radius \(R-1\) centered at origin.
Assume that there exists \(C_V\) such that
\begin{equation}\label{eq:Moreau_growthcond}
|V(t,x)|,|\pt V(t,x)| \le C_V (1+|x|^2), \quad |\nx V(t,x)| \le C_V (1+|x|),\ \forall (t,x)\in[0,T]\times\R^d.
\end{equation}
The following conclusions hold under $\io < \frac{1}{2C_V}$:
\begin{enumerate}[(1)]
\item $V_\io$ is semiconcave and $V^\io$ is semiconvex.
\item For any $x \in B_{R-1}$ and $t \in [0,T]$, if $\iota \le \frac{1}{4C_V R (2R^2+1)} \wedge \frac{1}{4C_V(2T+1)}$, then for any $(s,y) \in \prox_\io[V](t,x)$ or $\prox^\io[V](t,x)$,
\begin{equation}\label{eq:Moreau_result1}
2R|t-s| + |x-y| \le 4\,\io \, C_V R (2R^2+1) \quad \text{and} \quad |y| \le R.
\end{equation}
Additionally,
\begin{equation}\label{eq:Moreau_result2}
V(t,x) - V_\io(t,x), V^\io(t,x) - V(t,x) \in \sqbra{0, 4\io \, C_V^2 R^4}.
\end{equation}
\item $V_\io$ and $V^\io$ satisfy a local Lipschitz condition: for any $ x,y \in B_{R-1}$ and $t,s \in [0,T]$,
\begin{equation}\label{eq:Moreau_result3}
\abs{V_\io(t,x) - V_\io(s,y)}, \abs{V^\io(t,x) - V^\io(s,y)} \le \abs{(t,x) - (s,y)} \cdot 4 C_V R \sqrt{2R^2+1}.
\end{equation}
\item For any $(t,x) \in [0,T] \times \RR^d$,
\begin{equation}\label{eq:Moreau_result4}
\abs{V_\io(t,x)}, \abs{V^\io(t,x)} \le C (1+|x|^2),
\end{equation}
where $C$ is independent of $\io$.
\item The proximal operator satisfies the critical point equation: if $(s,y) \in \prox_\io[V](t,x)$ resp. $\prox^\io[V](t,x)$,
\begin{equation}\label{eq:Moreau_result5}
\begin{aligned}
&\pt V_\io(t,x) = \ps V(s,y), \quad \nx V_\io(t,x) = \ny V(s,y), \quad \nx^2  V_\io(t,x) \le \ny^2 V(s,y), \quad \text{resp.} \\
& \pt V^\io(t,x) = \ps V(s,y), \quad \nx V^\io(t,x) = \ny V(s,y), \quad \nx^2  V^\io(t,x) \ge \ny^2 V(s,y).
\end{aligned}
\end{equation}
\end{enumerate}
\end{lem}

The Moreau envelope has been well studied in \cite{jourani2014differential}.
When $\io < \frac{1}{2C_V}$, both envelopes $V_\io$, $V^\io$ are well-defined, and their associated proximal operators return non-empty sets.
By Alexandrov theorem~\cite{aleksandorov1939almost}, either semiconvexity or semiconcavity implies the almost everywhere existence of the second-order derivatives of $V_\io$ and $V^\io$.
As a result, \eqref{eq:Moreau_result5} holds in the almost everywhere sense.
\begin{proof}
It suffices to prove the statement for $V_\io$; the results for $V^\io$ follow symmetrically.  

We show (1) first. Define $g_\io(t,x) := V_\io(t,x) - \frac{1}{2\io} (t^2 + |x|^2)$. We show that $g_\io(t,x)$ is concave. It suffices to verify
\(
g_\io(t+h, x+z) + g_\io(t-h, x-z) \le 2g_\io(t, x), \ \forall t, h, x, z
\),
where \([t-h, t+h] \subset [0, T]\). For any \((s, y) \in \prox_{\io}[V](t, x)\),
\begin{align*}
& \quad g_\io(t+h, x+z) + g_\io(t-h, x-z) \\
&\leq V(s, y) + \tfrac{1}{2\io}(|t+h-s|^2 + |x+z-y|^2)  - \tfrac{1}{2\io}(t+h)^2  - \tfrac{1}{2\io}|x+z|^2\\
&\quad + V(s, y) + \tfrac{1}{2\io}(|t-h-s|^2 + |x-z-y|^2) - \tfrac{1}{2\io}(t-h)^2 - \tfrac{1}{2\io}|x-z|^2 \\
&= 2V(s, y) + \tfrac{1}{\io} \big( |t-s|^2 + |x-y|^2 \big) - \tfrac{1}{\io} (t^2 + |x|^2) \\
&= 2V_\io(t, x) - \tfrac{1}{\io} (t^2 + |x|^2) = 2 g_\io(t, x),
\end{align*}
which concludes the proof.

Next, we prove (2). Let \((s, y) \in \prox_{\io}[V](t, x)\), then
\(V(s, y) + \frac{1}{2\io}(|t-s|^2 + |x-y|^2) \le V(t, x).\)
Therefore,
\begin{equation*}
\begin{aligned}
& \quad \frac{1}{2\io}(|t-s|^2 + |x-y|^2) \leq |V(t, x) - V(s, y)| \leq C_V \sqbra{ (1+|x|^2 \vee |y|^2)|t-s| + (1+|x| \vee |y|)|x-y|} \\
& \le C_V \sqbra{ (1 + 2|x|^2 + 2|x-y|^2)|t-s| + (1+|x|+|x-y|)|x-y|}.
\end{aligned}
\end{equation*}
Moving the terms $|x-y|^2$ to the left, we get
\begin{equation*}
\frac{1}{2\io}|t-s|^2 + \frac{1}{4\io}|x-y|^2 \leq C_V(2R^2|t-s| + R |x-y|).
\end{equation*}
By Cauchy's inequality,
$$(2R^2+1)\, 4\io\, C_V R \,(2R|t-s| + |x-y|) \ge (2R^2+1) \parentheses{2|t-s|^2 + |x-y|^2} \ge (2R|t-s| + |x-y|)^2,$$
which implies
$2R|t-s| + |x-y| \le 4\io\, C_V R (2R^2+1).$
This inequality, together with the range for $\iota$, further implies $|x-y| \le 1$ and hence $|y| \le R$. As a consequence,
\begin{align*}
0 &\le V(t, x) - V_\io(t, x) =  V(t, x) - V(s,y) - \frac{1}{2\io}(|t-s|^2 + |x-y|^2) \\
& \le C_V \sqbra{(1+R^2)|t-s| + (1+R)|x-y|} - \frac{1}{2\io}(|t-s|^2 + |x-y|^2)\\
&\le \frac12 \io C_V^2[(1+R^2)^2 + (1+R)^2] \le 4\io C_V^2 R^4.
\end{align*}
This concludes the proof of (2).

Next, we prove (3). By \cite[Theorem 3.2]{jourani2014differential}, the superdifferential of $V_\io(t,x)$ is the convex hull of $\frac{1}{\io}[(t,x) - \prox_\io[V](t,x)]$. Let $(s,y)$ lie in the convex hull of $ \prox_\io[V](t,x)$, then the estimates in (2) still hold, i.e.,
$$|x-y|\le1, \quad |y|\le R,\quad \frac{1}{2\io}(|t-s|^2 + |x-y|^2) \le 2C_V(2R^2|t-s| + R |x-y|) \le 8 \io C_V^2 R^2 (2R^2 + 1).$$
Taking square root, we get $\frac{1}{\io} |(t,x) - (s,y)| \le 4C_V R \sqrt{2R^2+1}$. Therefore, $V_\io$ has a (local) Lipschitz constant $4C_V R \sqrt{2R^2+1}$.

Next, we show (4). For any \((s, y) \in \prox_{\io}[V](t, x)\), using $|x-y|\le 1$,
$$\abs{V(s,y) - V_\io(t,x)} = \frac{1}{2\io}|t-s|^2 + \frac{1}{2\io}|x-y|^2 \le 2C_V[2(1+|x|)^2T + (1+|x|) |x-y|] \le C(1+|x|^2).$$
Together with $|V(s,y)| \le C_V(1+|y|^2) \le C(1+|x|^2)$, we obtain $\abs{V_\io(t,x)} \le C(1+|x|^2)$.

Finally, we prove (5). Let \((s, y) \in \prox_{\io}[V](t, x)\), where $(t,x)$ is such that $\pt V_\io(t,x)$, $\nx V_\io(t,x)$, and $\nx^2 V_\io(t,x)$ exist. Then, for any $(\hat t, \hat x) \in [0,T] \times \RR^d$, we have
$$V_\io(\hat t,\hat x) \le V(\hat t-t+s,\hat x-x+y) + \frac{1}{2\io} \parentheses{|t-s|^2 + |x-y|^2},$$
and hence
$$V_\io(\hat t,\hat x) - V(\hat t-t+s,\hat x-x+y) \le \frac{1}{2\io} \parentheses{|t-s|^2 + |x-y|^2}  = V_\io(t,x) -V(s,y).$$
Therefore, the mapping
$(\hat t, \hat x) \mapsto V_\io(\hat t,\hat x) - V(\hat t-t+s,\hat x-x+y)$
attains its maximum at $(t,x)$. The first- and second-order optimality condition provides
$$\pt V_\io(t,x) = \ps V(s,y), \quad \nx V_\io(t,x) = \ny V(s,y), \quad \nx^2  V_\io(t,x) \le \ny^2 V(s,y),$$
which concludes the proof.
\end{proof}

\section{Proofs for the actor}\label{sec:proof_actor}
The proof for Proposition~\ref{prop:policy_gradient} is the same as \cite[Proposition 1]{zhou2025policy}, where we show 
$$\dfrac{\rd}{\rd\ve} J^\mu[\alpha+\ve\phi]\Big|_{\ve=0} = -\inttx{\na H(t,x,\mu_t,\alpha(t,x),-\nx V\ma(t,x)) \tp \phi(t,x) \,\rho\ma(t,x)},$$
for any smooth and $\rho\ma$-square integrable test function $\phi:[0,T] \times \RR^d \to \RR^n$.

\subsection{Proof of Theorem~\ref{thm:actor_convergence}}

\begin{proof}[Proof of Theorem~\ref{thm:actor_convergence}]
We decompose the derivative (in \(\tau\)) of $\mL_a^\tau = J^{\mu^\tau}[\alpha^\tau] - J^{\mu^\tau}[\alpha^{\mu^\tau,*}]$ into two parts
\begin{equation}
\pta\mL_a^\tau = \dfrac{\rd}{\rd \tau} \parentheses{ J^{\mu}[\alpha^\tau] - J^{\mu}[\alpha^{\mu,*}]} \Big|_{\mu=\mu^\tau} + \dfrac{\rd}{\rd \tau} \parentheses{J^{\mu^\tau}[\alpha] - J^{\mu^\tau}[\alpha^{\mu^\tau,*}]}\Big|_{\alpha=\alpha^\tau} =: (a\rom{1}) + (a\rom{2}),
\end{equation}
addressing the dependence on $\alpha^\tau$ and $\mu^\tau$ separately.

\noindent
\emph{Step 1.} We estimate $(a\rom{1})$ first. By the policy gradient dynamic \eqref{eq:actor_flow},
\begin{equation}
\begin{aligned}
(a\rom{1}) &= \inner{\fd{\mato}{\alpha}{J^{\mu^\tau}[\alpha^\tau]}}{\dfrac{\rd}{\rd \tau} \alpha^\tau}_{\mato} \\
& = -\beta_a \int_0^T \int_{\RR^d}  \na H\parentheses{t,x,\mu^\tau_t,\alpha^\tau(t,x),-\nx V\mat(t,x) }\tp \\
& \hspace{1in}  \na H\parentheses{t,x,\mu^\tau_t,\alpha^\tau(t,x),-\mG^\tau(t,x)}  \, \rho\mat(t,x) \,\rd x\,\rd t \\
& = -\frac12 \beta_a \inttx{\abs{\na H\parentheses{t,x,\mu^\tau_t,\alpha^\tau(t,x),-\nx V\mat(t,x) }}^2 \rho\mat(t,x) } \\
& \quad -\frac12 \beta_a \inttx{\abs{\na H\parentheses{t,x,\mu^\tau_t,\alpha^\tau(t,x),-\mG^\tau(t,x) } }^2 \rho\mat(t,x) } \\
& \quad + \frac12 \beta_a \int_0^T \int_{\RR^d} \left|\na H\parentheses{t,x,\mu^\tau_t,\alpha^\tau(t,x),-\nx V\mat(t,x)} \right. \\
& \hspace{1in} - \na H\parentheses{t,x,\mu^\tau_t,\alpha^\tau(t,x),-\mG^\tau(t,x) } \Big|^2 \rho\mat(t,x) \,\rd x \,\rd t \\
& =: \beta_a \parentheses{ -(a\rom{3}) - (a\rom{4}) + (a\rom{5}) }.
\end{aligned}
\end{equation}

Firstly, we show $(a\rom{3}) \ge c \mL_a^\tau$. 
We start with a technical definition. For any $\tau \ge 0$, $(t,x) \in [0,T] \times \RR^d$, define the local optimal control $\alpha^{\tau,\diamond}$ as 
\begin{equation}\label{eq:local_optimal}
\alpha^{\tau,\diamond}(t,x) := \argmax_{a\in\RR^n} H(t,x,\mu_t^\tau,a,-\nx V\mat(t,x), -\nx^2 V\mat(t,x)).
\end{equation}
We want to show that, there exists some constant $c > 0$ such that
\begin{equation}\label{eq:technical}
\norm{\alpha^\tau - \alpha^{\tau,\diamond}}_{\mato} \ge c \|\alpha^\tau - \alpha^{\mu^\tau,*}\|_{\mato},\ \forall \tau\geq 0.
\end{equation}

We prove \eqref{eq:technical} by contradiction, assuming that there exists an increasing sequence $\tau_k \to \infty$ such that
\begin{equation}\label{eq:assume_contrary}
\norm{\alpha^{\tau_k} - \alpha^{\tau_k,\diamond}}_\matk \le \frac1k \|\alpha^{\tau_k} - \alpha^{\mu^{\tau_k},*}\|_\matk,\ \forall k\in\mathbb{N}.
\end{equation}
With shorthand notations 
\begin{equation}\label{eq:technical_notation}
\alpha_k:=\alpha^{\tau_k},\, \mu^k:=\mu^{\tau_k},\, \alpha_k^* = \alpha^{\mu^k,*},\, V_k := V^\matk, \, V_k^* := V^{\mu^k,\alpha_k^*},\, \alpha_k^\diamond := \alpha^{\tau_k,\diamond},\, \norm{\cdot}_k := \norm{\cdot}_\matk,
\end{equation}
the inequality above becomes $\norm{\alpha_k - \alpha_k^\diamond}_k \le \frac1k \norm{\alpha_k - \alpha_k^*}_k$. With this condition, we can show that
\begin{equation}\label{eq:important_lemma}
\limsup_{k\to\infty} \int_{\RR^d}(V_k(t,x) - V_k^*(t,x)) \rho_0(x) \,\rd x = 0,\ \forall t \in [0,T],
\end{equation}
with its proof (motivated by \cite[Theorem 6.1]{yong2012stochastic}) left to Lemma~\ref{lem:technical} in Appendix~\ref{sec:technical}.
Since $V_k - V_k^*$ is non-negative by definition, setting $t=0$ yields $\lim_{k\to\infty}(J^{\mu^k}[\alpha_k] - J^{\mu^k}[\alpha_k^*])=0$. By Lemma~\ref{lem:modulus}, we get
\begin{equation}\label{eq:control_to_optimal}
\lim_{k\to\infty}\norm{\alpha_k - \alpha_k^*}_k = 0.
\end{equation}

As an intermediate step toward reaching contradiction, we prove that
\begin{equation}\label{eq:control_gradV}
\norm{\alpha_k^\diamond - \alpha_k^*}_k \le \dfrac{K}{\lam_H} \norm{\nx V_k - \nx V_k^*}_k.
\end{equation}
Since $\sigma$ does not depend on $\alpha$, by definition~ \eqref{eq:local_optimal}, we view $\alpha_k^\diamond(t,x)$ as an implicit function of $p = -\nx V_k(t,x)$ for any fixed tuple of $(t,x)$, with the functional relationship determined by the critical point equation
$$\na H(t,x,\mu^k_t,\alpha,p)\big|_{p=-\nx V_k(t,x)} = 0.$$
By the optimality of $\alpha_k^*$ (under measure $\mu^k$), for any $(t,x) \in [0,T] \times \RR^d$, $\alpha_k^*(t,x)$ maximizes the mapping
$\alpha \mapsto H(t,x,\mu^k_t, \alpha, -\nx V_k^*(t,x), -\nx^2 V_k^*(t,x)).$
Therefore, the same implicit function  evaluated at $p=-\nx V_k^*(t,x)$ provides $\alpha_k^*(t,x)$. 
Naturally, showing \eqref{eq:control_gradV} reduces to showing that this implicit function is Lipschitz in $p$ with Lipschitz constant $K / \lam_H$. 
Recall that \(H\) is \(\lam_H\)-strongly concave in \(\alpha\). 
By the implicit function theorem, the continuously differentiable implicit function $\alpha(p)$ globally exists.
Computing its Jacobian with respect to $p \in \RR^d$ yields
$$\nabla_p \alpha(p) = - \parentheses{\na^2 H(t,x,\mu^k_t,\alpha(p),p)}^{-1} \cdot \na b(t,x,\mu^k_t,\alpha(p)).$$
Since $\abs{\na b} \le K$ and $\| \parentheses{\na^2 H(t,x,\mu^k_t,\alpha(p),p)}^{-1}\|_2\leq \frac{1}{\lam_H}$,  we obtain $\abs{\nabla_p \alpha(p)} \le K / \lam_H$, which concludes the proof of  \eqref{eq:control_gradV}.

The final step toward reaching contradiction is to show the superlinear growth
\begin{equation}\label{eq:superlinear_growth}
\norm{\nx V_k - \nx V_k^*}_k \le C \norm{\alpha_k - \alpha_k^*}_k^{1+\chi},
\end{equation}
where $\chi = \frac{2}{d+5} > 0$. We leave the proof of~\eqref{eq:superlinear_growth} (motivated by \eqref{eq:value_gap}) to Lemma~\ref{lem:superlinear_growth} in Appendix~\ref{sec:superlinear_growth}.

At this point, we present the contradiction, which proves~\eqref{eq:technical}.
Using \eqref{eq:assume_contrary}, \eqref{eq:control_gradV} and \eqref{eq:superlinear_growth},
\begin{align*}
& \quad \norm{\alpha_k - \alpha_k^*}_k \le \norm{\alpha_k - \alpha_k^\diamond}_k + \norm{\alpha_k^\diamond - \alpha_k^*}_k \le \frac1k \norm{\alpha_k - \alpha_k^*}_k + C \norm{\alpha_k - \alpha_k^*}_k^{1+\chi},
\end{align*}
which contradicts \eqref{eq:control_to_optimal} as \(k\to\infty\).

Now, we return to showing $(a\rom{3}) \ge c \mL_a^\tau$ with the help of \eqref{eq:technical}:
\begin{equation}\label{eq:actor_term3}
\begin{aligned}
(a\rom{3})& = \frac12 \inttx{\abs{\na H\parentheses{t,x,\mu^\tau_t,\alpha^\tau(t,x),-\nx V\mat(t,x) }}^2 \rho\mat(t,x) }\\
& \ge \frac12 \inttx{\lam_H^2 \abs{\alpha^\tau(t,x) - \alpha^{\tau,\diamond}(t,x)}^2 \rho\mat(t,x) } = \frac12 \lam_H^2 \norm{\alpha^\tau - \alpha^{\tau,\diamond}}_\mato^2 \\
& \ge c \norm{\alpha^\tau - \alpha\mts}_\mato^2 \ge c_a \parentheses{J^{\mu^\tau}[\alpha^\tau] - J^{\mu^\tau}[\alpha^{\mu^\tau,*}]} = c_a \mL_a^\tau,
\end{aligned}
\end{equation}
where the first inequality is due to the mean value theorem and Assumption~\ref{assu:Hconcave}, and the last inequality comes from Lemma~\ref{lem:quadratic_growth}.

For $(a\rom{5})$, by Assumption~\ref{assu:basic},
\begin{equation}
\begin{aligned}
(a\rom{5})& = \frac{1}{2}\inttx{\abs{\na b(t,x,\mu_t,\alpha^\tau(t,x))\tp \parentheses{\nx V\mat(t,x) - \mG^\tau(t,x)}}^2 \rho\mat(t,x)} \\
& \le \frac{1}{2}K^2 \norm{\nx V\mat - \mG^\tau}^2_\mato.
\end{aligned}
\end{equation}

Combining estimations for $(a\rom{3})$ (cf. \eqref{eq:actor_term3}), $(a\rom{4})$ and $(a\rom{5})$ yields
\begin{equation}\label{eq:actor_term1}
(a\rom{1}) \le -c_a \beta_a \mL_a^\tau - \frac12 \beta_a \norm{\na H(t,x,\mu_t,\alpha^\tau(t,x), -\mG^\tau(t,x))}^2_\mato + \frac12 \beta_a K^2 \norm{\nx V\mat - \mG^\tau}_\mato^2.
\end{equation}

\noindent
\emph{Step 2.}
Next, we estimate $(a\rom{2})$. Since $\alpha^{\mu^\tau,*}$ minimizes $J^{\mu^\tau}[\cdot]$, by chain rule,
\begin{equation}\label{eq:actor_term2_optimal}
(a\rom{2})=\dfrac{\rd}{\rd\tau} \parentheses{J^{\mu^\tau}[\alpha] - J^{\mu^\tau}[\alpha^{\mu^\tau,*}]}\Big|_{\alpha=\alpha^\tau} = \dfrac{\rd}{\rd\tau} \parentheses{J^{\mu^\tau}[\alpha] - J^{\mu^\tau}[\alpha']}\Big|_{\alpha=\alpha^\tau, \alpha'=\alpha^{\mu^\tau,*}}.
\end{equation}
We claim that
\begin{equation}\label{eq:actor_distribution_update}
\dfrac{\rd}{\rd\tau} \parentheses{J^{\mu^\tau}[\alpha] - J^{\mu^\tau}[\alpha']}\Big|_{\alpha=\alpha^\tau, \alpha'=\alpha^{\mu^\tau,*}} \le C \beta_\mu \big\|\alpha^\tau - \alpha^{\mu^\tau,*}\big\|_\mato^2,
\end{equation}
the proof of which is deferred to Lemma~\ref{lem:actor_distribution_update} in Appendix~\ref{sec:actor_distribution_update}.
This inequality demonstrates the impact of the distribution flow on the actor loss function. 
Combining \eqref{eq:actor_term2_optimal}, \eqref{eq:actor_distribution_update} and Lemma~\ref{lem:modulus} yields
\begin{equation}\label{eq:actor_term2}
(a\rom{2}) \le C \beta_\mu \big\|\alpha^\tau - \alpha^{\mu^\tau,*} \,\big\|_\mato^2 \le \beta_\mu C_a \parentheses{J^{\mu^\tau}[\alpha^\tau] - J^{\mu^\tau}[\alpha^{\mu^\tau,*}]} = \beta_\mu C_a \mL_a^\tau.
\end{equation}

Finally, combining estimations for $(a\rom{1})$ (cf. \eqref{eq:actor_term1}) and $(a\rom{2})$ (cf. \eqref{eq:actor_term2}) concludes the proof.
\end{proof}

\subsection{Convergence of the gap for value function}\label{sec:technical}
\begin{lem}\label{lem:technical}
Under the notations of~\eqref{eq:technical_notation},
if the conditions of Theorem~\ref{thm:actor_convergence} hold and
\begin{equation}\label{eq:assume_contrary_lem}
\norm{\alpha_k - \alpha_k^\diamond}_k \le \frac1k \norm{\alpha_k - \alpha_k^*}_k,\ \forall k\geq 1,
\end{equation}
then we claim that
\begin{equation}\label{eq:important_lemma_lem}
\limsup_{k\to\infty} \int_{\RR^d}(V_k(t,x) - V_k^*(t,x)) \rho_0(x) \,\rd x = 0, \ \forall t \in [0,T].
\end{equation}
\end{lem}

\begin{proof}
We prove this lemma by contradiction. 
Since the proof is very long, we split it into 7 steps, introducing the strategy before diving into details. 

In \emph{Step 1}, we assume the existence of some $\bt$, for which \eqref{eq:important_lemma_lem} fails. 
We restrict our analysis within a small time interval $[t_-,t_+]$ containing $\bt$, apply the doubling of variable technique, and define a function $\Phi_k$, with its maximum attained by the tuple $\txsyk$.
In \emph{Step 2}, we estimate $\txsyk$ using the Schauder estimation for value functions. 
In \emph{Step 3}, we show that both $t_k$ and $s_k$ are not close to $t_-$ or $t_+$. 

In \emph{Step 4}, we consider $\hPk$, a perturbed version of $\Phi_k$, whose maximum is attained by $\htxsyk$. 
In \emph{Step 5}, we obtain a critical point system for $\hPk$ at its maximum. 
In \emph{Step 6}, we carry out estimations for the critical point system. 
In \emph{Step 7}, we integrate with respect to all local perturbations in \emph{Step 4} and reach a contradiction.

\begin{figure}[!ht]
\centering
\includegraphics[width=0.5\linewidth]{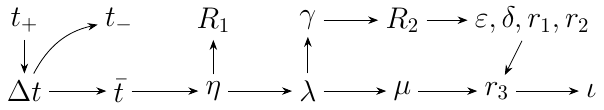}
\caption{A directed graph describing parameter dependencies.}
\label{fig:parameters}
\end{figure}
Throughout the proof, we will frequently use the properties of Moreau envelopes presented in Section~\ref{sec:Moreau}. Figure~\ref{fig:parameters} describes parameter dependencies in the proof through a directed graph. For example, an arrow from $r_3$ to $\io$ indicates that the choice of $\io$ potentially depends on $r_3$ and its ancestor nodes, including $\mu$, $\lam$, $\varepsilon$, $R_2$, etc. 
Without specification, $C$ and $c$ denote positive constants that only depend on $d$, $K$, $T$, $\sigma_0$, $\lam_H$, which are uniform with respect to $k$ and all the parameters in Figure~\ref{fig:parameters}. We may denote some of these constants by $C_1$, $c_1$, etc., for the specification of other parameters.

\medskip
\noindent \emph{Step 1}. Firstly, we reformulate the problem. Define 
$$h_k(t) := \int_{\RR^d}(V_k(t,x) - V_k^*(t,x)) \rho_0(x) \,\rd x.$$ 
By the optimality of $V_k^*$, $h_k(t) \ge 0,\ \forall t\in[0,T], \ k\ge1$.
By the quadratic growth and local Lipschitz property of value functions in Lemma~\ref{lem:value_growth}, $\{h_k(t)\}_{k=1}^\infty$ is uniformly bounded and uniformly Lipschitz. 
By the Arzel\`a--Ascoli theorem, $\{h_k(t)\}_{k=1}^\infty$ has a subsequence that converges uniformly on \([0,T]\).
Therefore, it suffices to show that any uniformly convergent subsequence of $\{h_k(t)\}_{k=1}^\infty$ must converge to $0$. 

We prove by contradiction, assuming that $\{h_k(t)\}_{k=1}^\infty$ has a subsequence that converges uniformly to a nonzero function $h(t)$.
For simplicity of notations, we still use \(\{h_k(t)\}_{k=1}^\infty\) to denote this subsequence in the following context without specification.
We remark that, the factor $1/k$ in the condition \eqref{eq:assume_contrary_lem} can be replaced by any positive sequence that decreases to $0$. 
Hence, using the same index $k$ for the subsequence causes no loss of generality.

Clearly, $h(t)\ge0$, $h(T)=0$, and $h$ is Lipschitz continuous. 
Since $h$ is not constantly zero, $t_+ := \inf \{ t\in[0,T] ~|~ h(s)=0,\ \forall s\in[t,T] \}$ is well defined and $t_+ > 0$. 
For a fixed value of $\dt \in (0,t_+ \wedge 1)$, which will be later specified, define $t_- := t_+ - \dt$, $I_t := [t_-,t_+]$. 
We pick $\bt \in (t_-, t_+)$ such that $h(\bt) > 0$ and define
$$3\eta := h(\bt) = \lim_{k\to\infty} \int_{\RR^d} \parentheses{V_k(\bt,x) - V_k^*(\bt,x)}\rho_0(x) \,\rd x.$$
There exists a further subsequence of the convergent subsequence $\{h_k(t)\}_{k=1}^\infty$ (which is still denoted by $\{h_k(t)\}_{k=1}^\infty$), such that
$$\int_{\RR^d} \parentheses{V_k(\bt,x) - V_k^*(\bt,x)}\rho_0(x) \,\rd x \ge \frac83 \eta,\ \forall k\ge 1.$$
Since $|V_k|$ and $|V_k^*|$ grow at most quadratically in $|x|$ (see Lemma~\ref{lem:value_growth}) and $\rho_0(x)$ decays exponentially in $|x|$, we claim that: there exists $R_1 > 0$ and $|\bxk| \le R_1,\ \forall k\ge1$ such that
\begin{equation}\label{eq:eta73}
V_k(\bt,\bxk) - V_k^*(\bt,\bxk) \ge \frac73 \eta,\ \forall k\ge 1.
\end{equation}
Otherwise, if such \(R_1\) and \(\bxk\) do not exist, then for some \(k\), \(V_k(\bt,\bxk) - V_k^*(\bt,\bxk) < \frac73 \eta\). This implies that
\begin{align*}
& \quad \int_{\RR^d} \parentheses{V_k(\bt,x) - V_k^*(\bt,x)}\rho_0(x) \,\rd x \\
& = \int_{B_{R_1}} \parentheses{V_k(\bt,x) - V_k^*(\bt,x)}\rho_0(x) \,\rd x + \int_{B_{R_1}^c} \parentheses{V_k(\bt,x) - V_k^*(\bt,x)}\rho_0(x) \,\rd x \\
& \le \frac73 \eta+ \int_{B_{R_1}^c} C (1+|x|^2)\rho_0(x) \,\rd x  \to \frac73 \eta 
< \frac83 \eta \quad \text{as}\quad R_1\to\infty,
\end{align*}
which contradicts the property of the further subsequence stated above.

Next, we apply the doubling of variable method \cite[Theorem 8.3]{crandall1992user} and define a barrier function $\vp: I_t \times \RR^d \times I_t \times \RR^d \ni(t,x,s,y)\to \vp(t,x,s,y)\in\RR$ as follows:
\begin{equation}\label{eq:barrier}
\vp\txsy := \gamma (t_+-t+\dt) |x|_l^l + \gamma (t_+-s+\dt) |y|_l^l + \frac{1}{2\ve}|t-s|^2 + \frac{1}{2\delta}|x-y|^2 + \frac{\lam}{t-t_-} + \frac{\lam}{s-t_-}.
\end{equation}
Here, $\gamma,\ve,\delta,\lam \in (0,1)$ are parameters, whose values will be later specified (cf. Figure~\ref{fig:parameters}).
$|x|_l^l := \sum_{i=1}^d |x_i|^l$ denotes the $l$-Euclidean norm  and $l > 2$ is a constant. 
In the proof, we assume that \(l\) is an even integer, e.g., $l=4$, for simplicity.
Nevertheless, the proof remains valid for a general value of $l$ and the argument can be extended to general growth conditions of value functions.

We define a sequence of functions $\Phi_k: I_t \times \RR^d \times I_t \times \RR^d \to \RR$ as follows:
\begin{equation}\label{eq:Phik}
\Phi_k\txsy :=\Vki(t,x) - \Vkis(s,y) - \vp\txsy,
\end{equation}
where for any $(t,x) \in [0,T] \times \RR^d$,
\begin{align*}
\Vki(t,x) &:= \sup_{(s,y) \in [0,T] \times \RR^d} \Big[ V_k(s,y) - \frac{1}{2\io} \parentheses{ |t-s|^2 + |x-y|^2 } \Big], \\
\Vkis(t,x) &:= \inf_{(s,y) \in [0,T] \times \RR^d} \Big[ V_k^*(s,y) + \frac{1}{2\io} \parentheses{ |t-s|^2 + |x-y|^2 } \Big],
\end{align*}
denote the Moreau envelopes for $V_k$ and $V_k^*$ respectively, with the value of $\io$ to be later specified (cf. Figure~\ref{fig:parameters}).
Since $l>2$ and the value functions have quadratic growth,
$\lim_{|x|\vee|y| \to\infty} \Phi_k\txsy = -\infty,\ \forall k\geq 1.$
Therefore, $\Phi_k$ attains its maximum at some point $\txsyk \in I_t \times \RR^d \times I_t \times \RR^d$. 

Since $\Vki$ is semiconvex and $\Vkis$ is semiconcave (cf. Lemma~\ref{lem:Moreau}), the function $\Vki(t,x) - \Vkis(s,y)$ is semiconvex in $\txsy$. Therefore, $\Vki$ and $\Vkis$ are twice differentiable almost everywhere \cite{aleksandorov1939almost}. 
We remark that, differentiability in time is the main reason why we are using the Moreau envelope of the value function. 
In contrast, standard Schauder estimate (see the next step) only provides $C^1$ H\"older continuity in time.

\medskip
\noindent \emph{Step 2}. We present estimations for $\txsyk$.
Since $\Phi_k\txsyk \ge \Phi_k(t_+,0,t_+,0)$,
$$\Vki\txk - \Vkis\syk - \vp\txsyk \ge \Vki(t_+,0) - \Vkis(t_+,0) - \frac{2\lam}{\dt},$$
which implies
\begin{align*}
&\quad \gamma (t_+-t_k+\dt) |x_k|_l^l + \gamma (t_+-s_k+\dt) |y_k|_l^l + \frac{1}{2\ve}|t_k-s_k|^2 + \frac{1}{2\delta}|x_k-y_k|^2 + \frac{\lam}{t_k-t_-} + \frac{\lam}{s_k-t_-} \\
&\le \Vki\txk - \Vkis\syk - \Vki(t_+,0) + \Vkis(t_+,0) + \frac{2\lam}{\dt}.
\end{align*}
Applying Lemma~\ref{lem:value_growth} and $t_+-t_k+\dt$, $t_+-s_k+\dt \ge \dt$, we obtain
\begin{equation}\label{eq:technical_step2_temp}
\gamma\dt\parentheses{|x_k|_l^l + |y_k|_l^l} + \frac{1}{2\ve}|t_k-s_k|^2 + \frac{1}{2\delta}|x_k-y_k|^2 + \frac{\lam}{t_k-t_-} + \frac{\lam}{s_k-t_-} \le C\parentheses{1 + |x_k|^2 + |y_k|^2} + \frac{2\lam}{\dt}.
\end{equation}
Take \(\lam\) such that $\lam \le \dt$ (cf. Figure~\ref{fig:parameters}). 
Since $\gamma\dt\parentheses{|x_k|_l^l + |y_k|_l^l} \le C\parentheses{1 + |x_k|^2 + |y_k|^2}$,
\begin{equation}\label{eq:xkyk_bound}
|x_k|, |y_k| \le C_0 (\gamma\dt)^{-\frac{1}{l-2}}.
\end{equation}
Denote $R_2 := C_0 (\gamma\dt)^{-\frac{1}{l-2}} + 2$ so that $|x_k|,|y_k| \le R_2-2$. Substituting \eqref{eq:xkyk_bound} back into \eqref{eq:technical_step2_temp} yields
\begin{equation}\label{eq:tk_bound}
\frac{\lam}{t_k-t_-},\frac{\lam}{s_k-t_-} \le C (\gamma\dt)^{-\frac{2}{l-2}} \quad \Rightarrow \quad (t_k-t_-), (s_k-t_-) \ge c_1 \lam (\gamma\dt)^{\frac{2}{l-2}},
\end{equation}
where we record the constant $c_1$ for the specification of the parameters.

From $2\Phi_k\txsyk \ge \Phi_k(t_k,x_k,t_k,x_k) + \Phi_k(s_k,y_k,s_k,y_k)$, we conclude that
$$\Vki\txk - \Vki\syk + \Vkis\txk - \Vkis\syk \ge \frac{1}{\ve}|t_k-s_k|^2 + \frac{1}{\delta}|x_k-y_k|^2.$$
Using the local Lipschitz property of the Moreau envelope (Lemma~\ref{lem:Moreau}),
\begin{align*}
& \quad \frac{1}{\ve+\delta}(|t_k-s_k|^2 + |x_k-y_k|^2) \le  \frac{1}{\ve}|t_k-s_k|^2 + \frac{1}{\delta}|x_k-y_k|^2 \\
& \le C (1 + |x_k|^2 + |y_k|^2)\parentheses{|t_k-s_k| + |x_k-y_k|} \le C (\gamma\dt)^{-\frac{2}{l-2}} \parentheses{|t_k-s_k| + |x_k-y_k|}.
\end{align*}
This implies
\begin{equation}\label{eq:bound_tsxy_diff}
|t_k-s_k| + |x_k-y_k| \le C_1 (\gamma\dt)^{-\frac{2}{l-2}} (\ve+\delta),
\end{equation}
where we record the constant $C_1$ for later specification of $\ve$ and $\delta$. In addition, \(\frac{1}{\ve}|t_k-s_k|^2 + \frac{1}{\delta}|x_k-y_k|^2 \le C (\gamma\dt)^{-\frac{4}{l-2}} (\ve+\delta)\).


Next, we present a Schauder estimation for the value functions $V_k$ and $V_k^*$ within the compact domain $[0,T] \times B_{R_2}$. Denote by $\zeta \in (0,1)$ the H\"older constant. For a function $V:[0,T] \times B_{R_2} \to \RR$, the parabolic H\"older semi-norm and H\"older norm (see \cite[Chapter 4]{ladyvzenskaja1988linear} for details) are defined as follows:
$$[V]_{\zeta/2,\zeta} := \sup_{(t,x)\neq(s,y)} \dfrac{|V(t,x) - V(s,y)|}{(|t-s|^\frac12 + |x-y|)^\zeta}, \quad [V]_{1+\zeta/2,2+\zeta} := [\pt V]_{\zeta/2,\zeta} + \sum_{i=1}^d [\partial_{x_i}V]_{\zeta/2,\zeta} + \sum_{i,j=1}^d [\partial_{x_i} \partial_{x_j} V]_{\zeta/2,\zeta},$$
$$\norm{V}_{C^{\zeta/2,\zeta}} := \norm{V}_{L^\infty} + [V]_{\zeta/2,\zeta}, ~~~ \norm{V}_{C^{1+\zeta/2,2+\zeta}} := \norm{V}_{L^\infty} + \norm{\pt V}_{L^\infty} + \norm{\nx V}_{L^\infty} + \norm{\nx^2 V}_{L^\infty} + [V]_{1+\zeta/2,2+\zeta}.$$

Denote $f_k(t,x) := f(t,x,\mu^k_t,\alpha_k(t,x))$ and define $b_k$, $\sigma_k$ similarly. 
For any $t,s \in [0,T]$, $x,y \in B_{R_2}$, by Assumption~\ref{assu:basic} and Assumption~\ref{assu:flow_in_class},
\begin{align*}
& \quad |f_k(t,x) - f_k(s,y)| \\
& \le K \parentheses{1+R_2^2+W_2(\mu^k_t,\delta_0)^2\vee W_2(\mu^k_s,\delta_0)^2 +  |\alpha_k(s,y)|^2\vee |\alpha_k(t,x)|^2} |t-s|^{\frac12}  + K (1+R_2+\\
&\quad W_2(\mu^k_t,\delta_0)\vee W_2(\mu^k_s,\delta_0)
+ |\alpha_k(t,x)| \vee |\alpha_k(s,y)|) \parentheses{|x-y|+W_2(\mu^k_t,\mu^k_s) + |\alpha_k(t,x) - \alpha_k(s,y)|} \\
& \le C (1+R_2^2) |t-s|^{\frac12} + C (1+R_2)(|x-y| + |t-s|^\frac12)  \le C_{R_2} \parentheses{|t-s|^{\frac12} + |x-y| },
\end{align*}
where $C_{R_2}>0$ is a constant that depends on $R_2$. Since $B_{R_2}$ is bounded, this implies $[f_k]_{\zeta/2,\zeta} \le C_{R_2}$ within $[0,T] \times B_{R_2}$. Similar estimations also hold for $b_k$ and $\sigma_k$. Therefore, applying standard H\"older estimation for linear parabolic equations \cite[Section 4.5]{ladyvzenskaja1988linear}, we have 
\begin{equation}\label{eq:bound_Vk_Holder}
\norm{V_k}_{C^{1+\zeta/2,2+\zeta}([0,T] \times B_{R_2})} \le C_{R_2}.
\end{equation}
Additionally, applying the Schauder estimation for the HJB equation \cite{mou2019remarks} yields
\begin{equation}\label{eq:bound_Vkstar_Holder}
\norm{V^*_k}_{C^{1+\zeta/2,2+\zeta}([0,T] \times B_{R_2})} \le C_{R_2}.
\end{equation}

\medskip
\noindent \emph{Step 3}. We show that $t_k$, $s_k$ are not close to $t_+$ and $t_-$. Since $\Phi_k(\bt,\bxk,\bt,\bxk) \le \Phi_k\txsyk$,
\begin{equation}\label{eq:technical_step3_temp1}
\Vki(\bt,\bxk) - \Vkis(\bt,\bxk) - 2\gamma(t_+-\bt+\dt)|\bxk|_l^l  - \frac{2\lam}{\bt-t_-} \le \Phi_k\txsyk.
\end{equation}
Since $\Vki(\bt,\bxk) \ge V_k(\bt,\bxk)$ and $\Vkis(\bt,\bxk) \le V^*_k(\bt,\bxk)$, \eqref{eq:eta73} implies
\begin{equation}\label{eq:technical_step3_ineq1}
\Vki(\bt,\bxk) - \Vkis(\bt,\bxk) \ge \frac73\eta,\ \forall k\geq 1.
\end{equation}
We set $\gamma$ and $\lam$ small enough such that (cf. Figure~\ref{fig:parameters})
\begin{equation}\label{eq:technical_step3_ineq2}
4 \gamma \dt R_1^l \le \frac13 \eta \quad \text{and} \quad 2\lam \le \frac13 \eta \,(\bt - t_-).
\end{equation}
Substituting \eqref{eq:technical_step3_ineq1} and \eqref{eq:technical_step3_ineq2} into \eqref{eq:technical_step3_temp1} yields
\begin{equation}\label{eq:technical_step3_temp2}
\begin{aligned}
&\frac53\eta \le \Phi_k\txsyk = \Vki\txk - \Vkis\syk - \vp\txsyk \\
& \hspace{-0.5em}\le V_k\txk - V_k^*\syk + C\,\io R_2^4 - \gamma \dt \parentheses{|x_k|^l_l + |y_k|^l_l} - \frac{1}{2\ve}|t_k-s_k|^2 - \frac{1}{2\delta}|x_k-y_k|^2 \\
& \hspace{-0.5em}\le V_k(t_+,x_k) - V_k^*(t_+,y_k) + C(t_+-t_k)(1+|x_k|^2) + C(t_+-s_k)(1+|y_k|^2) + C_2\,\io R_2^4 - \frac{1}{2\delta}|x_k-y_k|^2,
\end{aligned}
\end{equation}
where the second inequality follows from \eqref{eq:Moreau_result2} and the last inequality follows from Lemma~\ref{lem:value_growth}.
Here, we record the constant $C_2$ and set $\io$ to be small enough (cf. Figure~\ref{fig:parameters}) such that
\begin{equation}\label{eq:technical_iota}
C_2\,\io R_2^4 \le \frac13 \eta.
\end{equation}

For the sequence $\{V_k(t_+,x_k) - V_k^*(t_+,y_k)\}_{k=1}^\infty$ that appears in \eqref{eq:technical_step3_temp2}, we claim that: there exists a subsequence (which is still denoted by index \(k\)) such that
\begin{equation}\label{eq:technical_step3_claim}
V_k(t_+,y_k) - V_k^*(t_+,y_k) \le \frac13 \eta ,\  \forall k\ge 1.
\end{equation}
We prove this argument by contradiction.
If the claim does not hold, $\limsup_{k\to\infty} V_k(t_+,y_k) - V_k^*(t_+,y_k) \ge \frac13 \eta$ so that we can extract a subsequence (which is still denoted by index \(k\)) such that
\begin{equation}\label{eq:technical_step3_Vkgap}
V_k(t_+,y_k) - V_k^*(t_+,y_k) \ge \frac14 \eta ,\  \forall k\ge1.
\end{equation}
Let $r_1 := \min\{ \eta / (16C R_2) ,1\}$, where $C$ is the constant in $|\nx V(t,x)| \le C(1+|x|)$ within Lemma~\ref{lem:value_growth}. Since $|y_k| \le R_2-2$, for any $x$ such that $|x-y_k| \le r_1$, we have $|x| \le R_2-1$ and
$$|V_k(t_+,x) - V_k(t_+,y_k)|,|V_k^*(t_+,x) - V_k^*(t_+,y_k)| \le C R_2 r_1 \le \frac{1}{16}\eta.$$
Substituting this into \eqref{eq:technical_step3_Vkgap} yields: for any $x$ such that $|x-y_k| \le r_1$,
$V_k(t_+,x) - V_k^*(t_+,x) \ge \frac18 \eta,\ \forall k\ge1.$
Integrating both sides yields
\begin{align*}
h_k(t_+) &= \int_{\RR^d} \parentheses{V_k(t_+,x) - V_k^*(t_+,x)} \rho_0(x) \,\rd x \ge \int_{|x-y_k| \le r_1} \parentheses{V_k(t_+,x) - V_k^*(t_+,x)} \rho_0(x) \,\rd x \\
& \ge \frac18 \eta\int_{|x-y_k| \le r_1}  \,(2\pi)^{-d/2} \exp(-R_2^2/2) \,\rd x = \abs{B_{r_1}} \frac18 \eta \,(2\pi)^{-d/2} \exp(-R_2^2/2) > 0,
\end{align*}
which contradicts $h(t_+) = \lim_{k\to\infty}h_k(t_+) = 0$. Therefore, the claim~\eqref{eq:technical_step3_claim} is true.

In the following context, we take the subsequence such that \eqref{eq:technical_step3_claim} holds, while maintaining the notation of the index as \(k\). 
Substituting \eqref{eq:technical_iota} and \eqref{eq:technical_step3_claim} into \eqref{eq:technical_step3_temp2} yields
\begin{equation}\label{eq:technical_step3_temp3}
\begin{aligned}
\eta & \le V_k(t_+,x_k) - V_k(t_+,y_k) + C(t_+-t_k)(1+|x_k|^2) + C(t_+-s_k)(1+|y_k|^2) - \frac{1}{2\delta}|x_k-y_k|^2 \\
& \le C R_2 |x_k-y_k| + C R_2^2 [(t_+-t_k)+(t_+-s_k)] - \frac{1}{2\delta}|x_k-y_k|^2 \\
& \le \frac12 C^2 R_2^2 \delta + C R_2^2 [(t_+-t_k)+(t_+-s_k)],
\end{aligned}
\end{equation}
where we used Lemma~\ref{lem:value_growth} and $|x_k|,|y_k| \le R_2-2$. Set $\delta$ to be small enough (cf. Figure~\ref{fig:parameters}) such that $\frac12 C^2 R_2^2 \delta \le \frac13 \eta$. Then
\begin{equation}\label{eq:technical_step3_temp4}
\frac23\eta \le  C R_2^2 [(t_+-t_k)+(t_+-s_k)] ~~ \Rightarrow ~~ (t_+-t_k)\vee(t_+-s_k) \ge \frac{c_2 \eta}{R_2^2}. 
\end{equation}
We record the constant $c_2$ for specification of parameters. 
Recall from~\eqref{eq:bound_tsxy_diff} that $|t_k-s_k| \le C_1 (\gamma\dt)^{-\frac{2}{l-2}} (\ve+\delta)$. 
By setting $\ve$, $\delta$ small enough such that $\ve,\delta \le \frac14 C_1^{-1} (\gamma\dt)^{\frac{2}{l-2}} c_2 \eta / R_2^2$, we get $|t_k-s_k| \le (c_2 \eta) / (2R_2^2)$ and \((t_+-t_k), (t_+-s_k) \ge (c_2 \eta)/ (2R_2^2)\).
Recall from~\eqref{eq:tk_bound} that $(t_k-t_-), (s_k-t_-) \ge c_1 \lam (\gamma\dt)^{\frac{2}{l-2}}$. 
Therefore, $\txsyk$ is an interior point of $I_t \times B_{R_2} \times I_t \times B_{R_2}$. By denoting $r_2 := \frac12 \min \curlybra{c_1 \lam (\gamma\dt)^{\frac{2}{l-2}}, c_2 \eta / (2R_2^2), 1}$,
\begin{equation}\label{eq:tk_boundary_distance}
t_k-t_-,\, s_k-t_-,\,t_+-t_k,\, t_+-s_k \ge 2r_2.
\end{equation}

\medskip
\noindent \emph{Step 4}. In this step, we introduce perturbation to the system. Let $\mu>0$, whose value will be later specified.
The mapping
$\txsy \mapsto \Phi_k\txsy - \frac{\mu}{2}|\txsy-\txsyk|^2$
attains a strict maximum at $\txsyk$. For $\qp \in \RR\times\RR^d\times\RR\times\RR^d$, define
\begin{align*}
\hPk\txsy &:= \Phi_k\txsy - \frac{\mu}{2}|\txsy-\txsyk|^2 \\
& \quad \quad + \inner{\txsy-\txsyk}{\qp},
\end{align*}
whose maximum is attained by $\htxsyk$. Then, $\htxsyk$ must lie in the set
$$\curlybra{ \txsy ~\big|~ \frac{\mu}{2}\abs{\txsy-\txsyk}^2 \le \inner{\txsy-\txsyk}{\qp}},$$
which implies
$\abs{\htxsyk-\txsyk} \le \frac{2}{\mu} \abs{\qp}.$

Conversely, we establish an upper bound of $\abs{\qp}$ in terms of $\abs{\htxsyk-\txsyk}$. Consider $\htxsyk$ such that
\begin{equation}\label{eq:xhat_bound_r3}
\abs{\htxsyk-\txsyk} \le r_3,
\end{equation}
where $0<2r_3\le r_2$ (cf. Figure~\ref{fig:parameters}) will be later specified.  Due to \eqref{eq:tk_boundary_distance}, this guarantees 
\begin{equation}\label{eq:tkhat_boundary_distance}
\htk,\hsk \in [t_-+1.5r_2, t_+-1.5r_2].
\end{equation}
Recall that \(r_2\leq \frac{1}{2}\), which implies \(r_3\leq 1\) and $|\hxk|,|\hyk| \le R_2-1$. The optimality of $\htxsyk$ provides
\begin{equation}\label{eq:Phikhat_critical_point}
0 = \nabla \hPk\htxsyk = \nabla \Phi_k\htxsyk - \mu (\htxsyk - \txsyk) + \qp,
\end{equation}
where the gradient $\nabla$ is taken with respect to $\txsy$. 
The critical point equation \eqref{eq:Phikhat_critical_point} expresses $\qp$ in terms of $\htxsyk$.
In order to bound $|\qp|$ in terms of $r_3$ (cf. \eqref{eq:xhat_bound_r3}), our next task is to estimate $\abs{\nabla \Phi_k\htxsyk}$. 
Using $\nabla \Phi_k\txsyk = 0$,
\begin{equation}\label{eq:gradPhik_split_terms}
\begin{aligned}
& \quad \abs{\nabla \Phi_k\htxsyk} = \abs{\nabla \Phi_k\htxsyk - \nabla \Phi_k\txsyk} \\
& \le \abs{\pt \Vki\htxk - \pt \Vki(t_k,x_k)} + \abs{\nx \Vki\htxk - \nx \Vki(t_k,x_k)} + \abs{\ps \Vkis\hsyk - \ps \Vkis(s_k,y_k)} \\
& \quad + \abs{\ny \Vkis\hsyk - \ny \Vkis(s_k,y_k)} + \abs{\nabla \vp\htxsyk - \nabla \vp\txsyk}.
\end{aligned}
\end{equation}

We estimate each term on the right-hand side of \eqref{eq:gradPhik_split_terms}.
Denote by $\htxkp$ where the supremum within $\Vki\htxk = \sup_{(t',x')} \Big[  V_k(t',x')- \frac{1}{2\io}(|\htk - t'|^2 + |\hxk - x'|^2 ) \Big]$ is attained. 
Similarly, denote by $\hsykp$ where the infimum within
$\Vkis\hsyk = \inf_{(s',y')} \Big[  V_k^*(s',y') +  \frac{1}{2\io}(|\hsk - s'|^2 + |\hyk - y'|^2 ) \Big]$
is attained. 
The existence of $\htxkp$ and $\hsykp$ follows from Lemma~\ref{lem:Moreau}.
By \eqref{eq:Moreau_result1},
\begin{equation}\label{eq:step4_hatxk_bound}
\begin{aligned}
& 2R_2 |\htk-\htk'| + |\hxk-\hxk'| \le C_3\,\io\, R_2(2R_2^2+1), \quad |\hxk'|\le R_2, \\
& 2R_2 |\hsk-\hsk'| + |\hyk-\hyk'| \le C_3\,\io\, R_2(2R_2^2+1), \quad |\hyk'|\le R_2,
\end{aligned}
\end{equation}
where $C_3$ is recorded for later specification of $\io$. Set $\io$ to be small enough (cf. Figure~\ref{fig:parameters}) such that $C_3\,\io\, R_2(2R_2^2+1) \le 2r_3$. This implies 
\begin{equation}\label{eq:htk_htkp_diff_bound_r3}
|\htk-\htk'|,|\hsk-\hsk'|, |\hxk-\hxk'|, |\hyk-\hyk'|\le r_3.
\end{equation}
Combining with \eqref{eq:tk_boundary_distance} and \eqref{eq:xhat_bound_r3} (and recall $2r_3 \le r_2$) yields \(\htk', \hsk' \in [t_-+r_2, t_+-r_2]\).
By \eqref{eq:Moreau_result5} in Lemma~\ref{lem:Moreau},
\begin{equation}\label{eq:step4_Vhat_derivative}
\begin{aligned}
&\pt \Vki\htxk = \pt V_k\htxkp, ~~~~~  \nx \Vki\htxk = \nx V_k\htxkp, ~~~~ \nx^2 \Vki\htxk \ge \nx^2 V_k\htxkp,\\
&\ps \Vkis\hsyk = \ps V_k^*\hsykp, ~~   \ny \Vkis\hsyk = \ny V_k^*\hsykp, ~~ \ny^2 \Vkis\hsyk \le \ny^2 V_k^*\hsykp.
\end{aligned}
\end{equation}
Denote by $(t_k',x_k')$ where the supremum within
$\Vki(t_k,x_k) = \sup_{(t',x')} \Big[  V_k(t',x')- \frac{1}{2\io}(|t_k - t'|^2 + |x_k - x'|^2 ) \Big]$
is attained. Similarly, denote by $(s_k',y_k')$ where the infimum within
$\Vkis(s_k,y_k) = \inf_{(s',y')} \Big[  V_k^*(s',y') +  \frac{1}{2\io}(|s_k - s'|^2 + |y_k - y'|^2 ) \Big]$
is attained. By Lemma~\ref{lem:Moreau},
\begin{equation}\label{eq:step4_xk_bound}
\begin{aligned}
&2(R_2-1) |t_k-t_k'| + |x_k-x_k'| \le C_3\,\io\, (R_2-1)(2(R_2-1)^2+1), \quad |x_k'|\le R_2-1, \\
&2(R_2-1) |s_k-s_k'| + |y_k-y_k'| \le C_3\,\io\, (R_2-1)(2(R_2-1)^2+1), \quad |y_k'|\le R_2-1,
\end{aligned}
\end{equation}
\begin{equation}\label{eq:step4_V_derivative}
\begin{aligned}
&\pt \Vki\txk = \pt V_k\txkp, \quad~~ \nx \Vki\txk = \nx V_k\txkp, \\
&\ps \Vkis\syk = \ps V_k^*\sykp, ~~~ \ny \Vkis\syk = \ny V_k^*\sykp.
\end{aligned}
\end{equation}
Note that \eqref{eq:step4_xk_bound} implies \(|t_k-t_k'|, |s_k-s_k'| \le r_3\), which leads to \(t_k',s_k' \in [t_- + 1.5r_2, t_+ - 1.5r_2]\).
Therefore, using \eqref{eq:step4_Vhat_derivative}, \eqref{eq:step4_V_derivative} and the H\"older norm bound for $V_k$ \eqref{eq:bound_Vk_Holder},
\begin{equation}\label{eq:step4_ptVk_diff}
\begin{aligned}
& \quad \abs{\pt \Vki\htxk - \pt \Vki(t_k,x_k)} = \abs{\pt V_k\htxkp - \pt V_k\txkp} \le C_{R_2} \big( |\htk'-t_k'|^{\frac12} + |\hxk'-x_k'| \big)^\zeta \\
& \le C_{R_2} \Big[ \big( |\htk'-\htk| + |\htk - t_k| + |t_k-t_k'| \big)^{\frac12} + \big( |\hxk'-\hxk| + |\hxk - x_k| + |x_k-x_k'| \big) \Big]^\zeta \\
& \le C_{R_2} \Big[ \parentheses{C_3 \, \io \,(2R_2^2+1)/2 + r_3 + C_3 \, \io \,(2(R_2-1)^2+1)/2}^\frac12 \\
& \hspace{0.6in} + C_3 \, \io \,R_2(2R_2^2+1) + r_3 + C_3 \, \io \,(R_2-1)(2(R_2-1)^2+1)\Big]^\zeta \le C_{R_2} \big( \io^{\frac{\zeta}{2}} + r_3^{\frac{\zeta}{2}} \big).
\end{aligned}
\end{equation}
where the third inequality is from \eqref{eq:xhat_bound_r3} \eqref{eq:step4_hatxk_bound}, and \eqref{eq:step4_xk_bound}. Additionally, \eqref{eq:step4_V_derivative} and \eqref{eq:bound_Vk_Holder} imply
\begin{equation}\label{eq:step4_nxVk_diff}
\begin{aligned}
& \quad \abs{\nx \Vki\htxk - \nx \Vki(t_k,x_k)} = \abs{\nx V_k\htxkp - \nx V_k\txkp} \\
& \le C_{R_2} \big( |\htk'-t_k'|^{\frac12} + |\hxk'-x_k'| \big)^\zeta \le C_{R_2} \big( \io^{\frac{\zeta}{2}} + r_3^{\frac{\zeta}{2}} \big).
\end{aligned}
\end{equation}
Combining \eqref{eq:step4_ptVk_diff} and \eqref{eq:step4_nxVk_diff} yields
\begin{equation}\label{eq:step4_Vk_diff}
\abs{\pt \Vki\htxk - \pt \Vki(t_k,x_k)} + \abs{\nx \Vki\htxk - \nx \Vki(t_k,x_k)} \le C_{R_2} \big( \io^{\frac{\zeta}{2}} + r_3^{\frac{\zeta}{2}} \big).
\end{equation}
Similarly, \eqref{eq:step4_V_derivative} and \eqref{eq:bound_Vkstar_Holder} imply
\begin{equation}\label{eq:step4_Vkstar_diff}
\abs{\ps \Vkis\hsyk - \ps \Vkis(s_k,y_k)} + \abs{\ny \Vkis\hsyk - \ny \Vkis(s_k,y_k)} \le C_{R_2} \big( \io^{\frac{\zeta}{2}} + r_3^{\frac{\zeta}{2}} \big).
\end{equation}

We estimate the last term in \eqref{eq:gradPhik_split_terms}. By \eqref{eq:barrier},
$\pt \vp\txsy = -\gamma |x|^l_l + \frac{1}{\ve}(t-s) - \frac{\lam}{(t-t_-)^2}.$
Therefore, 
\begin{equation}\label{eq:step4_diff_ptphi}
\begin{aligned}
& \quad \abs{\pt\vp\htxsyk - \pt\vp\txsyk} \\
& \le \gamma \abs{|\hxk|_l^l - |x_k|_l^l} + \frac{1}{\ve} \abs{(\htk-\hsk) - (t_k-s_k)} + \lam \abs{\frac{1}{(\htk-t_-)^2} - \frac{1}{(t_k-t_-)^2}} \\
& \le \gamma \, l\sum_{i=1}^d (|(\hxk)_i|^{l-1}+|(x_k)_i|^{l-1})|(\hxk)_i - (x_k)_i| + \frac{1}{\ve} \abs{(\htk-t_k) - (\hsk-s_k)} + \frac{\lam |\htk+t_k-2t_-|}{(\htk-t_-)^2(t_k-t_-)^2}|\htk-t_k| \\
& \le \gamma \, l \parentheses{|\hxk|_{l-1}^{l-1} + |x_k|_{l-1}^{l-1}}\abs{\hxk-x_k} + \frac{1}{\ve}\parentheses{|\htk-t_k|+|\hsk-s_k|} + \frac{\lam \, 2\dt}{9r_2^4}|\htk-t_k|\\
& \le C\Big(\gamma R_2^{l-1} + \frac{1}{\ve} + \frac{\lam\dt}{r_2^4}\Big) \, r_3,
\end{aligned}
\end{equation}
where we use \eqref{eq:tk_boundary_distance}, \eqref{eq:xhat_bound_r3} and \eqref{eq:tkhat_boundary_distance}. Similarly, 
\begin{equation}\label{eq:step4_diff_psphi}
\abs{\ps\vp\htxsyk - \ps\vp\txsyk} \le C \Big(\gamma R_2^{l-1} + \frac{1}{\ve} + \frac{\lam\dt}{r_2^4}\Big) \, r_3.
\end{equation}
For derivatives in $x$, 
$\nx \vp\txsy = l\, \gamma (t_+-t+\dt) \, x^{l-1} + \frac{1}{\delta} (x-y),$
where $x^{l-1}$ denotes the component-wise power of \(x\). Therefore,
\begin{equation}\label{eq:step4_diff_nxphi}
\begin{aligned}
& \quad \abs{\nx\vp\htxsyk - \nx\vp\txsyk} \\
& \le l \,\gamma \abs{ (t_+-\htk+\dt)\,\hxk^{l-1} - (t_+-t_k+\dt)\,x_k^{l-1}} + \frac{1}{\delta} \abs{(\hxk-\hyk) - (x_k-y_k)} \\
& \le  l \,\gamma \abs{\htk-t_k} \abs{x_k^{l-1}} + l \,\gamma (t_+-\htk+\dt) \abs{\hxk^{l-1} - x_k^{l-1}} + \frac{1}{\delta}\abs{(\hxk-x_k) - (\hyk-y_k)} \\
& \le l \,\gamma \, r_3\abs{x_k}^{l-1} + 2l\,\gamma\dt \,r_3 + \frac{2r_3}{\delta}  \le \Big( l \,\gamma \, (R_2-1)^{l-1} + 2l\,\gamma\dt + \frac{2}{\delta} \Big) r_3.
\end{aligned}
\end{equation}
Similarly,
\begin{equation}\label{eq:step4_diff_nyphi}
\abs{\ny\vp\htxsyk - \ny\vp\txsyk}\le \Big( l \,\gamma \, (R_2-1)^{l-1} + 2l\,\gamma\dt + \frac{2}{\delta} \Big) r_3.
\end{equation}
Combining \eqref{eq:step4_diff_ptphi}, \eqref{eq:step4_diff_psphi}, \eqref{eq:step4_diff_nxphi}, and \eqref{eq:step4_diff_nyphi} yields
\begin{equation}\label{eq:step4_diff_nablaphi}
\abs{\nabla\vp\htxsyk - \nabla\vp\txsyk}\le C\Big( \gamma \, R_2^{l-1} + \frac{1}{\ve} + \frac{1}{\delta} + \frac{\lam\dt}{r_2^4} + \gamma\dt \Big) r_3.
\end{equation}
Substituting \eqref{eq:step4_Vk_diff}, \eqref{eq:step4_Vkstar_diff}, and \eqref{eq:step4_diff_nablaphi} into \eqref{eq:gradPhik_split_terms} yields
\begin{equation}\label{eq:step4_nablaPhik_bound}
\abs{\nabla \Phi_k\htxsyk} \le C_{R_2} \big( \io^{\frac{\zeta}{2}} + r_3^{\frac{\zeta}{2}} \big)+ C\Big( \gamma \, R_2^{l-1} + \frac{1}{\ve} + \frac{1}{\delta} + \frac{\lam\dt}{r_2^4} + \gamma\dt \Big) r_3.
\end{equation}
Substituting \eqref{eq:step4_nablaPhik_bound} and \eqref{eq:xhat_bound_r3} into \eqref{eq:Phikhat_critical_point} yields
\begin{equation}\label{eq:step4_bound_p}
\begin{aligned}
\abs{\qp} & \le C_{R_2} \big( \io^{\frac{\zeta}{2}} + r_3^{\frac{\zeta}{2}} \big)+ C\Big( \gamma \, R_2^{l-1} + \frac{1}{\ve} + \frac{1}{\delta} + \frac{\lam\dt}{r_2^4} + \gamma\dt \Big) r_3 + \mu r_3 \\
& \le C(R_2,\gamma,\ve,\delta,\lam,\dt,r_2) \big( \io^{\frac{\zeta}{2}} + r_3^{\frac{\zeta}{2}} \big),
\end{aligned}
\end{equation}
where $C(R_2,\gamma,\ve,\delta,\lam,\dt,r_2)$ denotes a constant that depends on those parameters.

\medskip
\noindent \emph{Step 5}. In this step, we compute the critical point system for $\hPk$ and define a quantity $B_k$, which is the key to deriving a contradiction. Since $\htxsyk$ maximizes $\hPk$ in the interior of $I_t \times B_{R_2} \times I_t \times B_{R_2}$, the first and second order necessary conditions provide
\begin{equation}\label{eq:step5_optimality_hPhik}
\left\{ \begin{aligned}
&0 = \pt\, \hPk(\htk,\hxk,\hsk,\hyk) = \pt \Vki(\htk,\hxk) - \pt\,  \varphi(\htk,\hxk,\hsk,\hyk) - \mu(\htk-t_k) + q\\
&0 = \ps\, \hPk(\htk,\hxk,\hsk,\hyk) = -\ps \Vkis(\hsk,\hyk) - \ps\,  \varphi(\htk,\hxk,\hsk,\hyk) - \mu(\hsk-s_k) + \hq\\
&0 = \nx \hPk(\htk,\hxk,\hsk,\hyk) = \nx \Vki(\htk,\hxk) - \nx  \varphi(\htk,\hxk,\hsk,\hyk) - \mu(\hxk-x_k) + p\\
&0 = \ny \hPk(\htk,\hxk,\hsk,\hyk) = -\ny \Vkis(\hsk,\hyk) - \ny  \varphi(\htk,\hxk,\hsk,\hyk) - \mu(\hyk-y_k) + \hp\\
&\begin{bmatrix} \nx^2 \Vki(\htk,\hxk) & 0 \\0 & -\ny^2 \Vkis(\hsk,\hyk) \end{bmatrix} \le
\nabla_{x,y}^2 \varphi\htxsyk
+ \mu I_{2n}
\end{aligned} \right.
\end{equation}
where matrix inequalities are in positive semi-definite sense. Therefore
\begin{equation}\label{eq:step5_ptVik}
\pt \Vki\htxk = - \gamma |\hxk|_l^l + \frac{1}{\ve}(\htk-\hsk) - \frac{\lam}{(\htk-t_-)^2} + \mu(\htk-t_k) - q,
\end{equation}
\begin{equation}\label{eq:step5_nxVik}
\nx \Vki\htxk = \gamma (t_+-\htk+\dt) \nx |\hxk|_l^l + \frac{1}{\delta}(\hxk-\hyk) + \mu(\hxk-x_k) - p,
\end{equation}
where $\nx |x|_l^l := lx^{l-1} \in \RR^d$, where the power applies component-wise. Similarly, 
\begin{equation}\label{eq:step5_psViks}
-\ps \Vkis\htxk = - \gamma |\hyk|_l^l + \frac{1}{\ve}(\hsk-\htk) - \frac{\lam}{(\hsk-t_-)^2} + \mu(\hsk-s_k) - \hq,
\end{equation}
\begin{equation}\label{eq:step5_nyViks}
-\ny \Vkis\hsyk = \gamma (t_+-\hsk+\dt) \ny |\hyk|_l^l + \frac{1}{\delta}(\hyk-\hxk) + \mu(\hyk-y_k) - \hp.
\end{equation}
Since $\nx^2 |x|_l^l = l(l-1) \, \diag\big( x^{l-2} \big)$, $t_+-\htk+\dt \le 2\dt$, 
the last equation in \eqref{eq:step5_optimality_hPhik} simplifies to
\begin{equation}\label{eq:Hessians_bound}
\begin{bmatrix} \nx^2 \Vki(\htk,\hxk) & 0 \\0 & -\ny^2 \Vkis(\hsk,\hyk) \end{bmatrix} \le 2\gamma \dt \, l (l-1) \begin{bmatrix} \wh{D}^k_x & 0 \\ 0 & \wh{D}^k_y \end{bmatrix}  + \frac{1}{\delta}\begin{bmatrix} I_n & -I_n \\ -I_n  & I_n \end{bmatrix}
+ \mu I_{2n},
\end{equation}
where
\begin{equation}\label{eq:Dkhat_def}
\wh{D}^k_x = \diag\big(\hxk^{l-2} \big),~~\wh{D}^k_y = \diag\big( \hyk^{l-2} \big).
\end{equation}

Define \(B_k := \ps \Vkis\hsyk - \pt \Vki\htxk = \ps V_k^*\hsykp - \pt V_k\htxkp\),
where the second equality follows from \eqref{eq:step4_Vhat_derivative}. The optimality conditions \eqref{eq:step5_ptVik} and \eqref{eq:step5_psViks} imply
\begin{equation}\label{eq:Bk_onehand}
\begin{aligned}
B_k &= \gamma (|\hxk|_l^l + |\hyk|_l^l) + \frac{\lam}{(\htk-t_-)^2} + \frac{\lam}{(\hsk-t_-)^2} -  \mu(\htk-t_k) - \mu(\hsk-s_k) + q + \hq \\
& \ge \gamma (|\hxk|_l^l + |\hyk|_l^l) + \frac{2\lam}{\dt^2} - 2\mu r_3 + q + \hq.
\end{aligned}
\end{equation}
Using the PDEs that characterize $V_k$ and $V_k^*$, we get
\begin{equation}\label{eq:Bk_otherhand}
\begin{aligned}
B_k &= H \big( \hsk', \hyk',\mu^k_{\hsk'}, \alpha^*_k\hsykp, -\ny V_k^*\hsykp, -\ny^2 V_k^*\hsykp \big) \\
& \quad - H \big( \htk', \hxk',\mu^k_{\htk'}, \alpha_k\htxkp, -\nx V_k\htxkp, -\nx^2 V_k\htxkp \big).
\end{aligned}
\end{equation}
Recall that 
$\alpha^\diamond_k(t,x) := \argmax_{a\in\RR^n} H(t,x,\mu_t^k,a,-\nx V_k(t,x), -\nx^2 V_k(t,x)).$
We split \eqref{eq:Bk_otherhand} into two terms $B_k = (\rom{1}) + (\rom{2})$, where
\begin{equation}\label{eq:Bk_otherhand_term1}
\begin{aligned}
(\rom{1}) &:= H \big( \hsk', \hyk',\mu^k_{\hsk'}, \alpha^*_k\hsykp, -\ny V_k^*\hsykp, -\ny^2 V_k^*\hsykp \big) \\
& \quad - H \big( \htk', \hxk',\mu^k_{\htk'}, \alpha_k^\diamond\htxkp, -\nx V_k\htxkp, -\nx^2 V_k\htxkp \big),
\end{aligned}
\end{equation}
\begin{equation}\label{eq:Bk_otherhand_term2}
\begin{aligned}
(\rom{2}) &:= H \big( \htk', \hxk',\mu^k_{\htk'}, \alpha_k^\diamond\htxkp, -\nx V_k\htxkp, -\nx^2 V_k\htxkp \big) \\
& \quad - H \big( \htk', \hxk',\mu^k_{\htk'}, \alpha_k\htxkp, -\nx V_k\htxkp, -\nx^2 V_k\htxkp \big).
\end{aligned}
\end{equation}

Next, we prove a local Lipschitz condition of the Hamiltonian in $\alpha$ and establish a bound for the local optimal control $\alpha^\diamond$. For fixed $(t,x,\mu,p,P) \in [0,T] \times \RR^d \times \mP(\RR^d) \times \RR^d \times \RR^{d\times d}$ with $W_2(\mu,\delta_0)\le K$, we temporarily denote the $\lam_H$-strongly concave mapping $\alpha \mapsto H(t,x,\mu,\alpha,p,P)$ by $H(\alpha)$.
Let $\alpha^\diamond$ denote its maximizer for given fixed $(\mu,p,P)$. For any $\alpha,\alpha'\in\RR^n$,
\begin{equation}\label{eq:Hamiltonian_Lip_alpha}
\begin{aligned}
& \quad \abs{H(\alpha) - H(\alpha')} \le \abs{f(t,x,\mu,\alpha) - f(t,x,\mu,\alpha')} + \abs{b(t,x,\mu,\alpha) - b(t,x,\mu,\alpha')} \,|p| \\
& \le K(1+|x|+K+|\alpha|\vee|\alpha'|)|\alpha-\alpha'| + K \,|\alpha-\alpha'| \, |p| \le C(1 + |x| + |\alpha|\vee|\alpha'| + |p|)\,|\alpha-\alpha'|.
\end{aligned}
\end{equation}
By the concavity of $H(\alpha)$,
$\inner{\na H(\alpha) - \na H(\alpha')}{\alpha - \alpha'} \le -\lam_H |\alpha-\alpha'|^2.$
Substituting $\alpha=0$, $\alpha'=\alpha^\diamond$ and using $\na H(\alpha^\diamond)=0$, we get $\inner{\na H(0)}{\alpha^\diamond} \ge \lam_H |\alpha^\diamond|^2$, which implies $|\alpha^\diamond| \le |\na H(0)| / \lam_H$. Therefore,
\begin{equation}\label{eq:alpha_diamond_bound}
\begin{aligned}
\abs{\alpha^\diamond_k\htxkp} &\le \frac{1}{\lam_H} \parentheses{ \abs{\na f(\htk',\hxk',\mu^k_{\htk'},0)} + \abs{\na b(\htk',\hxk',\mu^k_{\htk'},0)} \, \abs{\nx V_k\htxkp}} \\
& \le \frac{K}{\lam_H} \parentheses{ C +|\hxk'| + \abs{\nx V_k\htxkp}} \le C_{R_2}, 
\end{aligned}
\end{equation}
where we use Assumption~\ref{assu:basic}, \eqref{eq:bound_Vk_Holder} and \eqref{eq:step4_hatxk_bound}.

As a next step, we show that
\begin{equation}\label{eq:alpha_diamond_Holder}
\abs{\alpha^\diamond_k\htxkp - \alpha^\diamond_k\htxk} \le C_{R_2} r_3^{\frac{\zeta}{4}}.
\end{equation}
This time, we temporarily denote the mapping $(t,x,\alpha) \mapsto H(t,x,\mu^k_t,\alpha,-\nx V_k(t,x), -\nx^2 V_k(t,x))$ by $H(t,x,\alpha)$. 
For any pair of tuples $(t,x), (t',x')$ (later evaluated at $\htxk,\htxkp$), without loss of generality, we assume $H(t,x,\alpha^\diamond(t,x)) \ge H(t',x',\alpha^\diamond(t',x'))$, which implies
\begin{equation}\label{eq:Hamiltonian_diff_alpha}
\begin{aligned}
H(t,x,\alpha^\diamond(t,x)) - H(t',x',\alpha^\diamond(t,x)) & \ge H(t',x',\alpha^\diamond(t',x')) - H(t',x',\alpha^\diamond(t,x)) \\
&\ge \frac12 \lam_H \abs{\alpha^\diamond(t,x)-\alpha^\diamond(t',x')}^2,
\end{aligned}
\end{equation}
where the last inequality follows from the fact that $H$ is $\lam_H$-strongly concave in $\alpha$ and that $\alpha^\diamond(t',x')$ maximizes $H(t',x',\cdot)$. 
We remark that, if the converse $H(t',x',\alpha^\diamond(t',x')) > H(t,x,\alpha^\diamond(t,x))$ holds, subtracting $H(t,x,\alpha^\diamond(t',x'))$ (instead of $H(t',x',\alpha^\diamond(t,x))$) on both sides yields a similar inequality to \eqref{eq:Hamiltonian_diff_alpha}. 
We estimate the left-hand side of \eqref{eq:Hamiltonian_diff_alpha}, evaluated at $(t,x) = \htxk$, $(t',x')=\htxkp$, with $V_k$ satisfying \eqref{eq:bound_Vk_Holder}.
\begin{equation}\label{eq:Hamiltonian_diff_bound}
\begin{aligned}
& \quad \abs{H(t,x,\alpha^\diamond(t,x)) - H(t',x',\alpha^\diamond(t,x))} \\
& \le \abs{f(t,x,\mu^k_t,\alpha^\diamond(t,x)) - f(t',x',\mu^k_{t'},\alpha^\diamond(t,x))} + \abs{ b(t',x',\mu^k_{t'},\alpha^\diamond(t,x)) } \abs{\nx V_k(t,x) - \nx V_k(t',x')}  \\
& \quad  +\abs{b(t,x,\mu^k_t,\alpha^\diamond(t,x)) - b(t',x',\mu^k_{t'},\alpha^\diamond(t,x))} \abs{\nx V_k(t,x)} \\
&\quad + \abs{\sigma(t,x,\mu^k_t) - \sigma(t',x',\mu^k_{t'})} \abs{\nx^2 V_k(t,x)} + \abs{\sigma(t',x',\mu^k_{t'})} \abs{\nx^2 V_k(t,x) - \nx^2 V_k(t',x')}.
\end{aligned}
\end{equation}
Based on Assumption~\ref{assu:basic} and Assumption~\ref{assu:flow_in_class}, each term in \eqref{eq:Hamiltonian_diff_bound} can be estimated as follows:
\begin{equation}\label{eq:f_diff_same_alpha}
\begin{aligned}
& \quad \abs{f(t,x,\mu^k_t,\alpha^\diamond(t,x)) - f(t',x',\mu^k_{t'},\alpha^\diamond(t,x))} \\
& \le K \parentheses{1+|x|^2\vee|x'|^2 + W_2(\mu^k_t, \delta_0)^2 \vee W_2(\mu^k_{t'}, \delta_0)^2 + |\alpha^\diamond(t,x)|^2} |t-t'|^{\frac{1}{2}} \\
& \quad + K \parentheses{1+|x|\vee|x'| + W_2(\mu^k_t, \delta_0) \vee W_2(\mu^k_{t'}, \delta_0) + |\alpha^\diamond(t,x)|} \parentheses{|x-x'| + W_2(\mu^k_t,\mu^k_{t'})} \\
& \le C_{R_2} \parentheses{ |t-t'|^{\frac{1}{2}} + |x-x'|},
\end{aligned}
\end{equation}
where we use \eqref{eq:alpha_diamond_bound} in the second inequality. Similarly,
\begin{align*}
\abs{b(t,x,\mu^k_t,\alpha^\diamond(t,x)) - b(t',x',\mu^k_{t'},\alpha^\diamond(t,x))} &\le C_{R_2} \parentheses{ |t-t'|^{\frac{1}{2}} + |x-x'|},\\
\abs{\sigma(t,x,\mu^k_t) - \sigma(t',x',\mu^k_{t'})} &\le C \parentheses{ |t-t'| + |x-x'|}.
\end{align*}
Using bounds \(\abs{ b(t',x',\mu^k_{t'},\alpha^\diamond(t,x)) } \le C_{R_2}\), \(\abs{\sigma(t',x',\mu^k_{t'})} \le K\), $|x-x'|\vee|t-t'| \le r_3 < 1$ (cf.~\eqref{eq:htk_htkp_diff_bound_r3}), the H\"older condition~\eqref{eq:bound_Vk_Holder}, and all the estimations above, \eqref{eq:Hamiltonian_diff_bound} becomes
\begin{equation}\label{eq:Hamiltonian_diff_bound_final}
\abs{H(t,x,\alpha^\diamond(t,x)) - H(t',x',\alpha^\diamond(t,x))} \le C_{R_2} \parentheses{ |t-t'|^{\frac{\zeta}{2}} + |x-x'|^\zeta} \le C_{R_2} r_3^{\frac{\zeta}{2}},
\end{equation}
which concludes the proof of \eqref{eq:alpha_diamond_Holder}.

\medskip
\noindent \emph{Step 6}. We estimate $(\rom{1})$ \eqref{eq:Bk_otherhand_term1} and $(\rom{2})$ \eqref{eq:Bk_otherhand_term2} respectively. For the term $(\rom{2})$, by~\eqref{eq:Hamiltonian_Lip_alpha},
\begin{equation}\label{eq:Bk_term2_bound_temp}
\begin{aligned}
(\rom{2}) &\le C \parentheses{ 1 + |\hxk'| + |\alpha_k\htxkp|\vee|\alpha_k^\diamond\htxkp| + |\nx V_k\htxkp|} \, |\alpha_k\htxkp - \alpha_k^\diamond\htxkp| \\
& \le C_{R_2} \parentheses{ |\alpha_k\htxkp - \alpha_k\htxk| + |\alpha_k\htxk - \alpha_k^\diamond\htxk| + |\alpha_k^\diamond\htxk - \alpha_k^\diamond\htxkp|} \\
& \le C_{R_2} \parentheses{ C \, r_3 + |\alpha_k\htxk - \alpha_k^\diamond\htxk| + C_{R_2} r_3^{\frac{\zeta}{4}}} \le C_{R_2}  \parentheses{ |\alpha_k\htxk - \alpha_k^\diamond\htxk| +  r_3^{\frac{\zeta}{4}}},
\end{aligned}
\end{equation}
where the second inequality follows from \eqref{eq:bound_Vk_Holder}, \eqref{eq:step4_hatxk_bound}, and \eqref{eq:alpha_diamond_bound}, while the third inequality follows from \eqref{eq:htk_htkp_diff_bound_r3} and \eqref{eq:alpha_diamond_Holder}. 
Set $r_3$ to be small enough such that $C_{R_2} r_3^{\frac{\zeta}{4}} \le \frac{\lam}{4\dt^2}$ (cf. Figure~\ref{fig:parameters}).
We get
\begin{equation}\label{eq:Bk_term2_bound}
(\rom{2}) \le  C_{R_2}   |\alpha_k\htxk - \alpha_k^\diamond\htxk| + \frac{\lam}{4\dt^2}.
\end{equation}

Next, we estimate $(\rom{1})$. By the definition of $\alpha^\diamond_k$ (cf. \eqref{eq:local_optimal}),
\begin{equation}\label{eq:term1_split}
\begin{aligned}
(\rom{1}) & \le H \big( \hsk', \hyk',\mu^k_{\hsk'}, \alpha^*_k\hsykp, -\ny V_k^*\hsykp, -\ny^2 V_k^*\hsykp \big) \\
& \quad - H \big( \htk', \hxk',\mu^k_{\htk'}, \alpha_k^*\hsykp, -\nx V_k\htxkp, -\nx^2 V_k\htxkp \big) \\
& = \Tr\big[ D(\htk', \hxk',\mu^k_{\htk'}) \, \nx^2 V_k\htxkp - D(\hsk', \hyk',\mu^k_{\hsk'}) \, \ny^2 V_k^*\hsykp  \big] \\
& \quad + \big[ b(\htk', \hxk',\mu^k_{\htk'}, \alpha_k^*\hsykp)\tp \nx V_k\htxkp - b(\hsk', \hyk',\mu^k_{\hsk'}, \alpha^*_k\hsykp)\tp \ny V_k^*\hsykp \big] \\
& \quad + \big[ f(\htk', \hxk',\mu^k_{\htk'}, \alpha_k^*\hsykp) - f(\hsk', \hyk',\mu^k_{\hsk'}, \alpha^*_k\hsykp) \big] =: (\rom{3}) + (\rom{4}) + (\rom{5}).
\end{aligned}
\end{equation}
We estimate each term in \eqref{eq:term1_split} separately. We remark that, the estimation for \eqref{eq:term1_split} is similar to that in \eqref{eq:Hamiltonian_diff_bound}, both being the difference of Hamiltonian with the same input argument $\alpha$. Note that
\begin{equation}\label{eq:htsxykp_diff_bound}
\begin{aligned}
&|\htk'-\hsk'| \le |\htk' - \htk| + |\htk - t_k| + |t_k - s_k| + |s_k - \hsk| + |\hsk - \hsk'|  \le 4r_3 + C_1 (\gamma\dt)^{-\frac{2}{l-2}} (\ve+\delta),\\
&|\hxk'-\hyk'| \le |\hxk' - \hxk| + |\hxk - x_k| + |x_k - y_k| + |y_k - \hyk| + |\hyk - \hyk'| \le 4r_3 + C_1 (\gamma\dt)^{-\frac{2}{l-2}} (\ve+\delta),
\end{aligned}
\end{equation}
which follow from \eqref{eq:bound_tsxy_diff}, \eqref{eq:xhat_bound_r3}, and \eqref{eq:htk_htkp_diff_bound_r3}.

The estimation for $(\rom{5})$ is similar to \eqref{eq:f_diff_same_alpha}, except that $|\alpha_k^\diamond(\htk,\hxk)| \le C_{R_2}$ is replaced by $|\alpha_k(\htk,\hxk)| \le C(1+|\hxk|) \le CR_2$. By~\eqref{eq:htsxykp_diff_bound},
\begin{equation}\label{eq:Hdiff_termf}
(\rom{5}) \le C R_2^2 \big(|\htk'-\hsk'|^{\frac12} + |\hxk'-\hyk'|\big) \le C R_2^2 (r_3^{\frac12}+(\gamma\dt)^{-\frac{1}{l-2}}(\ve+\delta)^{\frac12}),
\end{equation}
Recall that we have set $\ve,\delta$ to be small enough such that $C_1(\gamma\dt)^{-\frac{2}{l-2}} (\ve+\delta) \le c_2 \eta / (2R_2^2) $ in \emph{Step 3}.

We split $(\rom{4})$ into two terms $(\rom{4}) = (\rom{6}) + (\rom{7})$, where 
$$(\rom{6}) := \big[ b(\htk', \hxk',\mu^k_{\htk'}, \alpha_k^*\hsykp) - b(\hsk', \hyk',\mu^k_{\hsk'}, \alpha^*_k\hsykp)\big] \tp \nx V_k\htxkp,$$
$$(\rom{7}) :=  b(\hsk', \hyk',\mu^k_{\hsk'}, \alpha^*_k\hsykp)\tp\big[ \nx V_k\htxkp - \ny V_k^*\hsykp \big].$$
For $(\rom{6})$, by \eqref{eq:bound_Vk_Holder} and \eqref{eq:htsxykp_diff_bound},
\begin{equation}\label{eq:diffb_gradV}
\begin{aligned}
(\rom{6}) & \le \Big[ K\parentheses{1 + | \hxk'|\vee |\hyk'| + W_2(\mu^k_{\htk'},\delta_0) \vee W_2(\mu^k_{\hsk'},\delta_0) + \abs{\alpha_k^*\hsykp}} |\htk'-\hsk'|^{\frac12}\\
& \quad + K \big(|\hxk'-\hyk'| + W_2(\mu^k_{\htk'},\mu^k_{\hsk'}) \big)  \Big] C_{R_2} \\
& \le \Big[C R_2 (r_3^{\frac12}+(\gamma\dt)^{-\frac{1}{l-2}}(\ve+\delta)^{\frac12})\Big] C_{R_2} = C_{R_2}  (r_3^{\frac12}+(\gamma\dt)^{-\frac{1}{l-2}}(\ve+\delta)^{\frac12}).
\end{aligned}
\end{equation}
For $(\rom{7})$, by \eqref{eq:step4_Vhat_derivative}, \eqref{eq:step5_nxVik}, \eqref{eq:step5_nyViks} and  \eqref{eq:xhat_bound_r3},
\begin{equation}\label{eq:b_diffgradV}
\begin{aligned}
(\rom{7}) & \le K (1 + |\hyk'| + W_2(\mu^k_{\hsk'},\delta_0) + \abs{\alpha_k^*\hsykp}) \abs{\nx\Vki\htxk - \ny\Vkis\hsyk} \\
& \le C (1+|\hyk|) \, \big|\gamma (t_+-\htk+\dt) \nx |\hxk|_l^l + \gamma (t_+-\hsk+\dt) \ny |\hyk|_l^l \\
& \quad + \mu(\hxk-x_k) + \mu(\hyk-y_k) - p - \hp \,\big| \\
& \le C (1+|\hyk|) \sqbra{2\gamma\, \dt \,l\, ( |\hxk|_{l-1}^{l-1} + |\hyk|_{l-1}^{l-1})  + 2\mu r_3 + |p| + |\hp|}\\
& \le C \gamma \, \dt  \,(|\hxk|_l^l + |\hyk|_l^l) + CR_2 (\mu r_3 + |p| + |\hp|),
\end{aligned}
\end{equation}
where the last inequality follows from Young's inequality $ab^{l-1}\le \frac{1}{l}a^l+\frac{l-1}{l}b^l,\ \forall a,b>0$. Combining \eqref{eq:diffb_gradV} and \eqref{eq:b_diffgradV} yields
\begin{equation}\label{eq:Hdiff_termbgradV}
(\rom{4}) \le C_{R_2}  (r_3^{\frac12}+(\gamma\dt)^{-\frac{1}{l-2}}(\ve+\delta)^{\frac12}) + C \gamma \dt \,(|\hxk|_l^l + |\hyk|_l^l) + CR_2 (\mu r_3 + |p| + |\hp|).
\end{equation}
We remark that, if $(\rom{4})$ is estimated using the same strategy as that in \eqref{eq:Hamiltonian_diff_bound}, the constant \(C\) in $C \gamma \dt \big(|\hxk|_l^l + |\hyk|_l^l\big)$ will depend on $R_2$, causing the failure to reaching a contradiction in the next step. 

For $(\rom{3})$, by \eqref{eq:step4_Vhat_derivative} and \eqref{eq:Hessians_bound}, 
\begin{align*}
(\rom{3}) &= \dfrac12 \Tr \Bigg[ \begin{pmatrix}  \sigma(\htk',\hxk',\mu^k_{\htk'}) \\ \sigma(\hsk', \hyk',\mu^k_{\hsk'}) \end{pmatrix}\tp \begin{bmatrix} \nx^2V_k(\htk',\hxk') &0 \\ 0& -\ny^2V_k^*(\hsk',\hyk') \end{bmatrix} \begin{pmatrix}  \sigma(\htk',\hxk',\mu^k_{\htk'}) \\ \sigma(\hsk', \hyk',\mu^k_{\hsk'}) \end{pmatrix} \Bigg]\\
& \le \dfrac12 \Tr \Bigg[ \begin{pmatrix}  \sigma(\htk',\hxk',\mu^k_{\htk'}) \\ \sigma(\hsk', \hyk',\mu^k_{\hsk'}) \end{pmatrix}\tp \begin{bmatrix} \nx^2 \Vki(\htk,\hxk) &0 \\ 0& -\ny^2 \Vkis(\hsk,\hyk) \end{bmatrix} \begin{pmatrix}  \sigma(\htk',\hxk',\mu^k_{\htk'}) \\ \sigma(\hsk', \hyk',\mu^k_{\hsk'}) \end{pmatrix} \Bigg]\\
&\le \dfrac12 \Tr \sqbra{  \begin{pmatrix}  \sigma(\htk',\hxk',\mu^k_{\htk'}) \\ \sigma(\hsk', \hyk',\mu^k_{\hsk'}) \end{pmatrix}\tp \parentheses{2\gamma \dt \, l (l-1) \begin{bmatrix} \wh{D}^k_x & 0 \\ 0 & \wh{D}^k_y \end{bmatrix} + \dfrac{1}{\delta}\begin{bmatrix} I_n & -I_n \\ -I_n & I_n \end{bmatrix} + \mu I_{2n}}  \begin{pmatrix}  \sigma(\htk',\hxk',\mu^k_{\htk'}) \\ \sigma(\hsk', \hyk',\mu^k_{\hsk'}) \end{pmatrix} }.
\end{align*}
Therefore, by \eqref{eq:Dkhat_def} and \eqref{eq:htsxykp_diff_bound},
\begin{equation}\label{eq:Hdiff_termsigmaHessV}
\begin{aligned}
(\rom{3}) \le & \gamma \dt \, l (l-1) \sqbra{\Tr\parentheses{\wh{D}^k_x \,\,(\sigma\sigma\tp)(\htk',\hxk',\mu^k_{\htk'})} + \Tr\parentheses{\wh{D}^k_y \,\, (\sigma\sigma\tp)(\hsk', \hyk',\mu^k_{\hsk'})}} \\
& \quad + \dfrac{1}{2\delta} \abs{\sigma(\htk',\hxk',\mu^k_{\htk'}) - \sigma(\hsk', \hyk',\mu^k_{\hsk'})}^2 + \dfrac{\mu}{2} \parentheses{\abs{\sigma(\htk',\hxk',\mu^k_{\htk'})}^2 + \abs{\sigma(\hsk', \hyk',\mu^k_{\hsk'})}^2} \\
& \le \gamma \dt \, l (l-1) K^2 \parentheses{|\hxk|_{l-2}^{l-2} + |\hyk|_{l-2}^{l-2}} + \frac{C}{\delta} \parentheses{|\htk'-\hsk'| + |\hxk'-\hyk'|}^2 + \mu K^2 \\
& \le C \gamma \dt  \parentheses{|\hxk|_{l-2}^{l-2} + |\hyk|_{l-2}^{l-2}} + \frac{C}{\delta} \parentheses{r_3^2 + (\gamma\dt)^{-\frac{4}{l-2}} (\ve+\delta)^2} + \mu K^2.
\end{aligned}
\end{equation}
Substituting \eqref{eq:Hdiff_termf}, \eqref{eq:Hdiff_termbgradV}, and \eqref{eq:Hdiff_termsigmaHessV} into \eqref{eq:term1_split} yields
\begin{equation}\label{eq:term1_bound_temp}
\begin{aligned}
(\rom{1}) &\le C R_2^2 (r_3^{\frac12}+(\gamma\dt)^{-\frac{1}{l-2}}(\ve+\delta)^{\frac12}) + C_{R_2}  (r_3^{\frac12}+(\gamma\dt)^{-\frac{1}{l-2}}(\ve+\delta)^{\frac12}) + C \gamma \dt \,(|\hxk|_l^l + |\hyk|_l^l) \\
&+ CR_2 (\mu r_3 + |p| + |\hp|) + C \gamma \dt  \parentheses{|\hxk|_{l-2}^{l-2} + |\hyk|_{l-2}^{l-2}} + \frac{C}{\delta} \parentheses{r_3^2 + (\gamma\dt)^{-\frac{4}{l-2}} (\ve+\delta)^2} + \mu K^2 \\
& \le \mu K^2 + C_{R_2} (\gamma\dt)^{-\frac{1}{l-2}}(\ve+\delta)^{\frac12} + \frac{C}{\delta} (\gamma\dt)^{-\frac{4}{l-2}} (\ve+\delta)^2 + (C_{R_2}r_3^{\frac12} + C r_3^2 / \delta)  \\
&\quad + CR_2 (|p| + |\hp|) + C_4 \gamma \dt \,(1+|\hxk|_l^l + |\hyk|_l^l),
\end{aligned}
\end{equation}
where we use $l|x|_{l-2}^{l-2} \le 2d + (l-2)|x|_l^l$.
We record the constant \(C_4\) for parameter specification.

Recall the parameter dependence illustrated in Figure~\ref{fig:parameters}. Following this dependence, we set $\mu$ to be small enough such that $\mu K^2 \le \frac{\lam}{4\dt^2}$. Then we set $\gamma$ to be small enough such that $C_4 \gamma \dt \le \frac{\lam}{4\dt^2}$. Then we set $\ve=\delta$ to be small enough (depending on $R_2, \gamma, \dt$) such that 
$$C_{R_2} (\gamma\dt)^{-\frac{1}{l-2}}(\ve+\delta)^{\frac12} + \frac{C}{\delta} (\gamma\dt)^{-\frac{4}{l-2}} (\ve+\delta)^2 \le \mu K^2 \le \frac{\lam}{4\dt^2},$$
where $C_{R_2}$ corresponds to the constants in \eqref{eq:term1_bound_temp}. Lastly, we set $r_3$ to be small enough such that $C_{R_2}r_3^{\frac12} + C r_3^2 / \delta \le \frac{\lam}{4\dt^2}$. Combining these settings together into \eqref{eq:term1_bound_temp}, we obtain 
\begin{equation}\label{eq:BK_term1_bound}
(\rom{1}) \le \frac{\lam}{\dt^2} + CR_2 (|p| + |\hp|) + C_4 \gamma \dt \,(|\hxk|_l^l + |\hyk|_l^l).
\end{equation}
Combining \eqref{eq:Bk_term2_bound} and \eqref{eq:BK_term1_bound} yields
\begin{equation}\label{eq:Bk_otherhand_bound}
B_k \le C_{R_2} |\alpha_k\htxk - \alpha_k^\diamond\htxk| + \frac{5\lam}{4\dt^2} + CR_2 (|p| + |\hp|) + C_4 \gamma \dt \,(|\hxk|_l^l + |\hyk|_l^l).
\end{equation}

\medskip
\noindent \emph{Step 7}. We combine previous estimations to reach a contradiction. By \eqref{eq:Bk_onehand} and \eqref{eq:Bk_otherhand_bound},
\begin{align*}
&\quad \gamma (|\hxk|_l^l + |\hyk|_l^l) + \frac{2\lam}{\dt^2} - 2\mu r_3 + q + \hq \\
& \le C_{R_2} |\alpha_k\htxk - \alpha_k^\diamond\htxk| + \frac{5\lam}{4\dt^2} + CR_2 (|p| + |\hp|) + C_4 \gamma \dt \,(|\hxk|_l^l + |\hyk|_l^l).
\end{align*}
Setting $\dt \le 1/C_4$, the inequality simplifies to
$$\frac{3\lam}{4\dt^2} \le C_{R_2} |\alpha_k\htxk - \alpha_k^\diamond\htxk| + C_5R_2 (|p| + |\hp|) + |q| + |\hq| + 2\mu r_3.$$
Let $r_3$ be small enough such that $2\mu r_3 \le \frac{\lam}{4\dt^2}$. Additionally, since $|\qp|$ satisfies \eqref{eq:step4_bound_p}, we can always first find $r_3$, then find $\io$ such that
$C_5 R_2 (|p| + |\hp|) + |q| + |\hq| \le \frac{\lam}{4\dt^2}.$
As a result,
\begin{equation}\label{eq:alphak_diff_lower}
C_{R_2} |\alpha_k\htxk - \alpha_k^\diamond\htxk| \ge \frac{\lam}{4\dt^2}.
\end{equation}

Squaring both sides of \eqref{eq:alphak_diff_lower} and integrating $\htxk$ with respect to the density function $\rho^k = \rho^{\mu^{\tau_k}, \alpha^{\tau_k}}$ on a small domain $I_k \times B(x_k,\tfrac12r_3)$, where $I_k := [t_k - \frac12 r_3, t_k + \frac12 r_3]$ and $B(x_k,\tfrac12r_3) := \{ x\in\RR^d ~|~ |x-x_k| \le \frac12 r_3 \}$, yields
\begin{equation}\label{eq:ineq_final1}
\begin{aligned}
& \quad \int_{I_k} \int_{B(x_k,\tfrac12r_3)} \frac{\lam^2}{16\dt^4} \rho^k(t,x) \,\rd x\,\rd t \le C_{R_2} \int_{I_k} \int_{B(x_k,\tfrac12r_3)} \abs{\alpha_k(t,x) - \alpha_k^\diamond(t,x)}^2 \rho^k(t,x) \,\rd x\,\rd t \\
& \le C_{R_2} \norm{\alpha_k - \alpha_k^\diamond}_k^2 \le \frac{C_{R_2}}{k^2}  \norm{\alpha_k - \alpha_k^*}_k^2 \le  \frac{C_{R_2}}{k^2},
\end{aligned}
\end{equation}
where the third inequality is based on the condition~\eqref{eq:assume_contrary_lem}, and the last inequality is based on the (uniform in \(k\)) boundedness of \(\norm{\alpha_k - \alpha_k^*}_k^2\), as implied by Assumption~\ref{assu:flow_in_class}.
Since the density function $\rho^k(t,x)$ has a lower bound $c_{R_2} >0$ when $|x| \le R_2$, which is uniform in \(k\) (see \eqref{eq:Aronson}), \eqref{eq:ineq_final1} becomes
$$\frac{C_{R_2}}{k^2} \ge \int_{I_k} \int_{B(x_k,\tfrac12r_3)} \frac{\lam^2}{16\dt^4} \rho^k(t,x) \,\rd x\,\rd t \ge |I_k| \, |B(\tfrac12r_3,x_k)| \,\frac{\lam^2}{16\dt^4} \, c_{R_2} = C r_3^{d+1} \,\frac{\lam^2}{16\dt^4} \, c_{R_2}.$$
Setting $k \to \infty$ provides a contradiction, which builds upon the assumption that $\{h_k(t)\}_{k=1}^\infty$ has a subsequence that converges to a nonzero function $h(t)$. Therefore,
$\limsup_{k \to \infty} h_k(t) = 0,\ \forall t \in [0,T],$
and this concludes the proof of \eqref{eq:important_lemma_lem}.
\end{proof}

\subsection{Superlinear growth lemma}\label{sec:superlinear_growth}
In this section, we prove~\eqref{eq:superlinear_growth}. 
The motivation comes from Lemma~\ref{lem:value_gap}, which proves that the optimality gap in value functions has a superlinear (actually quadratic in \eqref{eq:value_gap}) growth with respect to the optimality gap in controls.
\begin{lem}\label{lem:superlinear_growth}
Under the conditions of Theorem~\ref{thm:actor_convergence},
\begin{equation}\label{eq:superlinear_growth_lem}
\norm{\nx V\ma - \nx V\ms}_\mao \le C \norm{\alpha - \alpha\ms}_\mao^{1+\chi},
\end{equation}
where $\chi = \frac{2}{d+5}$.
\end{lem}
\begin{proof}
Fix any \((t,x)\in[0,T]\times\R^d\). Let $x_s := X\ma_s$ denote the state process under $(\mu,\alpha)$, with a given initial condition $x_t=x$. Let $\alpha_s := \alpha(s,x_s)$, $\alpha^*_s := \alpha^{\mu,*}(s,x_s)$, $\phi := \alpha- \alpha\ms$ and $\phi_s := \phi(s,x_s)$. By \eqref{eq:Lip_gradV} in Lemma~\ref{lem:Value_Lipschitz}, it suffices to prove the lemma in the case where $\norm{\phi}_\mao \le 1$. 

By Lemma~\ref{lem:value_gap},
\begin{equation}\label{eq:value_gap_tx}
\begin{aligned}
&V\ma(t,x) - V\ms(t,x) = -\EE \bigg[ \int_t^T \int_0^1 \int_0^u \phi_s \tp \\
& \na^2 H\parentheses{s,x_s,\mu_s,\alpha^*_s + v\phi_s, -\nx V\ms(s,x_s)} \, \phi_s \,\rd v\,\rd u\,\rd s ~\Big|~ x_t = x \bigg].
\end{aligned}
\end{equation}

\emph{Step 1.} Bound each component of the gradient $\norm{\partial_{x_1} V\ma - \partial_{x_1} V\ms}_\mao$.
Given $(t,x)$, define the tangent SDE \cite{kunita1990stochastic} by $y_s := \partial_{x_1} X_s\ma$, i.e., the partial derivative of $x_s = X\ma_s$ with respect to the first component of its initial condition. 
Denote $b\ma(t,x):= b(t,x,\mu_t,\alpha(t,x))$ and $\sigma^\mu(t,x):= \sigma(t,x,\mu_t)$. 
Then $y_s$ has an initial condition $y_t = e_1$, where \(e_1\) denotes the first standard basis of \(\R^d\), and
$$\rd y_s = \nx b\ma(t,x_s) y_s \,\rd s + \parentheses{\nx \sigma^\mu(s,x_s) \cdot y_s} \rd W_s,$$
where $\nx \sigma^\mu(s,x_s) \cdot y_s := \sum_{i=1}^d \partial_{x_i} \sigma^\mu(s,x)\, (y_s)_i=:\sigma^y_s$. By It\^o's formula, 
$$\rd \abs{y_s}^2 = \sqbra{2y_s\tp \nx b\ma(t,x_s) y_s + \Tr(\sigma^y_s \sigma^{y\top}_s) } \rd s + 2 y_s\tp \sigma^y_s \rd W_s.$$
Since both $\abs{\nx b\ma} = \abs{\nx b +\na b \,\nx \alpha}$ and $\abs{\nx\sigma^\mu}$ are bounded, we have
$$\ps \EE\big[ \abs{y_s}^2 ~\big|~ x_t=x\big] \le C \, \EE \big[ \abs{y_s}^2 ~\big|~ x_t=x\big].$$
Together with $|y_t|=1$, Gr\"onwall's inequality implies
\begin{equation}\label{eq:bound_tangentSDE}
\EE\sqbra{ \abs{y_s}^2 ~\big|~ x_t=x} \le C,\ \forall 0\leq t\leq s\leq T,\ \forall x\in\R^d,
\end{equation}
where $C$ is uniform in $t$ and $x$.

Using \eqref{eq:value_gap_tx} and $y_s$, we estimate $\partial_{x_1} V\ma - \partial_{x_1} V\ms$.
Recall that
$$\na^2 H\parentheses{s,x_s,\mu_s,(\alpha^*_s + v\phi_s)(s,x_s), -\nx V\ms(s,x_s)} = -\na^2 f - \sum_{i=1}^d \na^2 b_i \, \partial_{x_i} V\ms(s,x_s).$$
By the chain rule, taking derivative with respect to $(x_t)_1$ yields (first differentiate with respect to $x_s$, then multiply $\partial_{(x_t)_1} x_s = y_s$)
\begin{align*}
& \quad \partial_{(x_t)_1} \na^2 H\parentheses{s,x_s,\mu_s,(\alpha^*_s + v\phi_s)(s,x_s), -\nx V\ms(s,x_s)} \\
& = - \sum_{j=1}^d (y_s)_j  \,\,\partial_{x_j} \na^2 f  - \sum_{j=1}^n (\nx(\alpha^* + v\phi)\tp_j y_s) \,  \partial_{\alpha_j} \na^2 f - \sum_{i=1}^d  \Big[ \sum_{j=1}^d  \partial_{x_i} V\ms \,\partial_{x_j} \na^2 b_i (y_s)_j \\
& \quad  + \sum_{j=1}^n \partial_{x_i} V\ms\,  (\nx(\alpha^* + v\phi)\tp_j y_s) \, \partial_{\alpha_j} \na^2 b_i  +  (\partial_{x_i} \nx V^{\mu,* \top} y_s) \na^2 b_i\Big],
\end{align*}
where we omit dependence on $(s,x_s,\mu_s, (\alpha^*+v\phi)(s,x_s))$ and $(s,x_s)$ whenever the context is clear. 
By Assumption~\ref{assu:basic}, Assumption~\ref{assu:flow_in_class}, and the estimation for the value function in Lemma~\ref{lem:value_growth}, we obtain
\begin{equation}\label{eq:partialx1_na2_Hamiltonian_bound}
\abs{\partial_{(x_t)_1} \na^2 H\parentheses{s,x_s,\mu_s,(\alpha^*_s + v\phi_s)(s,x_s), -\nx V\ms(s,x_s)}} \le C (1+|x_s|) |y_s|,
\end{equation}
\begin{equation}\label{eq:na2_Hamiltonian_bound}
\abs{\na^2 H\parentheses{s,x_s,\mu_s,(\alpha^*_s + v\phi_s)(s,x_s), -\nx V\ms(s,x_s)}} \le C (1+|x_s|).
\end{equation}
For the term $\phi(s,x_s)$, we have $\partial_{(x_t)_1} \phi(s,x_s) = \nx \phi(s,x_s) y_s$. Therefore, taking derivative of \eqref{eq:value_gap_tx} with respect to $x_1$ yields
\begin{equation}\label{eq:px1_V_diff}
\partial_{x_1} V\ma(t,x) - \partial_{x_1} V\ms(t,x) = - \EE \Big[ \int_0^T \int_0^1 \int_0^u \Big( 2\phi_s\tp \na^2 H \, \nx\phi_s \, y_s + \phi_s\tp (\partial_{(x_t)_1} \na^2 H) \phi_s \Big) \rd v\, \rd u \,\rd s\Big].
\end{equation}
Substituting the estimations \eqref{eq:partialx1_na2_Hamiltonian_bound}, \eqref{eq:na2_Hamiltonian_bound} into \eqref{eq:px1_V_diff} yields
\begin{equation}\label{eq:px1_V_diff_bound}
\begin{aligned}
\abs{\partial_{x_1} V\ma(t,x) - \partial_{x_1} V\ms(t,x)} 
&\le C \, \EE \Big[ \int_t^T \int_0^1 \int_0^u (1+|x_s|) |\phi_s| \parentheses{|\phi_s|+|\nx\phi_s|} |y_s| \,\rd v\, \rd u \,\rd s\Big] \\
& \le C \, \EE \Big[ \int_t^T (1+|x_s|) |\phi_s| \parentheses{|\phi_s|+|\nx\phi_s|} |y_s| \, \rd s \Big].
\end{aligned}
\end{equation}
By consecutive applications of~\eqref{eq:px1_V_diff_bound}, H\"older's inequality,~\eqref{eq:bound_tangentSDE}, Fubini's theorem and tower property,
\begin{align*}
& \quad \norm{\partial_{x_1} V\ma - \partial_{x_1} V\ms}^2_\mao \\
& \le C \, \EE_{x_t\sim \rho^{\ma}_t} \Big\{ \int_0^T \EE \Big[ \int_t^T (1+|x_s|) |\phi_s| \parentheses{|\phi_s|+|\nx\phi_s|} |y_s| \, \rd s ~\Big|~ x_t \Big]^2 \rd t \Big\} \\
& \le C \, \EE_{x_t\sim \rho^{\ma}_t} \Big\{ \int_0^T \EE \Big[ \int_t^T (1+|x_s|)^2 |\phi_s|^2 \parentheses{|\phi_s|+|\nx\phi_s|}^2 \, \rd s ~\Big|~ x_t \Big] \cdot \EE\Big[ \int_t^T |y_s|^2 \, \rd s ~\Big|~ x_t \Big] \, \rd t \Big\} \\
& \le C \, \EE_{x_t\sim \rho^{\ma}_t} \Big\{ \int_0^T \EE \Big[ \int_t^T (1+|x_s|^2) |\phi_s|^2 \parentheses{|\phi_s|+|\nx\phi_s|}^2 \, \rd s ~\Big|~ x_t \Big] \, \rd t \Big\} \\
& \le C \, \EE \Big[ \int_0^T (1+|x_t|^2) |\phi_t|^2 \parentheses{|\phi_t|+|\nx\phi_t|}^2 \, \rd t \Big] \\
& = C \inttx{ (1+|x|^2) \, |\phi(t,x)|^2 \, \parentheses{|\phi(t,x)|+|\nx\phi(t,x)|}^2 \rho\ma(t,x)}.
\end{align*}
Applying the same analysis in each dimension yields
\begin{equation}\label{eq:nxV_diff_bound1}
\norm{\nx V\ma - \nx V\ma}^2_\mao \le  C \inttx{ (1+|x|^2) \, \abs{\Phi(t,x)}^2 \rho(t,x)},
\end{equation}
where for simplicity, we denote
$$\Phi(t,x) := |\phi(t,x)| \parentheses{|\phi(t,x)| + |\nx\phi(t,x)|}, \quad \rho(t,x) := \rho\ma(t,x).$$

\medskip
\noindent \emph{Step 2.} We estimate the right-hand side of \eqref{eq:nxV_diff_bound1}. 
Recall that the density function $\rho$ satisfies the Aronson-type bound \eqref{eq:Aronson}. Fix $R>0$ such that $1+R^2 \ge 2K$. 
Denote $B_R := \{ x\in\RR^d : |x| \le R \}$, $B_R^c := \{ x\in\RR^d : |x| > R \}$ and omit dependence on $(t,x)$ whenever the context is clear. We get
$$(1+R^2) \int_0^T \int_{B_R^c} \abs{\Phi}^2 \rho \,\rd x\,\rd t \le \int_0^T \int_{\RR^d} (1+|x|^2) \abs{\Phi}^2 \rho \,\rd x\,\rd t \le K \int_0^T \int_{\RR^d} \abs{\Phi}^2 \rho \,\rd x\,\rd t.$$
As a result,
$$\int_0^T \int_{B_R} \abs{\Phi}^2 \rho \,\rd x\,\rd t \ge \parentheses{1 - \dfrac{K}{1+R^2}} \int_0^T \int_{\RR^d} \abs{\Phi}^2 \rho \,\rd x\,\rd t \ge \frac12 \int_0^T \int_{\RR^d} \abs{\Phi}^2 \rho \,\rd x\,\rd t,$$
\begin{equation}\label{eq:central_2Mtimes}
\int_0^T \int_{\RR^d} (1+|x|^2) \abs{\Phi}^2 \rho \,\rd x\,\rd t \le K \int_0^T \int_{\RR^d} \abs{\Phi}^2 \rho \,\rd x\,\rd t \le 2K \int_0^T \int_{B_R} |\phi|^2 \parentheses{|\phi| + |\nx \phi|}^2 \rho \,\rd x\,\rd t.
\end{equation}

Next, we claim that:
\begin{equation}\label{eq:claim_superlinear}
|\phi(t,x)| + |\nx \phi(t,x)| \le C \norm{\phi}_\mao^{\frac{2}{d+5}},\ \forall |x|\leq R.
\end{equation}
Recall that we only have to prove the claim when $\norm{\phi}_\mao \le 1$. To proceed, we provide two arguments below.

\medskip
\noindent \emph{Argument 1.} If there exists $ (t^*,x^*) \in [0,T] \times B_R$, such that $\abs{\phi(t^*,x^*)} = \ve \in(0,1]$ then $\norm{\phi}^2_\rho \ge c \ve^{d+3}$. Denote $r:= \ve / (4K(R+2))$. For any $(t,x) \in \wt B$, where
$$\wt B:= \curlybra{ (t,x) \in [0,T] \times \RR^d : \abs{(t,x) - (t^*,x^*)} \le r},$$
since $\alpha, \alpha\ms \in \mA$,
\begin{align*}
\abs{\phi(t,x)} &\ge \abs{\phi(t^*,x^*)} - \abs{\phi(t,x) - \phi(t^*,x^*)} \ge \ve - (2K(R+1)|t-t^*| + 2K |x-x^*|) \\
&  \ge\ve - 2K(R+2) r = \frac12 \ve.
\end{align*}
Therefore, by \eqref{eq:Aronson},
$$\inttx{ \abs{\phi(t,x)}^2 \rho(t,x)} \ge \iint_{\wt B} \abs{\phi(t,x)}^2 \rho(t,x) \,\rd x \, \rd t \ge \big|\wt B\big| \, (\frac12\ve)^2 c_l \exp(-C_l(R+\tfrac{1}{4K(R+2)})^2) \ge c \,\ve^{d+3}.$$

\medskip

\noindent \emph{Argument 2.} If there exists $(t^*,x^*) \in [0,T] \times B_R$, such that $\abs{\nx \phi(t^*,x^*)} = \ve \in(0,1]$, then $\norm{\phi}^2_\rho \ge c \ve^{d+5}$. Note that the norm of at least one row of $\nx \phi$ must be greater than $|\nx\phi| / \sqrt{n}$, so we assume without loss of generality that $\abs{\nx \phi_1(t^*,x^*)} = \ve_1 > \ve / \sqrt{n}$. 

Define $v := \nx \phi_1(t^*,x^*)/\ve_1,\ r_1 := \tfrac{\ve_1}{16K}, \ \dt := \tfrac{\ve_1}{16K} \wedge T.$
For any $(t,x) \in [0,T] \times \RR^d$ such that $|t-t^*|\le \dt$ and $|x-x^*| \le r_1$,
$\abs{\nx\phi_1(t,x) - \nx\phi_1(t^*,x^*)} \le 2K (|t-t^*| + |x-x^*|) \le \frac14 \ve_1$,
which implies
$$\ps [\phi_1(t,x+sv)]\big|_{s=0} = \nx\phi_1(t,x)\tp v \ge \frac34 \ve_1.$$
By Assumption~\ref{assu:flow_in_class}, \(|\nx\phi_1(t,x+z+sv) - \nx\phi_1(t,x)|\leq 4Kr_1,\ \forall z \in v^\perp,|z|\le r_1,s \in [-r_1,r_1]\), implying
$$\nx\phi_1(t,x+z+sv)\tp v \ge \frac12 \ve_1, \ \forall z \in v^\perp,|z|\le r_1,s \in [-r_1,r_1].$$
Define $\psi_t(z,s):= \phi_1(t,x^*+z+sv)$ so that $\ps \psi_t(z,s) \ge \frac12 \ve_1$. 
Integrating both sides yields \(\psi_t(z,s) = \psi_t(z,0) + \frac{1}{2}\ve_1s\). Squaring and integrating both sides once more yield
$$\int_{-r_1}^{r_1} \abs{\psi_t(z,s)}^2 \rd s \ge  \int_{-r_1}^{r_1} \Big|\frac12 \ve_1 s \Big|^2 \rd s = \frac14 \ve_1^2 \cdot \frac23 r_1^3 = \frac16 \ve_1^2 r_1^3.$$
Set $I = [t^*-\dt, t^*+\dt] \cap [0,T]$ so that $|I| \ge \dt$. Further integrating with respect to $(t,z)$ yields
$$\int_I \int_{|z|\le r_1,z\perp v} \int_{-r_1}^{r_1} \abs{\psi_t(z,s)}^2 \rd s \, \rd z \, \rd t \ge \frac16 \ve_1^2 r_1^3 \, \omega_{d-1} r_1^{d-1} \dt = c\, \ve_1^{d+5},$$
where $\omega_{d-1}$ denotes the volume of the unit ball in $\RR^{d-1}$. Using \eqref{eq:Aronson} and \(|x|\leq |x^*| + |z+sv|\leq R + \frac{1}{8K}\),
\begin{align*}
& \quad \inttx{\abs{\phi(t,x)}^2 \rho(t,x)} \ge \int_0^T \int_{B_{R+\frac{1}{8K}}} \abs{\phi(t,x)}^2 \rho(t,x) \,\rd x \,\rd t  \\
&\ge c_l \exp(-C_l(R+\tfrac{1}{8K})^2)\int_I \int_{|z|\le r_1,z\perp v} \int_{-r_1}^{r_1} \abs{\psi_t(z,s)}^2 \rd s \, \rd z \, \rd t \ge c\, \ve_1^{d+5} \ge c \, \ve^{d+5}.
\end{align*}

We remark that \emph{Argument 2} has a similar sprit to a special case of the Gagliardo--Nirenberg interpolation inequality \cite{nirenberg1959elliptic}, where a small $L^2$ norm of $\phi$ implies a small $L^2$ norm of $\nx \phi$, provided that the higher order derivatives are bounded. 
We also remark that, these two arguments require small values of $\ve$. 
In cases where $|\phi(t,x)| > 1,\, \forall (t,x) \in [0,T] \times B_R$ or $|\nx\phi(t,x)| > 1,\, \forall (t,x) \in [0,T] \times B_R$, we can always show that $\norm{\phi}_\mao$ has a positive lower bound of order $\mO(1)$, so that \eqref{eq:Lip_gradV} directly implies \eqref{eq:superlinear_growth_lem}.

Combining \emph{Argument 1} and \emph{Argument 2}: for any $(t,x) \in [0,T] \times B_R$ such that $|\phi(t,x)| + |\nx \phi(t,x)| = \ve$,
$$\norm{\phi}^2_{\mao} = \inttx{ \abs{\phi(t,x)}^2 \rho(t,x)} \ge c \, \ve^{d+5},$$
which concludes the proof of the claim \eqref{eq:claim_superlinear}.

Finally, combining all previous estimations yields
\begin{align*}
& \quad \norm{\nx V\ma - \nx V\ms}_\mao^2  \le C \inttx{ (1+|x|^2) \, |\Phi(t,x)|^2 \rho(t,x)} \\
& \le C\int_0^T \int_{B_R} |\phi(t,x)|^2 \parentheses{|\phi(t,x)| + |\nx \phi(t,x)|}^2 \rho(t,x) \,\rd x\,\rd t \\
& \le C \norm{\phi}_\mao^{\frac{4}{d+5}} \int_0^T \int_{B_R} |\phi(t,x)|^2 \rho(t,x) \,\rd x\,\rd t \le C \norm{\phi}_\mao^{2+\frac{4}{d+5}} = C \norm{\alpha - \alpha\ms}_\mao^{2+\frac{4}{d+5}},
\end{align*}
concluding the proof of \eqref{eq:superlinear_growth_lem}. 
\end{proof}

\subsection{Effect of OTGP flow}\label{sec:actor_distribution_update}
In this section, we show \eqref{eq:actor_distribution_update}, which is stated as Lemma~\ref{lem:actor_distribution_update} below. 
\begin{lem}\label{lem:actor_distribution_update}
Under the conditions of Theorem~\ref{thm:actor_convergence},
\begin{equation}\label{eq:actor_distribution_update_lem}
\dfrac{\rd}{\rd\tau} \parentheses{J^{\mu^\tau}[\alpha] - J^{\mu^\tau}[\alpha']}\Big|_{\alpha=\alpha^\tau, \alpha'=\alpha^{\mu^\tau,*}} \le C \beta_\mu \,\big\|\alpha^\tau - \alpha^{\mu^\tau,*}\big\|_\mato^2.
\end{equation}
\end{lem}
\begin{proof}
By Lemma~\ref{lem:cost_gap},
\begin{equation}\label{eq:cost_gap_integration}
\begin{aligned}
& J^{\mu^\tau}[\alpha^\tau] - J^{\mu^\tau}[\alpha\mts] = - \int_0^T \int_{\RR^d} \int_0^1 \int_0^u (\alpha^\tau(s,x) - \alpha\mts(s,x))\tp \na^2 H\big( s,x,\mu_s^{\textcolor{red}{\tau}},  \\
& v\alpha^\tau(s,x) + (1-v)\alpha\mts(s,x),-\nx V^{\mu^{\textcolor{red}{\tau}}, \alpha^{\mu^\tau,*}}(s,x) \big) (\alpha^\tau(s,x) - \alpha\mts(s,x)) \, \rd v \, \rd u \, \rho^{\mu^{\textcolor{red}{\tau}},\alpha^\tau}(s,x) \, \rd x \,\rd s.
\end{aligned}
\end{equation}
Based on \eqref{eq:cost_gap_integration}, where the corresponding $\tau$s hit by the differentiation in \eqref{eq:actor_distribution_update_lem} (those within \(\mu^\tau\)) are colored in \textcolor{red}{red}, the derivative~\eqref{eq:actor_distribution_update_lem} can be decomposed into three parts
\begin{equation}\label{eq:costgap_split_3terms}
\dfrac{\rd}{\rd\tau} \parentheses{J^{\mu^\tau}[\alpha] - J^{\mu^\tau}[\alpha']}\Big|_{\alpha=\alpha^\tau, \alpha'=\alpha^{\mu^\tau,*}} = (\rom{1}) + (\rom{2}) + (\rom{3}).
\end{equation}
$(\rom{1})$ addresses the \(\tau\)-dependence through the third argument of $\na^2 H$. $(\rom{2})$ addresses the \(\tau\)-dependence through the first superscript of $\nx V^{\mu^{\tau}, \alpha^{\mu^\tau,*}}$ in the fifth argument of $\na^2 H$. $(\rom{3})$ addresses the \(\tau\)-dependence through the first superscript of the density $\rho\mat$. 
Note that we use the notation  $V^{\mu^{\tau}, \alpha^{\mu^\tau,*}}$ instead of $V\mts$ to clearly distinguish the \(\tau\)-dependence of the distribution and the control components.

In the following context, we estimate each of the three parts separately.

\medskip
\noindent \emph{Step 1}. By Lemma~\ref{lem:value_growth}, $|\nx V^{\mu^\tau, \alpha^{\mu^\tau,*}}(s,x)| \le C(1+|x|)$. 
For notational simplicity, we temporarily fix $s$, $x$, $\alpha := v\alpha^\tau(s,x) + (1-v)\alpha\mts(s,x)$ and $p:=-\nx V^{\mu^\tau, \alpha^{\mu^\tau,*}}(s,x)$. 
Differentiating $\na^2H(s,x,\mu^\tau_s, \alpha, p)$ with respect to $\tau$ yields
\begin{equation}\label{eq:pta_hessH_bound}
\begin{aligned}
& \quad \Big| \dfrac{\rd}{\rd\tau} \na^2 H(s,x,\mu^\tau_s, \alpha, p) \Big| = \Big| \lim_{\dta \to 0} \dfrac{1}{\dta} \big(\na^2 H(s,x,\mu^{\tau+\dta}_s, \alpha, p) - \na^2 H(s,x,\mu^\tau_s, \alpha, p)\big) \Big| \\
& \le \lim_{\dta \to 0} \dfrac{1}{|\dta|} \Big[ \Big| \na^2 f(s,x,\mu^{\tau+\dta}_s, \alpha) - \na^2 f(s,x,\mu^\tau_s, \alpha) \Big|  + \sum_{i=1}^d \Big| \na^2 b_i(s,x,\mu^{\tau+\dta}_s, \alpha) - \na^2 b_i(s,x,\mu^\tau_s, \alpha) \Big| \, |p_i| \,\Big] \\
& \le \lim_{\dta \to 0} \dfrac{1}{|\dta|} C \, W_2(\mu^{\tau+\dta}_s, \mu^\tau_s) \, (1+|x|) = C(1+|x|) \, \beta_\mu \, W_2(\rho\mat_s, \mu^\tau_s) \le C(1+|x|) \, \beta_\mu,
\end{aligned}
\end{equation}
where the last equality follows from Lemma~\ref{lem:constant_speed}.
Here, we are also using the uniform (in \(\tau\) and \(t\)) boundedness of \(W_2(\rho\mat_t, \mu^\tau_t)\le W_2(\rho\mat_t, \delta_0) + W_2(\delta_0, \mu^\tau_t)\), which is implied by \(\mu^\tau \in \mM\) and the Aronson-type bound~\eqref{eq:Aronson}.
Therefore, term $(\rom{1})$ satisfies
\begin{equation}\label{eq:dJgap_term1}
\begin{aligned}
(\rom{1}) & \le \int_0^T \int_{\RR^d} \int_0^1 \int_0^u  \Big| \dfrac{\rd}{\rd\tau} \na^2 H(s,x,\mu^\tau_s, \alpha, p) \,\Big| \abs{\alpha^\tau(s,x) - \alpha\mts(s,x)}^2 \rd v \, \rd u \, \rho^{\mu^{\tau},\alpha^\tau}(s,x) \, \rd x \,\rd s\\
& \le C  \beta_\mu \int_0^T \int_{\RR^d} \int_0^1 \int_0^u (1+|x|) \abs{\alpha^\tau(s,x) - \alpha\mts(s,x)}^2 \rd v \, \rd u \, \rho^{\mu^{\tau},\alpha^\tau}(s,x) \, \rd x \,\rd s \\
& \le C  \beta_\mu \int_0^T \int_{\RR^d} \abs{\alpha^\tau(s,x) - \alpha\mts(s,x)}^2 \rho^{\mu^{\tau},\alpha^\tau}(s,x) \, \rd x \,\rd s  = C  \beta_\mu \norm{\alpha^\tau - \alpha\mts}^2_\mato,
\end{aligned}
\end{equation}
where the last inequality is due to $\alpha^\tau - \alpha\mts \in \mC$.

\medskip
\noindent \emph{Step 2}. Motivated by~\eqref{eq:value_growth}, we first show
$$\Big|\dfrac{\rd}{\rd\tau} \partial_{x_1} V^{\mu^\tau,\alpha}(s,x)\Big| \le C \beta_\mu (1+|x|).$$

Let $\dta,\delta \in \RR$, denote $\mu:=\mu^\tau$, $\alpha:=\alpha^\tau$, $\mu':=\mu^{\tau+\dta}$ and $\xd:= x + \delta e_1$ for $x \in \RR^d$.
\begin{equation}\label{eq:pta_px1V}
\dfrac{\rd}{\rd\tau} \partial_{x_1} V^{\mu^\tau,\alpha}(s,x) = \lim_{\dta\to0} \lim_{\delta\to0} \frac{1}{\dta}\frac{1}{\delta} \sqbra{ \big(V^{\mu',\alpha}(s,\xd) - V^{\mu',\alpha}(s,x)\big) - \parentheses{V\ma(s,\xd) - V\ma(s,x)}}.
\end{equation}
Denote $b_\alpha(t,x,\mu_t) := b(t,x,\mu_t,\alpha(t,x))$ and let $f_\alpha$ be similarly defined.
Define $x_t,\xd_t, x'_t,\xdp_t$ as state processes driven by the same Brownian motion that have initial conditions $x,\xd,x,\xd$ at time \(s\), with respective drifts
$$b_t:= b_\alpha(t,x_t,\mu_t), ~ \bdt:= b_\alpha(t,\xd_t,\mu_t), ~ b'_t:=b_\alpha(t,x_t',\mu_t'), ~ \bdtp:= b_\alpha(t,\xdp_t,\mu_t),$$
and diffusions $\sigma_t, \sdt, \sigma'_t,\sdtp$ that are similarly defined. 
The processes $f_t, \fdt, f'_t,\fdtp$ are defined in a similar manner. 
By definition \eqref{eq:value_function}, 
\begin{equation}\label{eq:ptaVx_4terms}
\begin{aligned}
& \quad \big(V^{\mu',\alpha}(s,\xd) - V^{\mu',\alpha}(s,x)\big) - \big(V\ma(s,\xd) - V\ma(s,x)\big) \\
& = \EE\Big[\int_s^T \big[(\fdtp - f'_t) - (\fdt - f_t)\big]\, \rd t + (g(\xdp_T,\mu'_T) - g(x'_T,\mu'_T)) + (g(\xd_T,\mu_T) - g(x_T,\mu_T)) \Big].
\end{aligned}
\end{equation}
By the mean value theorem,
\begin{equation}\label{eq:fdelta_f}
\begin{aligned}
\fdt - f_t &= f_\alpha(t,\xd_t,\mu_t) - f_\alpha(t,x_t,\mu_t) = \int_0^1 (\xd_t-x_t)\tp \, \nx f_\alpha(t,(1-u)x_t + u\xd_t,\mu_t) \,\rd u,\\
\fdtp - f'_t &= f_\alpha(t,\xdp_t,\mu'_t) - f_\alpha(t,x'_t,\mu'_t) = \int_0^1 (\xdp_t-x'_t)\tp \, \nx f_\alpha(t,(1-u)x'_t + u\xdp_t,\mu'_t) \,\rd u.
\end{aligned}
\end{equation}
By Assumption~\ref{assu:basic}, $\nx f_\alpha = \nx f + \nx\alpha\tp \na f$ is Lipschitz in $x,\mu$ and grows at most linearly in $|x|$.
Subtracting the two equations in \eqref{eq:fdelta_f} yields
\begin{equation}\label{eq:f_4terms}
\begin{aligned}
& \quad \abs{(\fdtp - f'_t) - (\fdt - f_t)} \\
& \le \int_0^1 \abs{(\xdp_t-x'_t) - (\xd_t-x_t)} \, \abs{\nx f_\alpha(t,(1-u)x'_t + u\xdp_t,\mu'_t)} \, \rd u\\
& \quad + \int_0^1 |\xd_t - x_t| \, \abs{\nx f_\alpha(t,(1-u)x'_t + u\xdp_t,\mu'_t) - \nx f_\alpha(t,(1-u)x_t + u\xd_t,\mu_t)} \, \rd u \\
& \le C \Big[(1+|x_t'|) \abs{(\xdp_t-x'_t) - (\xd_t-x_t)} + |\xd_t-x_t| \, \parentheses{|x'_t-x_t| + |\xdp_t-\xd_t| + W_2(\mu'_t,\mu_t)} \Big].
\end{aligned}
\end{equation}
Taking expectations on both sides yields 
\begin{equation}\label{eq:f_4terms_bound}
\begin{aligned}
& \quad \EE\big[ \big|(\fdtp - f'_t) - (\fdt - f_t)\big| \big]\\
& \le C \Big[ \EE\big[(1+|x_t'|)^2\big]^{\frac12} \, \EE\big[ |(\xdp_t-x'_t) - (\xd_t-x_t)|^2\big]^{\frac12} \\
& \quad + \EE\big[|\xd_t-x_t|^2\big]^{\frac12} \Big(\EE\big[|x'_t-x_t|^2 + |\xdp_t-\xd_t|^2\big] + W_2(\mu'_t,\mu_t)^2\Big)^{\frac12} \Big] \\
& \le C \Big[(1+|x|) \, |x-\xd|\, \mW_2(\mu,\mu') + |x-\xd| \, \mW_2(\mu_t,\mu'_t) \Big] \\
& \le C \, \delta \, [(1+|x|)\, \mW_2(\mu,\mu') + W_2(\mu'_t,\mu_t)],
\end{aligned}
\end{equation}
where Gr\"onwall's inequalities \eqref{eq:Gronwall_square}, \eqref{eq:Gronwall_diff_square}, \eqref{eq:Gronwall_4terms} are applied. Similarly,
\begin{equation}\label{eq:g_4terms_bound}
\EE \Big[\abs{(g(\xdp_T,\mu'_T) - g(x'_T,\mu'_T)) + (g(\xd_T,\mu_T) - g(x_T,\mu_T))} \Big] \le C \, \delta \, [(1+|x|)\, \mW_2(\mu,\mu')+ W_2(\mu'_T,\mu_T)].
\end{equation}
Substituting \eqref{eq:f_4terms_bound} and \eqref{eq:g_4terms_bound} into \eqref{eq:ptaVx_4terms} yields
\begin{equation}\label{eq:ptaVx_4terms_bound}
\Big|\big(V^{\mu',\alpha}(s,\xd) - V^{\mu',\alpha}(s,x)\big) - \big(V\ma(s,\xd) - V\ma(s,x)\big) \Big| \le C \, \delta \, [(1+|x|)\, \mW_2(\mu,\mu')+ W_2(\mu'_T,\mu_T)]. 
\end{equation}
Substituting \eqref{eq:ptaVx_4terms_bound} into \eqref{eq:pta_px1V} yields
\begin{equation}\label{eq:pta_px1V_bound}
\begin{aligned}
& \quad \bigg|\dfrac{\rd}{\rd\tau} \partial_{x_1} V^{\mu^\tau,\alpha}(s,x)\bigg| \le C \Big[(1+|x|) \lim_{\dta\to0} \frac{1}{|\dta|} \mW_2(\mu^\tau,\mu^{\tau+\dta}) + \lim_{\dta\to0} \frac{1}{|\dta|} W_2(\mu^\tau_T,\mu^{\tau+\dta}_T) \Big] \\
& = C \beta_\mu [(1+|x|) \, \mW_2(\mu^\tau, \rho\mat) + \beta_\mu W_2(\mu^\tau_T, \rho\mat_T)] \le C \beta_\mu (1+|x|)
\end{aligned}
\end{equation}
where Lemma~\ref{lem:constant_speed} is applied. Repeating the argument \eqref{eq:pta_px1V_bound} for each dimension yields $\bigg|\dfrac{\rd}{\rd\tau} \nx V^{\mu^\tau,\alpha}(s,x)\bigg| \le C \beta_\mu (1+|x|)$. By Assumption~\ref{assu:basic}, $\alpha^\tau - \alpha\mts \in \mC$ and previous estimations,
\begin{equation}\label{eq:dJgap_term2}
\begin{aligned}
(\rom{2}) & \le \int_0^T \int_{\RR^d} \int_0^1 \int_0^u \abs{\alpha^\tau(s,x) - \alpha\mts(s,x)}^2 \abs{\na^2 b(s,x,\mu^\tau_s,  v\alpha^\tau(s,x) + (1-v)\alpha\mts(s,x)) }\\
& \quad \quad \Big|\dfrac{\rd}{\rd\tau} \nx V^{\mu^\tau,\alpha}(s,x)\Big| \, \rd v \, \rd u \, \rho^{\mu^{\tau},\alpha^\tau}(s,x) \, \rd x \,\rd s \\
& \le C \beta_\mu \int_0^T \int_{\RR^d} \abs{\alpha^\tau(s,x) - \alpha\mts(s,x)}^2 (1+|x|) \, \rho^{\mu^{\tau},\alpha^\tau}(s,x) \, \rd x \,\rd s \le C  \beta_\mu \, \big\|\alpha^\tau - \alpha\mts\big\|^2_\mato.
\end{aligned}
\end{equation}

\medskip
\noindent \emph{Step 3}. We estimate $(\rom{3})$. Define 
$$q^\tau(t,x) := \dfrac{\rd}{\rd\tau} \log \rho^{\mu^\tau,\alpha}(t,x)\Big|_{\alpha=\alpha^\tau}.$$
We claim that: $\abs{q^\tau(t,x)} \le C \beta_\mu (1+|x|^2)$. 

We use shorthand notations $b$, $D$, $\rho$ to denote  $b(t,x,\mu^\tau_t,\alpha^\tau(t,x))$, $D(t,x,\mu^\tau_t)$, $\rho\mat(t,x)$. Since $\rho$ satisfies the FP equation \eqref{eq:FokkerPlanck},
\begin{equation}\label{eq:logrho_PDE}
\begin{aligned}
& \quad \pt \log\rho = \pt\rho / \rho = -\frac{\nx\cdot(b\rho)}{\rho} + \frac{\nx^2:(D\rho)}{\rho}\\
& = -\nx \cdot b + \nx^2 : D - b\tp\nx\log\rho + 2 (D\cdot \nx)\tp \nx\log\rho + \Tr[D (\nx^2\log\rho + \nx\log\rho \,\nx\log\rho\tp)],
\end{aligned}
\end{equation}
where $\nx^2:$ denotes the matrix inner product (i.e., \(\langle A,B\rangle := \Tr(A\tp B)\)) between the Hessian operator and a matrix-valued function. Differentiating with respect to $\tau$ yields
\begin{equation}\label{eq:q_PDE}
\begin{aligned}
\pt q^\tau &=  -\nx \cdot \pta b + \nx^2 : \pta D - \pta b\tp\nx\log\rho + 2 (\pta D\cdot\nx)\tp \nx\log\rho - b\tp \nx q^\tau + 2(D\cdot \nx)\nx q^\tau \\
& \quad +  \Tr[\pta D (\nx^2\log\rho + \nx\log\rho \,\nx\log\rho\tp)] +  \Tr[D (\nx^2 q^\tau + 2 \nx\log\rho \,\nx q^{\tau\top})] \\
& =: a_q + b_q\tp \nx q^\tau + \Tr[D \nx^2 q^\tau],
\end{aligned}
\end{equation}
which is a linear parabolic equation for $q^\tau$ with initial condition $q^\tau(0,x) = 0$. By the Lipschitz condition of $b, \nx b, \nx D, \nx^2 D$ in $\mu$ and Lemma~\ref{lem:constant_speed}, we have
$$|\nx \cdot \pta b|, \, |\nx^2 : \pta D|, \,  |\pta b|, \, |\nx\pta D| \le K \lim_{\dta \to 0} \frac{1}{|\dta|} W_2(\mu^\tau_t, \mu^{\tau+\dta}_t)  \le C \beta_\mu.$$
By the logarithmic Aronson bounds, $|a_q| \le C\beta_\mu(1+|x|^2)$, $b_q \le C\beta_\mu(1+|x|)$.
Applying standard maximum principle with a quadratic barrier function \cite{krylov1996lectures} to the PDE $\pt q^\tau = a_q + b_q\tp \nx q^\tau + \Tr[D \nx^2 q^\tau]$ with initial condition $q^\tau(0,x)=0$ yields $|q^\tau(t,x)| \le C \beta_\mu (1+|x|^2)$, which implies $\tfrac{\rd}{\rd\tau} \rho^{\mu^{\tau},\alpha}(s,x)\big|_{\alpha=\alpha^\tau} \le C \beta_\mu (1+|x|^2) \rho^{\mu^{\tau},\alpha^\tau}(s,x)$. 
By~\eqref{eq:na2_Hamiltonian_bound},
\begin{equation}\label{eq:dJgap_term3}
\begin{aligned}
(\rom{3}) &\le \int_0^T \int_{\RR^d} \int_0^1 \int_0^u (1+|x|) \abs{\alpha^\tau(s,x) - \alpha\mts(s,x)}^2  \, \rd v \, \rd u \, \Big|\dfrac{\rd}{\rd\tau} \rho^{\mu^{\tau},\alpha}(s,x)\big|_{\alpha=\alpha^\tau}\Big| \, \rd x \,\rd s \\
& \le C \beta_\mu \int_0^T \int_{\RR^d} (1+|x|) \abs{\alpha^\tau(s,x) - \alpha\mts(s,x)}^2 (1+|x|^2) \, \rho\mat(s,x) \, \rd x \,\rd s \\
& \le C \beta_\mu \int_0^T \int_{\RR^d} \abs{\alpha^\tau(s,x) - \alpha\mts(s,x)}^2 \, \rho\mat(s,x) \, \rd x \,\rd s = C  \beta_\mu \norm{\alpha^\tau - \alpha\mts}^2_\mato.
\end{aligned}
\end{equation}

Combining \eqref{eq:dJgap_term1}, \eqref{eq:dJgap_term2}, and \eqref{eq:dJgap_term3} yields
$$\dfrac{\rd}{\rd\tau} \parentheses{J^{\mu^\tau}[\alpha] - J^{\mu^\tau}[\alpha']}\Big|_{\alpha=\alpha^\tau, \alpha'=\alpha^{\mu^\tau,*}} \le C \beta_\mu \norm{\alpha^\tau - \alpha^{\mu^\tau,*}}_\mato^2,$$
which concludes the proof.
\end{proof}

\section{Proofs for the critic}\label{sec:proof_critic}
\subsection{Proof of Proposition~\ref{prop:critic_loss}}
\begin{proof}[Proof of Proposition~\ref{prop:critic_loss}]
Substituting \eqref{eq:g_Ito} into \eqref{eq:critic_loss} yields
\begin{equation*}
\begin{aligned}
\quad \mL_c &= \frac12 \EE\Big[ \Big(\mV_0(X_0\ma) -V\ma(0,X_0\ma) + \int_0^T \parentheses{\mG(t,X_t\ma) - \nx V\ma(t, X_t\ma)}\tp \sigma(t,X_t\ma,\mu_t) \,\rd W_t\Big)^2\Big] \\
& = \frac12 \EE \Big[ \parentheses{\mV_0(X_0\ma) - V\ma(0, X_0\ma)}^2 + \int_0^T \abs{ \sigma(t,X_t\ma,\mu_t)\tp \parentheses{  \mG(t, X_t\ma) - \nx V\ma(t,X_t\ma) } }^2 \rd t \Big]\\
& = \frac12 \int_{\RR^d} \parentheses{\mV_0(x) - V\ma(0,x)}^2 \rho_0(x)\,\rd x \\
& \quad + \frac12 \inttx{ \abs{\sigma(t,x,\mu_t)\tp \parentheses{\mG(t, x) - \nx V\ma(t,x)}}^2 \rho\ma(t,x)},
\end{aligned}
\end{equation*}
where the second equality follows from the It\^o isometry. This validates \eqref{eq:critic_loss_2terms} and the derivatives \eqref{eq:critic_derivative} follow directly from the definition. 
Note that a similar argument also appears in \cite{zhou2021actor}.
\end{proof}

\subsection{Proof of Theorem~\ref{thm:critic_convergence}.}
\begin{proof}[Proof of Theorem~\ref{thm:critic_convergence}]
Motivated by Proposition~\ref{prop:critic_loss}, define
\begin{align}
\mL^\tau_0 &:= \frac12 \int_{\RR^d} \parentheses{\mV_0^\tau(x) - V\mat(0,x)}^2 \rho_0(x) \,\rd x, \label{eq:L0_critic}\\
\mL^\tau_1 &:= \frac12 \inttx{ \abs{\sigma(t,x,\mu_t^\tau)\tp \parentheses{ \mG^\tau(t,x) - \nx V\mat(t,x)}}^2\rho\mat(t,x)}. \label{eq:L1_critic}
\end{align}

\noindent\emph{Step 1}. We bound the derivative of $\mL_0^\tau$ in $\tau$. By definition \eqref{eq:V0_flow},
\begin{equation}\label{eq:pta_L0}
\begin{aligned}
& \quad \pta \mL_0^\tau = \int_{\RR^d} \rho_0(x) \parentheses{V\mat(0,x) - \mV_0^\tau(x)} \parentheses{\dfrac{\rd}{\rd \tau}V\mat(0,x) - \pta\mV_0^\tau(x)} \,\rd x \\
& = \int_{\RR^d} \rho_0(x) \parentheses{V\mat(0,x) - \mV_0^\tau(x)} \dfrac{\rd}{\rd \tau}V\mat(0,x) \,\rd x - 2\beta_c \mL_0^\tau \\
& \le \int_{\RR^d} \rho_0(x) \sqbra{ \dfrac{\beta_c}{4}\parentheses{V\mat(0,x) - \mV_0^\tau(x)}^2 + \dfrac{1}{\beta_c} \abs{\dfrac{\rd}{\rd \tau}V\mat(0,x)}^2} \,\rd x - 2\beta_c \mL_0^\tau \\
& = \dfrac{1}{\beta_c} \int_{\RR^d} \rho_0(x) \abs{\dfrac{\rd}{\rd \tau}V\mat(0,x)}^2 \,\rd x - \frac32 \beta_c \mL_0^\tau.
\end{aligned}
\end{equation}
By~\eqref{eq:Lip_V0} from Lemma~\ref{lem:Value_Lipschitz}, Lemma~\ref{lem:constant_speed} and \eqref{eq:actor_flow},
\begin{equation}\label{eq:pta_V0_bound}
\begin{aligned}
& \quad \int_{\RR^d} \rho_0(x) \abs{\dfrac{\rd}{\rd \tau}V\mat(0,x)}^2 \,\rd x \\
& = \int_{\RR^d} \rho_0(x) \abs{ \lim_{\dta\to 0} \dfrac{1}{\dta}   \parentheses{V^{\mu^{\tau+\dta},\alpha^{\tau+\dta}}(0,x) - V\mat(0,x)}   }^2 \,\rd x \\
& \le \liminf_{\dta\to 0} \dfrac{1}{\dta^2} \int_{\RR^d} \rho_0(x) \parentheses{V^{\mu^{\tau+\dta},\alpha^{\tau+\dta}}(0,x) - V\mat(0,x)}^2 \,\rd x \\
& = \liminf_{\dta\to 0} \dfrac{1}{\dta^2} \norm{V^{\mu^{\tau+\dta},\alpha^{\tau+\dta}}(0,\cdot) - V\mat(0,\cdot)}^2_{\rho_0} \\
& \le C\liminf_{\dta\to 0} \dfrac{1}{\dta^2} \parentheses{ \mW_2(\mu^{\tau+\dta}, \mu^\tau)^2 + W_2(\mu^{\tau+\dta}_T, \mu^\tau_T)^2 + \norm{\alpha^{\tau+\dta} - \alpha^\tau}^2_\mato} \\
& = C \sqbra{ \beta_\mu^2 \mW_2\parentheses{\mu^\tau, \rho\mat}^2 + \beta_\mu^2 W_2\parentheses{\mu^\tau_T, \rho\mat_T}^2 + \beta_a^2 \norm{\na H(t,x,\mu_t,\alpha^\tau(t,x), -\mG^\tau(t,x))}^2_\mato}.
\end{aligned}
\end{equation}
Substituting \eqref{eq:pta_V0_bound} into \eqref{eq:pta_L0} yields
\begin{equation}\label{eq:dL0_dtau}
\begin{aligned}
\pta \mL_0^\tau &\le - \frac32\beta_c\mL_0^\tau + \dfrac{C}{\beta_c} \Big[ \beta_\mu^2 \mW_2\parentheses{\mu^\tau, \rho\mat}^2 + \beta_\mu^2 W_2\parentheses{\mu^\tau_T, \rho\mat_T}^2 \\
& \quad + \beta_a^2 \norm{\na H(t,x,\mu_t,\alpha^\tau(t,x), -\mG^\tau(t,x))}^2_\mato\Big].
\end{aligned}
\end{equation}
Later, we will set $\beta_c$ sufficiently large relative to $\beta_a$ and $\beta_\mu$ (cf. \eqref{eq:speed_ratio}), so that the positive terms are offset by the decay of the other Lyapunov functions. A similar idea was applied in \cite{zhou2023single}.

\medskip
\noindent \emph{Step 2}. Next, we bound the derivative of $\mL^\tau_1$. 
We treat $\mL_1$ as a function of $\mu^\tau$, $\alpha^\tau$, and $\mG^\tau$, and define
$$\wt\mL_1(\mu,\alpha,\mG) := \frac12  \int_0^T \int_{\RR^d} \rho\ma(t,x) \abs{\sigma(t,x,\mu_t)\tp \parentheses{\nx V\ma(t, x) - \mG(t, x)}}^2 \,\rd x \,\rd t,$$
so that $\mL_1^\tau = \wt\mL_1(\mu^\tau,\alpha^\tau,\mG^\tau)$. 
The derivative $\pta \mL^\tau_1$ is  decomposed into two parts:
\begin{equation}\label{eq:pta_L1_split}
\pta \mL^\tau_1 = \dfrac{\rd}{\rd\tau}\wt\mL_1(\mu,\alpha,\mG^\tau) \Big|_{\mu=\mu^\tau, \alpha=\alpha^\tau} +  \dfrac{\rd}{\rd\tau}\wt\mL_1(\mu^\tau,\alpha^\tau,\mG) \Big|_{\mG=\mG^\tau}  =: (c\rom{1}) + (c\rom{2}),
\end{equation}
where $(c\rom{1})$ takes care of the \(\tau\)-dependence through $\mG^\tau$ and $(c\rom{2})$ deals with the \(\tau\)-dependence through $(\mu^\tau, \alpha^\tau)$. 
From the flow equation \eqref{eq:G_flow}, $(c\rom{1})$ satisfies
\begin{equation}\label{eq:critic_term1}
\begin{aligned}
- (c\rom{1}) & = 4\beta_c \int_0^T \int_{\RR^d} \rho\mat(t,x) \parentheses{\nx V\mat - \mG^\tau}\tp D(t,x,\mu^\tau_t)^2 \parentheses{\nx V\mat - \mG^\tau} \,\rd x \,\rd t \\
& \ge 4 \beta_c \sigma_0 \int_0^T \int_{\RR^d} \rho\mat(t,x) \parentheses{\nx V\mat - \mG^\tau}\tp D(t,x,\mu^\tau_t) \parentheses{\nx V\mat - \mG^\tau} \,\rd x \,\rd t \\
& = 4 \beta_c \sigma_0 \mL^\tau_1.
\end{aligned}
\end{equation}

For part $(c\rom{2})$, let $x_t := X_t\mat$ be the state process under $(\mu^\tau, \alpha^\tau)$. Denote $\sigma_t := \sigma(t,x_t,\mu_t^\tau)$, $p_t := \nx V\mat(t,x_t)$ and $G_t := \mG^\tau(t,x_t)$. We have
$$\wt\mL_1(\mu^\tau,\alpha^\tau,\mG^\tau) = \frac12\EE \sqbra{\int_0^T \abs{\sigma_t\tp(p_t-G_t)}^2 \,\rd t}.$$
For $\tau' > \tau$, denote by $x'_t := X_t\matp$ the state process under $(\mu^{\tau'}, \alpha^{\tau'})$ driven by the same Brownian motion, starting from the same initial condition $x_0'\overset{\mathrm{a.s.}}{=}x_0$. Denote $\sigma'_t := \sigma(t,x'_t,\mu_t^{\tau'})$, $p'_t := \nx V\matp(t,x'_t)$ and $G'_t := \mG^\tau(t,x'_t)$.
It is worth noting that \(G'_t\) uses $\mG^\tau$ instead of $\mG^{\tau'}$. Then,
\begin{equation}\label{eq:critic_term2}
\begin{aligned}
& \quad \abs{(c\rom{2})} = \bigg|\lim_{\tau'\to\tau} \dfrac{1}{\tau'-\tau} \parentheses{\wt\mL_1(\mu^{\tau'},\alpha^{\tau'},\mG^\tau) - \wt\mL_1(\mu^\tau,\alpha^\tau,\mG^\tau)} \bigg| \\
& = \frac12\bigg|\lim_{\tau'\to\tau} \dfrac{1}{\tau'-\tau} \EE \Big[ \int_0^T \parentheses{\abs{\sigma_t'^{\top}(p'_t-G'_t)}^2 - \abs{\sigma_t^{\top}(p_t-G_t)}^2}   \,\rd t \Big]\bigg| \\
& \le \frac12\lim_{\tau'\to\tau} \dfrac{1}{|\tau'-\tau|} \EE \Big[ \int_0^T  \abs{(\sigma_t'^{\top}p'_t - \sigma_t^{\top}p_t) - (\sigma_t'^{\top}G'_t - \sigma_t^{\top}G_t)}\cdot \abs{\sigma_t'^{\top}(p'_t-G'_t)+\sigma_t^{\top}(p_t-G_t)}  \,\rd t \Big] \\
& \le  \lim_{\tau'\to\tau} \dfrac{1}{|\tau'-\tau|} \EE \Big[ \int_0^T   \parentheses{\abs{\sigma_t'^{\top}p'_t - \sigma_t^{\top}p_t} + \abs{\sigma_t'^{\top}G'_t - \sigma_t^{\top}G_t}}\abs{\sigma_t^{\top}(p_t-G_t)}   \,\rd t \Big] \\
& \le \lim_{\tau'\to\tau} \dfrac{1}{|\tau'-\tau|} \EE \Big[ \int_0^T \dfrac{2}{\sigma_0\beta_c|\tau'-\tau|}  \parentheses{\abs{\sigma_t'^{\top}p'_t - \sigma_t^{\top}p_t} + \abs{\sigma_t'^{\top}G'_t - \sigma_t^{\top}G_t}}^2 + \dfrac{\sigma_0\beta_c|\tau'-\tau|}{2} \abs{\sigma_t^{\top}(p_t-G_t)}^2   \,\rd t \Big] \\
& \le \lim_{\tau'\to\tau} \dfrac{1}{|\tau'-\tau|^2} \EE \Big[ \int_0^T \dfrac{4}{\sigma_0\beta_c}  \parentheses{\abs{\sigma_t'^{\top}p'_t - \sigma_t^{\top}p_t}^2 + \abs{\sigma_t'^{\top}G'_t - \sigma_t^{\top}G_t}^2} \,\rd t\Big] + \sigma_0\beta_c \mL_1^\tau.
\end{aligned}
\end{equation}
By \eqref{eq:Lip_sigma_gradV} from Lemma~\ref{lem:Value_Lipschitz},
\begin{equation}\label{eq:critic_term2a}
\EE \Big[ \int_0^T  \abs{\sigma_t'^{\top}p'_t - \sigma_t^{\top}p_t}^2\,\rd t\Big] \le C \parentheses{ \mW_2(\mu^{\tau'},\mu^\tau)^2 + W_2(\mu^{\tau'}_T,\mu^\tau_T)^2 + \norm{\alpha^{\tau'} - \alpha^\tau}_\mato^2}. 
\end{equation}
By Assumption~\ref{assu:flow_in_class} and \eqref{eq:Gronwall_xsq_diffxsq011} from Lemma~\ref{lem:Gronwall_square},
\begin{equation}\label{eq:critic_term2b}
\begin{aligned}
& \quad \EE \Big[ \int_0^T  \abs{\sigma_t'^{\top}G'_t - \sigma_t^{\top}G_t}^2 \,\rd t\Big] \le 2\EE \Big[ \int_0^T  \parentheses{\abs{\sigma_t'^{\top}(G'_t - G_t)}^2 + \abs{(\sigma_t' - \sigma_t)^{\top}G_t}^2 }\,\rd t\Big] \\
& \le 2\EE \Big[ \int_0^T \parentheses{K^2 |x'_t-x_t|}^2 + \parentheses{K[|x'_t-x_t| + W_2(\mu^\tau_t,\mu^{\tau'}_t)]K(1+|x_t|)}^2 \,\rd t\Big] \\
&\le C \parentheses{ \mW_2(\mu^{\tau'},\mu^\tau)^2  + \norm{\alpha^{\tau'} - \alpha^\tau}_\mato^2}.
\end{aligned}
\end{equation}
Substituting \eqref{eq:critic_term2a} and \eqref{eq:critic_term2b} into \eqref{eq:critic_term2} yields
\begin{equation}\label{eq:critic_term2_final}
\begin{aligned}
\abs{(c\rom{2})} &\le \dfrac{C}{\beta_c} \lim_{\tau'\to\tau} \dfrac{1}{|\tau'-\tau|^2} \parentheses{ \mW_2(\mu^{\tau'},\mu^\tau)^2 + W_2(\mu^{\tau'}_T,\mu^\tau_T)^2 + \norm{\alpha^{\tau'} - \alpha^\tau}_\mato^2} + \dfrac{1}{2}\sigma_0\beta_c \mL_1^\tau \\
& = \dfrac{C}{\beta_c} \Big[ \beta_\mu^2 \mW_2\parentheses{\mu^\tau, \rho\mat}^2 + \beta_\mu^2 W_2\parentheses{\mu^\tau_T, \rho\mat_T}^2 \\
& \quad + \beta_a^2 \norm{\na H(t,x,\mu_t,\alpha^\tau(t,x), -\mG^\tau(t,x))}^2_\mato\Big] + \sigma_0\beta_c \mL_1^\tau,
\end{aligned}
\end{equation}
where Lemma~\ref{lem:constant_speed} is applied.
Substituting \eqref{eq:critic_term1} and \eqref{eq:critic_term2_final} into \eqref{eq:pta_L1_split} yields
\begin{equation}\label{eq:dL1_dtau}
\begin{aligned}
\pta \mL^\tau_1  \le -3 \sigma_0\beta_c \mL_1^\tau + \dfrac{C}{\beta_c} \Big[ \beta_\mu^2 \mW_2\parentheses{\mu^\tau, \rho\mat}^2 + \beta_\mu^2 W_2\parentheses{\mu^\tau_T, \rho\mat_T}^2 \\
 + \beta_a^2 \norm{\na H(t,x,\mu_t,\alpha^\tau(t,x), -\mG^\tau(t,x))}^2_\mato\Big].
\end{aligned}
\end{equation}

Combining \eqref{eq:dL0_dtau} and \eqref{eq:dL1_dtau} yields
\begin{equation*}
\begin{aligned}
\pta \mL^\tau_c \le -c_c\beta_c \mL_c^\tau + \dfrac{C_c}{\beta_c} \Big[ \beta_\mu^2 \mW_2\parentheses{\mu^\tau, \rho\mat}^2 + \beta_\mu^2 W_2\parentheses{\mu^\tau_T, \rho\mat_T}^2 \\
+ \beta_a^2 \norm{\na H(t,x,\mu_t,\alpha^\tau(t,x), -\mG^\tau(t,x))}^2_\mato\Big],
\end{aligned}
\end{equation*}
which concludes the proof.
\end{proof}

\section{Proof for the distribution: Theorem~\ref{thm:distribution_convergence}}\label{sec:proof_distribution}
\begin{proof}[Proof of Theorem~\ref{thm:distribution_convergence}]
We bound the derivative \(\pta \mL^\tau_\mu\) by taking two steps.

\emph{Step 1}. We bound the derivative of $\frac{1}{2} d_\beta(\mu^\tau, \rho\mat)^2$ with respect to $\tau$, which is further decomposed into $3$ terms that address different sources of \(\tau\)-dependence: the dependence on \(\mu^\tau\) through the first argument of $d_\beta(\mu^\tau, \rho\mat)$, the dependence on $\mu^\tau$ through the density $\rho\mat$, and the dependence on $\alpha^\tau$ through $\rho\mat$:
\begin{equation}\label{eq:Lmu_first_part}
\begin{aligned}
& \quad\frac{1}{2}\dfrac{\rd}{\rd \tau}  d_\beta(\mu^\tau, \rho\mat)^2 = (\mu\rom{1}) + (\mu\rom{2}) + (\mu\rom{3}) \\
& := \frac{1}{2}\dfrac{\rd}{\rd \tau}  d_\beta(\mu^\tau, \nu)^2 \Big|_{\nu=\rho\mat} + \frac{1}{2}\dfrac{\rd}{\rd \tau}  d_\beta(\mu, \rho^{\mu^\tau,\alpha})^2 \Big|_{\mu = \mu^\tau, \alpha=\alpha^\tau} + \frac{1}{2}\dfrac{\rd}{\rd \tau}  d_\beta(\mu, \rho^{\mu,\alpha^\tau})^2 \Big|_{\mu = \mu^\tau}.
\end{aligned}
\end{equation}

For $(\mu\rom{1})$, by \cite[Theorem 5.24]{santambrogio2015optimal}, for any \(t\in[0,T]\),
\begin{align*}
& \quad \frac12\dfrac{\rd}{\rd\tau}  W_2(\mu^\tau_t, \nu_t)^2 \Big|_{\nu_t = \rho\mat_t} = - \beta_\mu \int_{\RR^d} |\nx\varphi^\tau_t(x)|^2 \rd\mu^\tau_t(x) \\
&= - \beta_\mu \int_{\RR^d} |x - T^\tau_t(x)|^2 \rd\mu^\tau_t(x) = - \beta_\mu W_2(\mu^\tau_t,\rho\mat_t)^2.
\end{align*}
Multiplying \(e^{-2\beta t}\) and integrating both sides with respect to $t$ yield
\begin{equation}\label{eq:distribution_term1}
(\mu\rom{1}) = - \beta_\mu \int_0^T e^{-2\beta t} \, W_2(\mu^\tau_t,\rho\mat_t)^2 \,\rd t = - \beta_\mu d_\beta(\mu^\tau, \rho^{\mu^\tau,\alpha^\tau})^2.
\end{equation}

Next, we estimate $(\mu\rom{2})$. By definition,
\begin{equation}\label{eq:distribution_term2_temp}
\begin{aligned}
(\mu\rom{2}) & = \lim_{\dta\to0^+} \dfrac{1}{\dta} \frac12 \int_0^T e^{-2\beta t} \parentheses{ W_2(\mu^\tau_t,\rho^{\mu^{\tau+\dta},\alpha^\tau}_t)^2 - W_2(\mu^\tau_t,\rho\mat_t)^2} \,\rd t \\
& = \lim_{\dta\to0^+} \dfrac{1}{\dta} \int_0^T e^{-2\beta t} \, W_2(\mu^\tau_t,\rho\mat_t) \parentheses{ W_2(\mu^\tau_t,\rho^{\mu^{\tau+\dta},\alpha^\tau}_t) - W_2(\mu^\tau_t,\rho\mat_t)} \,\rd t \\
& \le \lim_{\dta\to0^+} \dfrac{1}{\dta} \int_0^T e^{-2\beta t} \, W_2(\mu^\tau_t,\rho\mat_t) \, W_2(\rho^{\mu^{\tau+\dta},\alpha^\tau}_t,\rho\mat_t) \,\rd t \\
& \le \Big(\int_0^T e^{-2\beta t} \, W_2(\mu^\tau_t,\rho\mat_t)^2 \,\rd t\Big)^{\frac12} \lim_{\dta\to0^+} \dfrac{1}{\dta} \Big(\int_0^T e^{-2\beta t} \, W_2(\rho^{\mu^{\tau+\dta},\alpha^\tau}_t,\rho\mat_t)^2 \,\rd t\Big)^{\frac12} \\
& = d_\beta(\mu^\tau,\rho\mat) \lim_{\dta\to0^+} \dfrac{1}{\dta} d_\beta(\rho^{\mu^{\tau+\dta},\alpha^\tau},\rho\mat) \\
&  \le d_\beta(\mu^\tau,\rho\mat) \lim_{\dta\to0^+} \dfrac{\kappa}{\dta} d_\beta(\mu^{\tau+\dta},\mu^\tau)= d_\beta(\mu^\tau,\rho\mat) \,\kappa \,\dfrac{\rd}{\rd\tau} d_\beta(\mu^\tau, \nu)\Big|_{\nu = \mu^\tau},
\end{aligned}
\end{equation}
where the last inequality follows from Lemma~\ref{lem:controaction_FPI}. By Lemma~\ref{lem:constant_speed},
\begin{equation}\label{eq:pta_dbeta_mutau}
\begin{aligned}
& \quad \dfrac{\rd}{\rd\tau} d_\beta(\mu^\tau, \nu)\Big|_{\nu = \mu^\tau} = \lim_{\dta\to0} \dfrac{1}{\dta} \Big[\Big(\int_0^T e^{-2\beta t} \,  W_2(\mu^{\tau+\dta}_t, \mu^\tau_t)^2 \,\rd t\Big)^{\frac12} - 0\Big] \\
& = \Big[\int_0^T e^{-2\beta t} \, \Big( \lim_{\dta\to0} \dfrac{1}{\dta} W_2(\mu^{\tau+\dta}_t, \mu^\tau_t)\Big)^2 \,\rd t\Big]^{\frac12} = \Big[\int_0^T e^{-2\beta t} \, \Big( \dfrac{\rd}{\rd\tau} W_2(\mu^\tau_t,\nu_t)\Big|_{\nu_t=\mu^\tau_t}\Big)^2 \,\rd t\Big]^{\frac12} \\
& =\Big[\int_0^T e^{-2\beta t} \, \parentheses{ \beta_\mu W_2(\mu^\tau_t,\rho\mat_t)}^2 \,\rd t\Big]^{\frac12} = \beta_\mu d_\beta(\mu^\tau, \rho\mat).
\end{aligned}
\end{equation}
Substituting \eqref{eq:pta_dbeta_mutau} into \eqref{eq:distribution_term2_temp} yields
\begin{equation}\label{eq:distribution_term2}
(\mu\rom{2}) \le \beta_\mu \kappa d_\beta(\mu^\tau, \rho\mat)^2 .
\end{equation}

For part $(\mu\rom{3})$, we carry out similar estimation to $(\mu\rom{2})$:
\begin{equation}\label{eq:Lmu_term3}
\begin{aligned}
(\mu\rom{3})& = \frac12 \lim_{\dta\to0^+} \dfrac{1}{\dta} \sqbra{d_\beta(\mu^\tau, \rho^{\mu^\tau,\alpha^{\tau+\dta}})^2 - d_\beta(\mu^\tau, \rho\mat)^2} \\
& = \frac12 \lim_{\dta\to0^+} \dfrac{1}{\dta} \int_0^T e^{-2\beta t} \parentheses{ W_2(\mu^\tau_t,\rho^{\mu^\tau,\alpha^{\tau+\dta}}_t)^2 - W_2(\mu^\tau_t,\rho\mat_t)^2} \,\rd t \\
& = \lim_{\dta\to0^+} \dfrac{1}{\dta} \int_0^T e^{-2\beta t} \,  W_2(\mu^\tau_t,\rho\mat_t) \parentheses{ W_2(\mu^\tau_t,\rho^{\mu^\tau,\alpha^{\tau+\dta}}_t) - W_2(\mu^\tau_t,\rho\mat_t)} \,\rd t \\
& \le \lim_{\dta\to0^+} \dfrac{1}{\dta} \int_0^T e^{-2\beta t} \,  W_2(\mu^\tau_t,\rho\mat_t) \, W_2(\rho^{\mu^\tau,\alpha^{\tau+\dta}}_t,\rho\mat_t)  \,\rd t \\
& \le \Big(\int_0^T e^{-2\beta t} \, W_2(\mu^\tau_t,\rho\mat_t)^2 \,\rd t\Big)^{\frac12} \lim_{\dta\to0^+} \dfrac{1}{\dta} \Big(\int_0^T e^{-2\beta t} \, W_2(\rho^{\mu^\tau,\alpha^{\tau+\dta}}_t,\rho\mat_t)^2 \,\rd t\Big)^{\frac12} \\
& = d_\beta(\mu^\tau,\rho\mat) \sqbra{ \lim_{\dta\to0^+} \dfrac{1}{\dta^2} d_\beta(\rho^{\mu^\tau,\alpha^{\tau+\dta}},\rho\mat)^2}^{\frac12}.
\end{aligned}
\end{equation}
By \eqref{eq:Gronwall_Wasserstein} in Corollary~\ref{cor:Gronwall_Wasserstein} and \eqref{eq:actor_flow},
\begin{equation}\label{eq:distribution_term3}
\begin{aligned}
(\mu\rom{3}) & \le C d_\beta(\mu^\tau,\rho\mat) \sqbra{ \lim_{\dta\to0^+} \dfrac{1}{\dta^2} \norm{\alpha^{\tau+\dta} - \alpha^\tau}_\mato^2}^{\frac12}\\
& = C \beta_a \, d_\beta(\mu^\tau,\rho\mat) \, \norm{\na H(t,x,\mu_t,\alpha^\tau(t,x), -\mG^\tau(t,x))}_\mato \\
& \le \frac14 \beta_\mu \,  d_\beta(\mu^\tau,\rho\mat)^2 + C \dfrac{\beta_a^2}{\beta_\mu} \norm{\na H(t,x,\mu_t,\alpha^\tau(t,x), -\mG^\tau(t,x))}_\mato^2.
\end{aligned}
\end{equation}

Substituting~\eqref{eq:distribution_term1}, \eqref{eq:distribution_term2}, and \eqref{eq:distribution_term3} into \eqref{eq:Lmu_first_part}, and using $\kappa \le \frac14$, we obtain
\begin{equation}\label{eq:distribution_term_main}
\dfrac{\rd}{\rd\tau}\frac{1}{2} d_\beta(\mu^\tau, \rho\mat)^2 \le - \frac12  \beta_\mu d_\beta(\mu^\tau, \rho\mat)^2 + C \dfrac{\beta_a^2}{\beta_\mu} \norm{\na H(t,x,\mu_t,\alpha^\tau(t,x), -\mG^\tau(t,x))}_\mato.
\end{equation}

\noindent
\emph{Step 2}. We estimate the \(\tau\)-derivative of $\frac12 W_2(\mu^\tau_T, \rho\mat_T)^2$, which is decomposed into two terms:
\begin{equation}\label{eq:Lmu_second_part}
\begin{aligned}
&\quad \dfrac{\rd}{\rd \tau} \frac12 W_2(\mu^\tau_T, \rho\mat_T)^2  = (\mu\rom{4}) + (\mu\rom{5})\\
&:= \dfrac{\rd}{\rd \tau} \frac12 W_2(\mu^\tau_T, \rho_T)^2\Big|_{\rho_T=\rho\mat_T} + \dfrac{\rd}{\rd \tau} \frac12 W_2(\mu_T, \rho\mat_T)^2\Big|_{\mu_T=\mu_T^\tau},
\end{aligned}
\end{equation}
respectively addressing the \(\tau\)-dependence through $\mu^\tau_T$ and $\rho\mat_T$.
By \cite[Theorem 7.2.2]{ambrosio2005gradient}, 
\begin{equation}\label{eq:distribution_term4}
(\mu\rom{4}) = -\beta_\mu \int_{\RR^d}\abs{\nabla \varphi^\tau_T(x)}^2 \,\rd\mu^\tau_T(x) = -\beta_\mu \int_{\RR^d}\abs{x-T^\tau_T(x)}^2 \,\rd\mu^\tau_T(x) = - \beta_\mu W_2(\mu^\tau_T, \rho\mat_T)^2.
\end{equation}

Similar to the analysis for $(\mu\rom{2})$,
\begin{equation}\label{eq:Lmu_term5}
\begin{aligned}
(\mu\rom{5}) & = \lim_{\dta\to0^+} \dfrac{1}{\dta} \frac12 \parentheses{ W_2(\mu^\tau_T,\rho^{\mu^{\tau+\dta},\alpha^{\tau+\dta}}_T)^2 - W_2(\mu^\tau_T,\rho\mat_T)^2}  \\
& = \lim_{\dta\to0^+} \dfrac{1}{\dta}  W_2(\mu^\tau_T,\rho\mat_T) \parentheses{ W_2(\mu^\tau_T,\rho^{\mu^{\tau+\dta},\alpha^{\tau+\dta}}_T) - W_2(\mu^\tau_T,\rho\mat_T)}  \\
& \le W_2(\mu^\tau_T,\rho\mat_T)  \lim_{\dta\to0^+} \dfrac{1}{\dta} \, W_2(\rho^{\mu^{\tau+\dta},\alpha^{\tau+\dta}}_T,\rho\mat_T) \\
& \le \dfrac{\beta_\mu}{2} W_2(\mu^\tau_T,\rho\mat_T)^2 + \dfrac{1}{2\beta_\mu}  \lim_{\dta\to0^+} \dfrac{1}{\dta^2} \, W_2(\rho^{\mu^{\tau+\dta},\alpha^{\tau+\dta}}_T,\rho\mat_T)^2.
\end{aligned}
\end{equation}
By \eqref{eq:Gronwall_Wasserstein} in Corollary~\ref{cor:Gronwall_Wasserstein}, Lemma~\ref{lem:constant_speed} and \eqref{eq:actor_flow},
\begin{equation}\label{eq:Wasserstein_rhoT}
\begin{aligned}
& \quad \lim_{\dta\to0^+} \dfrac{1}{\dta^2} \, W_2(\rho^{\mu^{\tau+\dta},\alpha^{\tau+\dta}}_T,\rho\mat_T)^2 \\
& \le C \lim_{\dta\to0^+} \dfrac{1}{\dta^2} \, \parentheses{\mW_2(\mu^{\tau+\dta}, \mu^\tau)^2 + \norm{\alpha^{\tau+\dta} - \alpha^\tau}_\mato^2}\\
& = C \parentheses{\beta_\mu^2 \mW_2(\mu^\tau,\rho\mat)^2 + \beta_a^2 \norm{\na H(t,x,\mu_t,\alpha^\tau(t,x), -\mG^\tau(t,x))}_\mato^2}\\
& \le C \parentheses{\beta_\mu^2 \, d_\beta(\mu^\tau,\rho\mat)^2 + \beta_a^2 \norm{\na H(t,x,\mu_t,\alpha^\tau(t,x), -\mG^\tau(t,x))}_\mato^2}.
\end{aligned}
\end{equation}
Substituting \eqref{eq:Wasserstein_rhoT} into \eqref{eq:Lmu_term5} provides
\begin{equation}\label{eq:distribution_term5}
\begin{aligned}
(\mu\rom{5}) &\le \frac12 \beta_\mu W_2(\mu^\tau_T,\rho\mat_T)^2 + C_T \beta_\mu \, d_\beta(\mu^\tau,\rho\mat)^2 \\
& \quad + C \dfrac{\beta_a^2}{\beta_\mu} \norm{\na H(t,x,\mu_t,\alpha^\tau(t,x), -\mG^\tau(t,x))}_\mato^2.
\end{aligned}
\end{equation}
Here, we record $C_T$ for the specification of $\lam_T$. 
Substituting \eqref{eq:distribution_term4} and \eqref{eq:distribution_term5} into \eqref{eq:Lmu_second_part} yields
\begin{equation}\label{eq:distribution_term_T}
\begin{aligned}
\dfrac{\rd}{\rd\tau} W_2(\mu^\tau_T, \rho\mat_T)^2 &\le -\frac12\beta_\mu W_2(\mu^\tau_T, \rho\mat_T)^2 + C_T \beta_\mu \, d_\beta(\mu^\tau,\rho\mat)^2 \\
&\quad + C \dfrac{\beta_a^2}{\beta_\mu} \norm{\na H(t,x,\mu_t,\alpha^\tau(t,x), -\mG^\tau(t,x))}_\mato^2.
\end{aligned}
\end{equation}

Since $\lam_T \le \frac{1}{4C_T}$, combining \eqref{eq:distribution_term_main} and \eqref{eq:distribution_term_T} yields
\begin{align*}
\dfrac{\rd}{\rd\tau} \mL^\tau_\mu & = \dfrac{\rd}{\rd\tau} \parentheses{\frac{1}{2} d_\beta(\mu^\tau, \rho\mat)^2 + \frac12 \lam_T W_2(\mu^\tau_T, \rho\mat_T)^2} \\
& \le (- \frac12 + \lam_T C_T)  \beta_\mu d_\beta(\mu^\tau, \rho\mat)^2 -\frac12 \lam_T \beta_\mu W_2(\mu^\tau_T, \rho\mat_T)^2 \\
& \quad + C \dfrac{\beta_a^2}{\beta_\mu} \norm{\na H(t,x,\mu_t,\alpha^\tau(t,x), -\mG^\tau(t,x))}_\mato^2\\
& \le - \frac14 \beta_\mu d_\beta(\mu^\tau, \rho\mat)^2 -\frac14 \lam_T \beta_\mu W_2(\mu^\tau_T, \rho\mat_T)^2 \\
& \quad + C \dfrac{\beta_a^2}{\beta_\mu} \norm{\na H(t,x,\mu_t,\alpha^\tau(t,x), -\mG^\tau(t,x))}_\mato^2 \\
& = - c_\mu \beta_\mu \mL^\tau_\mu + C_\mu \dfrac{\beta_a^2}{\beta_\mu} \norm{\na H(t,x,\mu_t,\alpha^\tau(t,x), -\mG^\tau(t,x))}_\mato^2,
\end{align*}
where $c_\mu = \frac12$. This concludes the proof.
\end{proof}

\section{Baseline derivations of models in Section~\ref{sec:numerical_example}}\label{app:baseline}

In this section, we derive the mean-field equilibria for the systemic risk model and the optimal execution problem, serving as analytical baselines for the numerical comparisons presented in Section~\ref{sec:numerical_example}.

\subsection{Systemic risk model (Section~\ref{sec:SR})}\label{app:baseline_SR}

Denote by \(m_t\) the mean of \(\mu_t\) for any \(t\in[0,T]\) and we use the shorthand notation \(v:=V\ms\) for the optimal value function of the representative agent under the given flow of measure \((\mu_t)\).

For fixed \((m_t)_{t\in[0,T]}\), the value function $v$ satisfies the HJB equation:
\begin{equation}
    \label{eqn:SR_HJB}
    \partial_t v + \inf_\alpha\left\{[a(m_t - x) + \alpha]\partial_x v + \tfrac{1}{2} \alpha^2 - q\alpha(m_t - x) + \tfrac{1}{2} \varepsilon(m_t - x)^2\right\} + \tfrac{1}{2} \sigma^2\partial_{xx} v = 0,
\end{equation}
with terminal condition \(v(T, x) = \frac{c}{2}(x - m_T)^2\).
We adopt a quadratic ansatz \(v(t, x) = \tfrac{1}{2} \eta_t (x - m_t)^2 + \xi_t\), where \(\eta,\xi\) are deterministic measurable functions of time. Minimizing over $\alpha$ yields the optimal control
\begin{equation}
    \hat \alpha(t, x) = (q + \eta_t)(m_t - x).
\end{equation} 
Plugging \(\hat{\alpha}\) into the state dynamics \eqref{eqn:Xt-ex1}, 
integrating and taking expectations on both sides yield \(\dot m_t = 0\), indicating $m_t = m_0 = \EE[X_0]$, for any $t\in[0,T]$.
Therefore, at equilibrium, the population measure \(\hat{\mu}_t\) is Gaussian with mean $\EE[X_0]$ and variance $e^{-2\int_0^t a+q+\eta_s \ud s} \Var[X_0] + \sigma^2 \int_0^t e^{-2\int_s^t a+q+\eta_u \ud u}\ud s $.  

Plugging the quadratic ansatz into the HJB equation and matching coefficients yield an ODE system for $\eta_t$ and $\xi_t$:
\begin{equation}
    \dot\eta_t = \eta_t^2 + 2(a+q)\eta_t - (\varepsilon - q^2),\quad \dot \xi_t = - \tfrac{1}{2} \sigma^2 \eta_t,
\end{equation}
with terminal conditions \(\eta_T = c, \ \xi_T = 0\).
The solutions to the ODEs are given by
\begin{equation}
\eta_t = \frac{-(\varepsilon - q^2)(e^{(\delta^+ - \delta^-)(T-t)} - 1) - c(\delta^+ e^{(\delta^+ - \delta^-)(T-t)} - \delta^-)}{(\delta^-e^{(\delta^+ - \delta^-)(T-t)} - \delta^+) - c(e^{(\delta^+ - \delta^-)(T-t)}-1)},\quad  \xi_t = \tfrac{1}{2} \sigma^2 \int_t^T \eta_s \ud s,
\end{equation}
where \(\delta^\pm := -(a+q) \pm \sqrt{(a+q)^2 + (\varepsilon - q^2)}\).

To evaluate the Lyapunov function of the actor~\eqref{eq:Lyapunov_actor}, we need analytic expressions for the control \(\alpha\ms\), which requires calculations of \(V\ms\) for any fixed flow of measure \((\mu_t)\).
Given \(m_t = \int x \ud\mu_t(x)\), the function \(V\ms\) satisfies the HJB equation~\eqref{eqn:SR_HJB}. Using a quadratic ansatz \(V\ms(t, x) = \tfrac{1}{2} \eta^\mu_t x^2 + \rho^\mu_tx + \xi^\mu_t\), where \(\eta^\mu,\rho^\mu,\xi^\mu\) are deterministic measurable functions of time, we get
\begin{equation}
    \alpha\ms(t,x) = q(m_t - x) - (\eta^\mu_tx + \rho^\mu_t).
\end{equation}
Plugging back into the HJB equation~\eqref{eqn:SR_HJB} and collecting coefficients yield the following ODEs:
\begin{equation}
    \dot{\eta}^\mu_t  =  (\eta^\mu_t)^2 + 2(a+q)\eta^\mu_t - (\varepsilon - q^2),\quad \dot{\rho}^\mu_t = -(a+q)(m_t\eta^\mu_t - \rho^\mu_t)+\eta^\mu_t\rho^\mu_t + (\varepsilon - q^2)m_t,
\end{equation}
with terminal conditions \(\eta^\mu_T = c,\ \rho^\mu_T = -cm_T\).
Consequently, \(\eta \equiv \eta^\mu\), and it suffices to solve for \(\rho^\mu\) for the evaluation of \(\alpha\ms\):
\begin{equation}
    \rho^\mu_t = \Big[-cm_T- \int_t^T m_s((\varepsilon - q^2) - (a+q)\eta_s)e^{(a+q)(T-s) +\int_s^T \eta_u\ud u}\ud s\Big] e^{-(a+q)(T-t) - \int_t^T \eta_s\ud s}.
\end{equation}
As a sanity check, when \(\rho^\mu_t = -\eta_tm_t\) and \(m_t\) is constant, the control reduces to \(\alpha\ms\equiv \alpha^*\), recovering the mean-field equilibrium.

\subsection{Optimal execution (Section~\ref{sec:Trader})}\label{app:baseline_Trader}

Let \(m_t\) be the mean of \(\mu_t\) for any \(t\in[0,T]\) and \(v:=V\ms\) be the optimal value function of the representative agent under the given flow of measure \((\mu_t)\).
Since the optimal execution problem is an extended MFG, \(\mu_t\) denotes a measure on the action space, while the state population distribution is denoted by  \(\nu_t\) with mean  \(p_t := \int x \ud\nu_t(x)\).

For fixed \((m_t)_{t\in[0,T]}\), the HJB equation charaterizing the optimal control reads:
\begin{equation}
    \partial_tv + \inf_\alpha\left\{  \alpha \partial_x v + \tfrac{1}{2} c_\alpha \alpha^2 + \tfrac{1}{2} c_X x^2 - \gamma x m_t\right\} + \tfrac{1}{2} \sigma^2 \partial_{xx}v = 0,
\end{equation}
with  terminal condition \(v(T,x) = \tfrac{1}{2} c_g x^2\).
Using a quadratic ansatz \(v(t, x) = \tfrac{1}{2} \eta_t x^2 + \xi_t x + \zeta_t\), where \(\eta,\xi,\zeta\) are deterministic measurable functions. Optimizing over $\alpha$ yields
\begin{equation}
    \hat \alpha(t, x)  = -\frac{\eta_t x + \xi_t}{c_\alpha}.
\end{equation} 

Plugging the ansatz into the HJB equation and collecting coefficients yield the ODEs for $\eta_t$, $\xi_t$ and $\zeta_t$:
\begin{equation}
    \dot\eta_t = \frac{1}{c_\alpha}\eta_t^2 - c_X,\quad  \dot \xi_t = \frac{1}{c_\alpha} \eta_t \xi_t + \gamma m_t,\quad \dot \zeta_t = \frac{1}{2 c_\alpha} \xi_t^2 - \tfrac{1}{2} \sigma^2 \eta_t,
\end{equation}
with terminal conditions \(\eta_T = c_g,\ \xi_T = 0,\ \zeta_T = 0\).
The Riccati ODE for $\eta_t$ has the explicit solution:
\begin{equation}
    \eta_t = -c_\alpha \sqrt{c_X/c_\alpha}\frac{c_\alpha \sqrt{c_X/c_\alpha} - c_g - (c_\alpha \sqrt{c_X/c_\alpha} + c_g) e^{2\sqrt{c_X/c_\alpha}(T-t)}}{c_\alpha \sqrt{c_X/c_\alpha} - c_g + (c_\alpha \sqrt{c_X/c_\alpha} + c_g) e^{2\sqrt{c_X/c_\alpha}(T-t)}}.
\end{equation}
To solve for $\xi_t$, we propose the ansatz \(\xi_t = p_t (\BAR{\eta}_t - \eta_t)\), where \(\BAR{\eta}_t\) is deterministic and measurable.
Taking expectations on both sides of the state dynamics~\eqref{eqn:Xt-ex3} yields \(\dot{p}_t = -\frac{1}{c_\alpha}p_t\BAR{\eta}_t\).
Combining with \(m_t = -\frac{\eta_tp_t + \xi_t}{c_\alpha}\), the ODE for \(\xi_t\) is essentially a Riccati equation for $\BAR{ \eta}_t$:
\begin{equation}
    \dot {\BAR{\eta}}_t  = -\frac{\gamma}{c_\alpha}\BAR{\eta}_t + \frac{1}{c_\alpha}\BAR{\eta}_t^2 - c_X, 
\end{equation}
with terminal condition \(\BAR{\eta}_T = c_g\).
The explicit solution of \(\BAR{\eta}_t\) is given by
\begin{equation}
    \BAR{\eta}_t = \frac{(c_g - \delta^+)\delta^- - (c_g - \delta^-)\delta^+e^{\frac{\delta^+ - \delta^-}{c_\alpha}(T-t)}}{(c_g - \delta^+) - (c_g - \delta^-)e^{\frac{\delta^+ - \delta^-}{c_\alpha}(T-t)}},
\end{equation}
where \(\delta^\pm := \frac{\gamma\pm \sqrt{\gamma^2 + 4c_\alpha c_X}}{2}\).
With both \(\eta_t\) and \(\xi_t\) explicitly solved, the equilibrium control \(\hat{\alpha}\) is fully determined.

At equilibrium, \(\hat{\nu}_t\) is Gaussian with mean $p_t =e^{-\frac{1}{c_\alpha}\int_0^t \BAR{\eta}_s\ud s} \EE [X_0]$ and variance $e^{-\frac{2}{c_\alpha}\int_0^t \eta_s \ud s }\Var[X_0] + \sigma^2 \int_0^t e^{-\frac{2}{c_\alpha}\int_s^t \eta_u \ud u} \ud s$.
Clearly, \(\hat{\mu}_t = \mL(\hat{\alpha}(t,\hat{X}_t))\), which is Gaussian with mean \(m_t = -\frac{1}{c_\alpha}\mu_t\BAR{\eta}_t\) and variance \(\frac{\eta_t^2}{c_\alpha^2}\left(e^{-\frac{2}{c_\alpha}\int_0^t \eta_s \ud s }\Var[X_0] + \sigma^2 \int_0^t e^{-\frac{2}{c_\alpha}\int_s^t \eta_u \ud u} \ud s\right)\).

\section{Additional numerical experiments for MFAC}\label{app:numerics}

In this appendix, we present additional numerical results for the MFAC algorithm applied to the flocking model, complementing the discussion in Section~\ref{sec:Flocking}. Unless otherwise stated, all model parameters are identical to those in Section~\ref{sec:Flocking}, and all the hyperparameters follow Appendix~\ref{app:hyper}. The only modification concerns the value of the parameter \(\beta\).

\begin{figure}[!ht]
    \centering
    \includegraphics[width=1.0\linewidth]{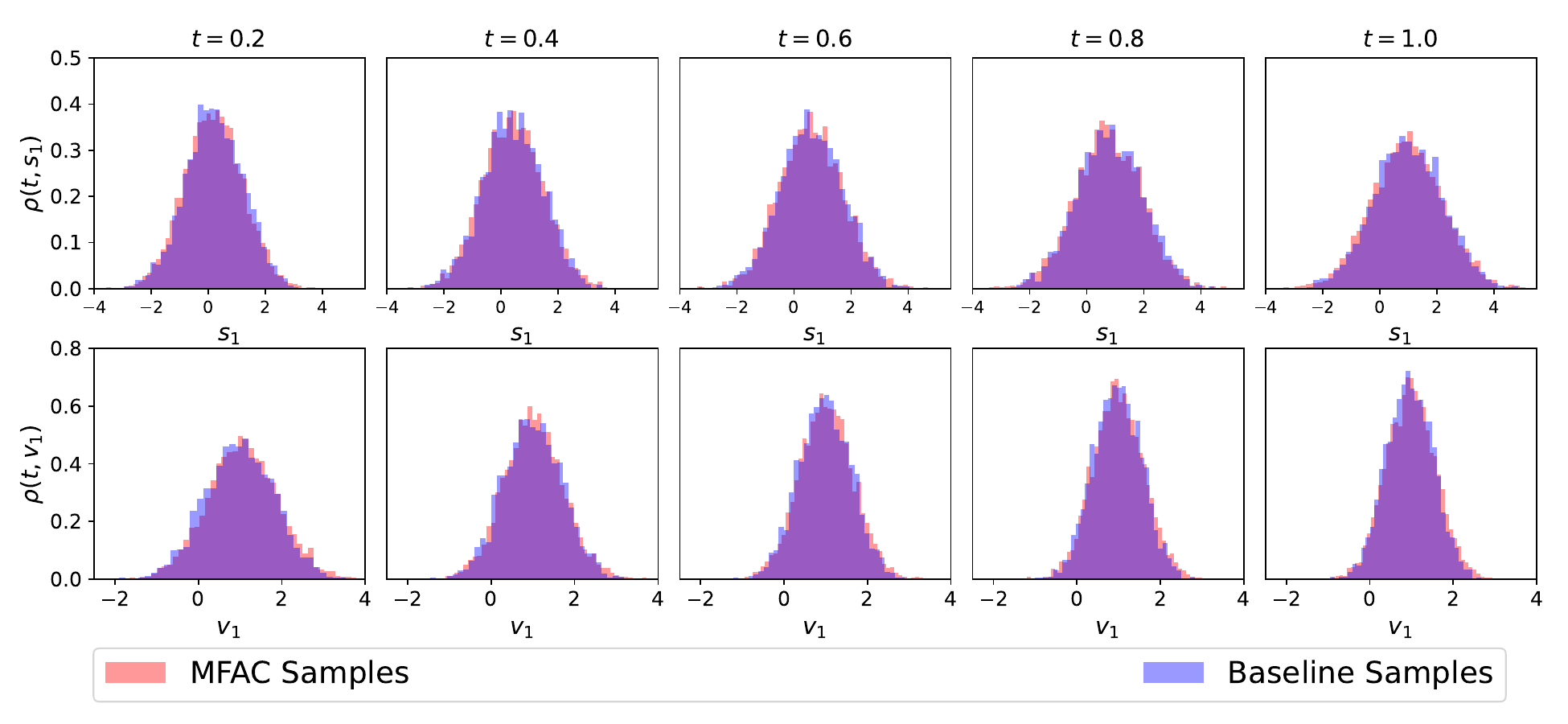}
    \caption{Comparisons of equilibrium population measures in the flocking model (cf. Section~\ref{sec:Flocking}) with \(\beta = 0.1\).
    Blue histograms: baseline results from \cite{han2024learning}, red histograms: MFAC sample paths of \(\check{X}^m_t\).}
    \label{fig:Flocking_Hist_0.1}
\end{figure}

\begin{figure}[!ht]
    \centering
    \includegraphics[width=0.8\linewidth]{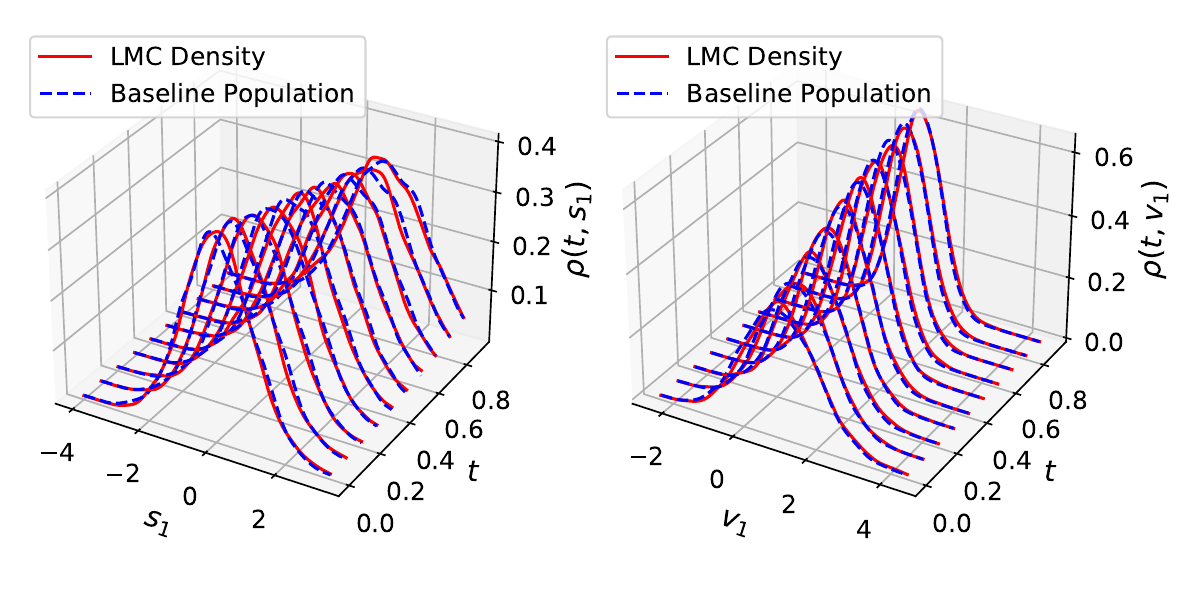}
    \caption{Comparisons of equilibrium population measures in the flocking model (cf. Section~\ref{sec:Flocking}) with \(\beta = 0.1\).
    Blue dashed lines: baseline results from \cite{han2024learning}, red solid lines: kernel density estimations of \(\tilde{\mu}_t\), computed from LMC samples.}
    \label{fig:Flocking_Density_0.1}
\end{figure}

\begin{figure}[!ht]
    \centering
    \includegraphics[width=1.0\linewidth]{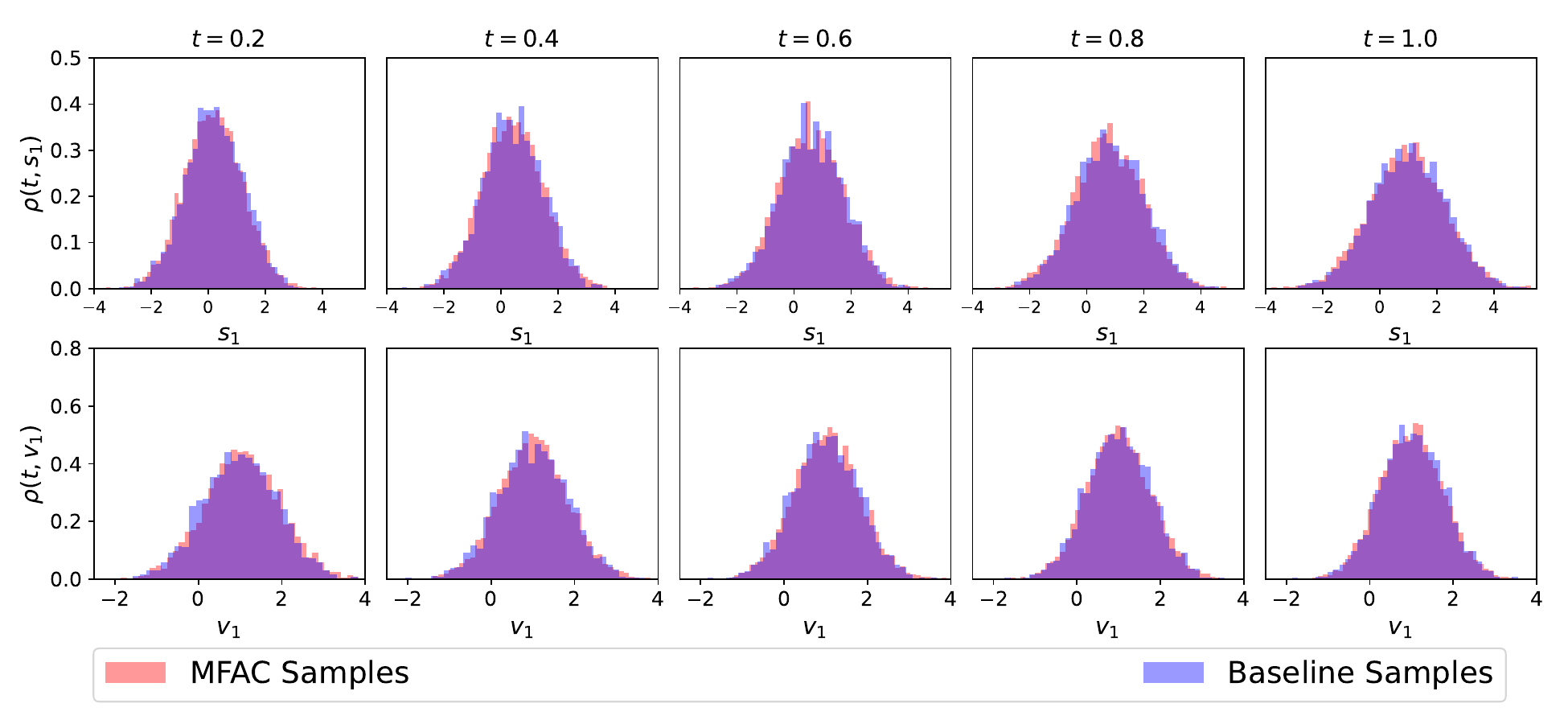}
    \caption{Comparisons of equilibrium population measures in the flocking model (cf. Section~\ref{sec:Flocking}) with \(\beta = 0.3\).
    Blue histograms: baseline results from \cite{han2024learning}, red histograms: MFAC sample paths of \(\check{X}^m_t\).}
    \label{fig:Flocking_Hist_0.3}
\end{figure}

\begin{figure}[!ht]
    \centering
    \includegraphics[width=0.8\linewidth]{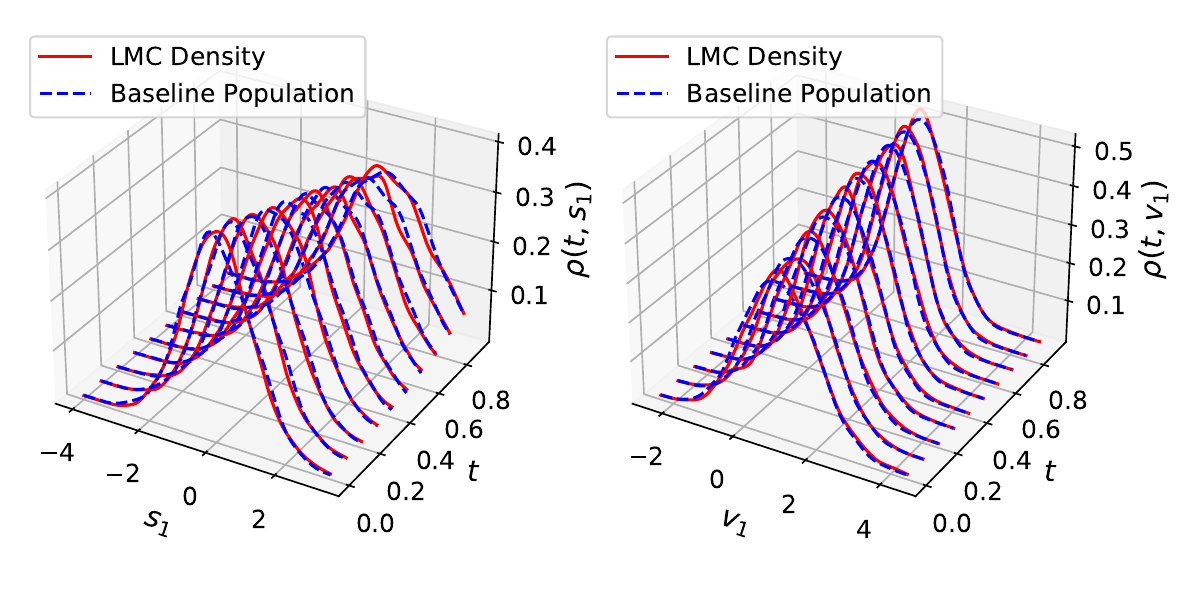}
    \caption{Comparisons of equilibrium population measures in the flocking model (cf. Section~\ref{sec:Flocking}) with \(\beta = 0.3\).
    Blue dashed lines: baseline results from \cite{han2024learning}, red solid lines: kernel density estimations of \(\tilde{\mu}_t\), computed from LMC samples.}
    \label{fig:Flocking_Density_0.3}
\end{figure}

Figures~\ref{fig:Flocking_Hist_0.1}--\ref{fig:Flocking_Density_0.1} compare baseline \textit{vs.} MFAC equilibrium population measures when \(\beta = 0.1\), while Figures~\ref{fig:Flocking_Hist_0.3}--\ref{fig:Flocking_Density_0.3} correspond to the case \(\beta = 0.3\).
The alignment of baseline and MFAC approximations for different values of \(\beta\) shows the general applicability and robustness of MFAC for solving high-dimensional MFGs with general distributional dependencies.

\section{Hyperparamters for numerical experiments}\label{app:hyper}

This section summarizes the hyperparameters used to produce the numerical results in Section~\ref{sec:numerical_example} and Appendix~\ref{app:numerics}.

All neural networks \(\mA,\mV_0,\mG,\mS\) have one hidden layer with \(64\) hidden neurons, one output layer, and \texttt{ReLU} activation functions.  ResNet-type skip connections \cite{he2016identity} are adopted to mitigate the vanishing gradient issue over long time horizons.

The neural network parameters are updated using the Adam optimizer with initial learning rate \(\eta\), and a scheduler that reduces the rate by a factor \(\gamma\in(0,1)\) when the iteration index \(k\) reaches certain milestones.
Subscripts \(a,c,s\) denote hyperparameters that belong to the actor, critic, and score networks, respectively.

Using the notations introduced in Section~\ref{sec:algorithm} and Algorithm~\ref{alg:MFAC}, the training hyperparameters are summarized as follows:
\begin{align*}
    &\eta_a = 0.005, \quad \gamma_a = 0.1,\quad \eta_c = 0.01,\quad  \gamma_c = 0.1, \quad \eta_s = 0.0015, \quad \gamma_s = 0.85,\quad  N_c =N_a = N_s= 5,\\
    &N_T = 50,\quad k_{\mathrm{end}} = 250,\quad\Delta\tau = 0.5,\quad \beta_a = 1.0,\quad \beta_\mu = 1.5, \quad \mathrm{milestones} = \{150, 200\}, \\
    &N_{\batch} = 500,\quad N_T^{\LMC} = 300,\quad h^{\LMC} = 0.05,\quad T^\LMC = 15.
\end{align*}
For the flocking model (Section~\ref{sec:Flocking}), the score-network learning rate is slightly reduced to \(\eta_s = 0.001\), while all other hyperparameters remain unchanged.

For the subroutine of kernel density estimation, which has been used to produce density curves in the figures, we follow state-of-the-art practices, adopting Gaussian kernels and Silverman's rule for bandwidth selection.

\end{document}